\documentclass[oneside, abstract=true]{scrartcl}

\usepackage{preamble}

\title{The stable cohomology of self-equivalences of connected sums of products of spheres}
\author{Robin Stoll}
\date{\today}

\begin{document}

\maketitle

\begin{abstract}
  We identify the cohomology of the stable classifying space of homotopy automorphisms (relative to an embedded disk) of connected sums of $\Sphere k \times \Sphere l$, where $3 \le k < l \le 2k - 2$.
  The result is expressed in terms of Lie graph complex homology.
\end{abstract}

\tableofcontents

\section{Introduction} \label{sec:intro}

The starting point of this article is the following manifold with boundary
\[ M^{k, l}_{g, 1}  \defeq  \connsum_g \big( \Sphere k \times \Sphere l \big)  \setminus \openDisk{k + l} \]
i.e.\ the $g$-fold connected sum of $\Sphere k \times \Sphere l$ with an open disk removed.
Given this, a natural object to consider is the classifying space
\[ \B \autbdry \big( M^{k, l}_{g, 1} \big) \]
of the topological monoid of homotopy automorphisms of $M^{k, l}_{g, 1}$ that fix the boundary pointwise.
It classifies fibrations with fiber $M^{k, l}_{g, 1}$ under the trivial fibration with fiber its boundary (see e.g.\ \cite[Appendix~B]{HL}).
In particular the cohomology of $\B \autbdry(M^{k, l}_{g, 1})$ consists of the characteristic classes for such fibrations.

This cohomology appears to be hard to understand completely.
However, at least rationally, it is possible to say something about the behavior of the \emph{stabilization maps}
\begin{equation} \label{eq:intro_stabilize}
  \Coho * \big( \B \autbdry(M^{k, l}_{g + 1, 1}); \QQ \big)  \longto  \Coho * \big( \B \autbdry(M^{k, l}_{g, 1}); \QQ \big)
\end{equation}
obtained by gluing on, along one of its boundary components, a copy of $\Sphere k \times \Sphere l$ with two disjoint open disks removed.
Namely, in the case $2 \le k = l$, Berglund--Madsen \cite{BM} have shown this map to be an isomorphism in a certain stable range of degrees\footnote{This stable range was extended significantly by Krannich \cite{Kra}.}, and this was extended by Grey \cite{Gre} to the cases $3 \le k \le l \le 2k - 2$.

Moreover \cite{BM} contains, in the case $3 \le k = l$, a combinatorial description (in terms of graph complex homology) of the \emph{stable cohomology}, i.e.\ the limit of the maps \eqref{eq:intro_stabilize} as $g$ goes to infinity.
In this paper we provide a similar description when $k \neq l$.
The following is our main result.

\begin{theorem}[see \cref{thm:main}] \label{thm:intro_main}
  Let $3 \le k < l \le 2k - 2$ and $2 \le g$ be integers.
  Then there is, in cohomological degrees $\le g - 2$, an isomorphism of graded algebras
  \[ \Coho * \big( \B \autbdry (M^{k,l}_{g,1}); \QQ \big)  \iso  \Coho * \big( \GL(\ZZ); \QQ \big) \tensor \Coho * \big( \dual{\UGC {k + l - 2} \Lie} \big) \]
  compatible with the stabilization maps on the left hand side.
  Here we set $\GL(\ZZ) \defeq \colim{g \in \NNo} \GL[g](\ZZ)$, denote by $\UGC m \Lie$ the \emph{$m$-twisted graph complex} associated to the cyclic Lie operad (with its canonical coalgebra structure), and write $\dual{(\blank)}$ for linear dualization.
  
  In particular, after stabilizing, we obtain an isomorphism
  \[ \lim {g \in \NNo} \Coho * \big( \B \autbdry (M^{k,l}_{g,1}); \QQ \big)  \iso  \Coho * \big( \GL(\ZZ); \QQ \big) \tensor \Coho * \big( \dual{\UGC {k + l - 2} \Lie} \big) \]
  of graded algebras.
\end{theorem}

\begin{remark*}
  We also prove a version of \cref{thm:intro_main} for cohomology with certain local coefficients, see \cref{thm:main_coeff}.
  In this generality the description we give is in terms of a \emph{directed} graph complex (see below).
  It should be possible to relate this to an undirected graph complex, just as in the case of trivial coefficients above, but we do not carry this out in this paper.
\end{remark*}

\begin{remark*}
  The proof of \cref{thm:intro_main} does not rely on the stability result of Grey \cite{Gre} (though we do reuse some of its ingredients).
  In particular we obtain a new proof of his result (for the manifolds $M^{k, l}_{g, 1}$), while also roughly doubling the stable range.
  This is made possible by a result of Li--Sun \cite{LS19}, which yields a slope $1$ vanishing of the rational cohomology of $\GL[n](\ZZ)$ with certain coefficients.
  That a result of this form would lead to an improvement of the stable range had been suggested by Krannich \cite[p.~1070]{Kra}.
\end{remark*}

Surprisingly, the graph complex appearing in \cref{thm:intro_main} only depends on the dimension $k + l$ and not on $k$ and $l$ themselves.
Moreover, when $k + l$ is even, it is the same graph complex that appears in \cite{BM}, i.e.\ in the case $k = l$ (when $k + l$ is odd, it is defined similarly to the even case but appears to have quite different homology; see further below).
For both of these observations there seems to be no a priori reason to expect them to be true.
The second one is particularly interesting when considering our proof, where first a different graph complex appears, of which we then show that it actually has the same homology as $\UGC {k + l - 2} \Lie$.

We will now describe the terms appearing on the right hand sides of the isomorphisms of \cref{thm:intro_main} in more detail.
First note that
\[ \Coho{*} \big( \GL(\ZZ); \QQ \big)  \iso  \freegca \gen \QQ {x_i \mid i \in \NNpos}  \qquad\text{where}\quad \deg{x_i} = 4i + 1 \]
by a classical result of Borel \cite{Bor74}.
The graph complex $\UGC m {\operad C}$, for $\operad C$ a cyclic operad, is also a well-known object.
Versions of it were first described by Kontsevich \cite{Kon93,Kon94}, but they have, since then, appeared in many different places, see e.g.\ \cite{GK98,CV,LV,BM}.

We now sketch the definition of this graph complex.
It is generated by graphs (potentially with loops and parallel edges) equipped with an orientation of each edge and a labeling of each vertex with an element of $\operad C$ of cyclic arity the valence of the vertex (which we assume to be at least three).
We consider a vertex to have homological degree $1 - m$ and an edge to have homological degree $m$.
The resulting graded vector space is quotiented by the actions of the automorphism groups of the graphs, taking into account the homological degrees of the vertices and edges, and yielding an extra sign $(-1)^{m+1}$ whenever the orientation of an edge is flipped.
The differential is defined to be the sum over all possible contractions of non-loop edges.
When an edge is contracted, the label of the new vertex is obtained via the cyclic operad composition of the labels of the two contracted vertices, as in the following picture.
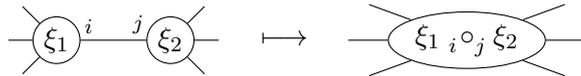
\begin{center}
  \begin{tikzpicture} [scale = 1.5]
  \tikzstyle{vertex} = [circle, draw, inner sep = 0.05cm]
  \tikzstyle{bigvertex} = [ellipse, draw, inner sep = 0.05cm]
  \tikzstyle{label} = [above = -0.05cm, font = \scriptsize]
  
  \node[vertex] (l) at (0,0) {$\xi_1$};
  \node[vertex] (r) at (1,0) {$\xi_2$};
  
  \draw (l.east) -- (r.west) node[label, very near start] {$i$} node[label, very near end] {$j$};
  
  \draw (l.north west) -- ($2*(l.north west) - (l.center)$);
  \draw (l.west) -- ($2*(l.west) - (l.center)$);
  \draw (l.south west) -- ($2*(l.south west) - (l.center)$);
  \draw (r.north east) -- ($2*(r.north east) - (r.center)$);
  \draw (r.east) -- ($2*(r.east) - (r.center)$);
  \draw (r.south east) -- ($2*(r.south east) - (r.center)$);
  
  \node at (2,0) {$\longmapsto$};
  
  \node[bigvertex] (c) at (3.6,0) {$\xi_1 \opcomp[i]{j} \xi_2$};
  
  \draw (c.north west) -- ($1.75*(c.north west) - 0.75*(c.center)$);
  \draw (c.west) -- ($1.5*(c.west) - 0.5*(c.center)$);
  \draw (c.south west) -- ($1.75*(c.south west) - 0.75*(c.center)$);
  \draw (c.north east) -- ($1.75*(c.north east) - 0.75*(c.center)$);
  \draw (c.east) -- ($1.5*(c.east) - 0.5*(c.center)$);
  \draw (c.south east) -- ($1.75*(c.south east) - 0.75*(c.center)$);
  \end{tikzpicture}
\end{center}
The graph complex has the structure of a cocommutative differential graded coalgebra with comultiplication given by the sum over all possible ways to partition the connected components of a graph into two sets.

Note that, up to regrading, the graph complex and in particular its homology only depend on the parity of $m$.
Even though this homology could theoretically be computed directly from its definition, this has only be feasible in low degrees so far.
Its structure in higher degrees remains largely mysterious.
By a result of Kontsevich \cite{Kon93} (for $m$ even; see also Conant--Vogtmann \cite{CV}) and Lazarev--Voronov \cite{LV} (for $m$ odd), Lie graph complex homology is closely related to the rational homology (with twisted coefficients if $m$ is odd) of the group $\Out(F_g)$ of outer automorphisms of the free group on $g$ generators.
See \cref{rem:lie_graph_homology} for a summary of what is known about this.
In particular, computations of Brun--Willwacher \cite{BW} imply that $\Ho p (\UGC m \Lie)$ is trivial for $0 < p \le 3m + 2$ when $m \ge 3$ is odd (see \cref{rem:lie_graph_homology}).
Using this, we obtain the following consequence of our main theorem.

\begin{corollary} \label{cor:intro}
  Let $3 \le k < l \le 2k - 2$ and $2 \le g$ be integers such that $k + l$ is odd.
  Then there is, in cohomological degrees $\le {\min} \big( g - 2, 3(k + l) - 4 \big)$, an isomorphism of graded algebras
  \[ \Coho * \big( \B \autbdry (M^{k,l}_{g,1}); \QQ \big)  \iso  \Coho * \big( \GL(\ZZ); \QQ \big) \]
  compatible with the stabilization maps on the left hand side.
\end{corollary}

\paragraph{Sketch of the proof of \cref{thm:intro_main}.}
From now on we fix some $3 \le k < l \le 2k - 2$ and drop them from the notation.
The starting point of our proof are methods recently developed by Berglund--Zeman \cite{BZ} which in particular yield an isomorphism
\begin{equation} \label{eq:intro_BZ}
  \Coho * \big( \B \autbdry(M_{g, 1}); P \big)  \iso  \Coho * \big( \Gamma_g^\ZZ; \CEcoho * (\lie g_g) \tensor P \big)
\end{equation}
for any $\Gamma^\ZZ_g$-representation $P$ (the details of applying their methods to relative self-equivalences are worked out in the upcoming paper \cite{Sto} joint with Berglund).
Here the differential graded Lie algebra $\lie g_g$ is the Quillen model of the 1-connected cover of $\B \autbdry(M_{g, 1})$, we write $\CEcoho * (\lie g_g)$ for its Chevalley--Eilenberg cohomology, and we set
\[ \Gamma_g^R  \defeq  {\Aut} \big( \rHo{*}(M_{g,1}; R), \intpair{\blank}{\blank} \big)  \iso  \GL[g](R) \]
where $R$ is a commutative ring and $\intpair{\blank}{\blank}$ is the intersection pairing.
The Lie algebra $\lie g_g$ has an explicit description which was shown by Berglund--Madsen \cite{BM} to be isomorphic to
\[ \trunc {\Schur[\big] {(\shift[-(k + l - 2)] \operad \Lie)} {\shift[-1] \rHo{*}(M_{g,1}; \QQ)}} \]
equipped with the trivial differential.
Here $\Lie$ is (the symmetric sequence underlying) the cyclic Lie operad, $\shift$ denotes a degree shift, $\Schur {\bijmod A} {\blank}$ is the Schur functor associated to $\bijmod A$, and $\trunc {(\blank)}$ denotes positive truncation.

The next step is to simplify the right hand side of \eqref{eq:intro_BZ}.
To this end we prove, using the previously mentioned work of Li--Sun \cite{LS19}, that the canonical map
\[ \Coho{p} \big( \Gamma_g^\ZZ; \QQ \big) \tensor \Coho{q} \Big( \big( \CEcochains{*}(\lie g_g) \tensor P \big)^{\Gamma_g^\QQ} \Big)  \longto  \Coho{p} \big( \Gamma_g^\ZZ; \CEcoho{q}(\lie g_g) \tensor P \big) \]
is an isomorphism in a stable range when $P$ is an algebraic representation of $\Gamma^\QQ_g$.
(See also work of Krannich \cite{Kra}, where an analogous argument is made to obtain a similar stable range in the case $k = l$.)

This shows that it is enough to identify the (stable) cohomology of the invariants $\big( \CEcochains{*}(\lie g_g) \tensor P \big)^{\Gamma^\QQ_g}$.
For sake of exposition, we will now focus on the case $P = \QQ$.
In the paper, we do carry out the first part of the argument for a certain class of non-trivial representations as well, though.
Similar situations, arising from $\SP[2g]$ or $\O[g,g]$ instead of $\GL[g]$, have been considered by Kontsevich \cite{Kon93,Kon94} (see also Conant--Vogtmann \cite{CV}) in the symplectic case and by Berglund--Madsen \cite{BM} in both cases.
In our situation we use coinvariant theory methods similar to those of \cite{BM} to obtain a description of the stable cohomology in terms of a directed version of the Lie graph complex, which we denote by $\DGC {m} {\Lie}$.
It is defined in the same way as $\UGC{m}{\Lie}$, except that every edge is equipped with a direction, which (unlike an orientation) must be respected by automorphisms of the graph.

\begin{theorem}[see \cref{lemma:trunc_dgc}] \label{thm:intro_dgc}
  There is an isomorphism
  \[ \colim{g \in \NNo} \coinv {\CEchains * (\lie g_g) } {\Gamma_g^\QQ}  \iso  \DGCtrunc {k + l - 2} {\Lie} \]
  of differential graded coalgebras, compatible with the stabilization maps on the left hand side.
  Here $\DGCtrunc m {\Lie}$ is the subcomplex of $\DGC m {\Lie}$ spanned by \emph{truncated} directed graphs, which are those that do not have vertices of valence zero or one and whose vertices of valence two have only incoming edges.
\end{theorem}

\begin{remark*}
  At the heart of this theorem actually lies an unstable identification of a certain subcomplex of $\coinv {\CEchains * (\lie g_g) } {\Gamma_g}$.
  Moreover we prove a more general statement, which also applies to a certain class of non-trivial representations $P$, as well as to all cyclic operads and not just $\Lie$.
\end{remark*}

While less ubiquitous than undirected graph complexes, directed versions have appeared in the literature before, see for example \cite{Wil,Ziv,DR}.
Surprisingly it turns out that the homology of the truncated directed graph complex $\DGCtrunc m {\Lie}$ is actually isomorphic to the one of $\UGC m {\Lie}$.
We will now explain the idea of this argument.
Restricting the differential of $\DGCtrunc m {\Lie}$ to the two edges incident to a vertex of valence two (which must be labeled by the identity operation) yields the following picture.
\begin{center}
  \begin{tikzpicture} [scale = 1.5]
  \tikzstyle{vertex}=[circle, draw, inner sep = 0.05cm]
  \tikzstyle{bigvertex}=[ellipse, draw, inner sep = 0.05cm, minimum width = 1cm]
  
  \node[vertex] (l1) at (-1,0) {$\xi_1$};
  \node[vertex] (m1) at (0,0) {$\id$};
  \node[vertex] (r1) at (1,0) {$\xi_2$};
  
  \draw (l1.north west) -- ($2*(l1.north west) - (l1.center)$);
  \draw (l1.west) -- ($2*(l1.west) - (l1.center)$);
  \draw (l1.south west) -- ($2*(l1.south west) - (l1.center)$);
  \draw (r1.north east) -- ($2*(r1.north east) - (r1.center)$);
  \draw (r1.east) -- ($2*(r1.east) - (r1.center)$);
  \draw (r1.south east) -- ($2*(r1.south east) - (r1.center)$);
  
  \node at (2,0) {$\longmapsto$};
  
  \node[vertex] (l2) at (3,0) {$\xi_1$};
  \node[bigvertex] (r2) at (4,0) {$\xi_2$};
  
  \draw (l2.north west) -- ($2*(l2.north west) - (l2.center)$);
  \draw (l2.west) -- ($2*(l2.west) - (l2.center)$);
  \draw (l2.south west) -- ($2*(l2.south west) - (l2.center)$);
  \draw (r2.north east) -- ($1.75*(r2.north east) - 0.75*(r2.center)$);
  \draw (r2.east) -- ($1.5*(r2.east) - 0.5*(r2.center)$);
  \draw (r2.south east) -- ($1.75*(r2.south east) - 0.75*(r2.center)$);
  
  \node at (4.75,0) {$+$};
  
  \node[bigvertex] (l3) at (5.5,0) {$\xi_1$};
  \node[vertex] (r3) at (6.5,0) {$\xi_2$};
  
  \draw (l3.north west) -- ($1.75*(l3.north west) - 0.75*(l3.center)$);
  \draw (l3.west) -- ($1.5*(l3.west) - 0.5*(l3.center)$);
  \draw (l3.south west) -- ($1.75*(l3.south west) - 0.75*(l3.center)$);
  \draw (r3.north east) -- ($2*(r3.north east) - (r3.center)$);
  \draw (r3.east) -- ($2*(r3.east) - (r3.center)$);
  \draw (r3.south east) -- ($2*(r3.south east) - (r3.center)$);
  
  \begin{scope}[decoration = {markings, mark = at position 0.6 with {\arrow{>}}}]
    \draw[postaction={decorate}] (l1.east) -- (m1.west);
    \draw[postaction={decorate}] (r1.west) -- (m1.east);
    
    \draw[postaction={decorate}] (l2.east) -- (r2.west);
    
    \draw[postaction={decorate}] (r3.west) -- (l3.east);
  \end{scope}
  \end{tikzpicture}
\end{center}
Thus, intuitively, taking homology should kill the difference between the two orientations of an edge.
This can be formalized using a spectral sequence argument to prove the following theorem.
Together with the preceding results, this implies \cref{thm:intro_main}.

\begin{theorem}[see \cref{lemma:UGC_diff}] \label{thm:intro_ugc}
  Let $\operad C$ be a cyclic operad such that $\cyc {\operad C} 2 \iso \linspan{\id}$.
  Then there is, for any $m \in \NN$, an isomorphism
  \[ \Ho * \big( \DGCtrunc m {\operad C} \big)  \iso  \Ho * \big( \UGC m {\operad C} \big) \]
  of graded coalgebras.
\end{theorem}

\begin{remark*}
  We only carry out this argument in the situation arising from trivial coefficients $P = \QQ$.
  There should exist a more general statement of this form, though.
\end{remark*}

\begin{remark*}
  For the non-truncated directed graph complex an argument similar to this has been sketched by Willwacher \cite[Appendix~K]{Wil}.
\end{remark*}

\paragraph{Related work and further research.}
Building on their work for homotopy automorphisms, Berglund--Madsen \cite{BM} also computed, for $k \ge 3$, the stable cohomology of the classifying spaces $\B \widetilde{\mathrm{Diff}}_\del \big( M^{k,k}_{g,1} \big)$ of block diffeomorphisms.
Combining their methods with the methods of our paper (as well as an upgrade of the methods of \cite{BZ} to automorphisms of bundles, which is planned to be contained in the upcoming paper \cite{Sto} joint with Berglund or a follow-up) should yield a computation of the stable cohomology of $\B \widetilde{\mathrm{Diff}}_\del \big( M^{k,l}_{g,1} \big)$ for $3 \le k < l \le 2k - 2$.
This is currently work in progress.
It should significantly simplify and extend a result of Ebert--Reinhold \cite{ER}, who compute, in a hands-on way, the stable cohomology of $\B \widetilde{\mathrm{Diff}}_\del \big( M^{k,k+1}_{g,1} \big)$ up to degree roughly $k$.
From this they deduce, using classical methods as well as a recent result of Krannich \cite{Kra22}, the stable cohomology of $\B \mathrm{Diff}_\del \big( M^{k,k+1}_{g,1} \big)$ in the same range.
A complete calculation of the stable cohomology of block diffeomorphisms would, in the same way, allow to extend this result to degree roughly $2k$.
Moreover, one might hope to extend this calculation even further, for example by using methods recently developed by Krannich--Randal-Williams \cite{KR}.
This is the goal of current work in progress joint with Krannich.

One question one might hope to attack with this approach is to determine the images of the cohomology classes coming from the graph complex under the morphism
\[ \Coho * \big( \B \autbdry \big(M^{k,l}_{g,1} \big); \QQ \big)  \longto  \Coho * \big( \B \mathrm{Diff}_\del \big( M^{k,l}_{g,1} \big); \QQ \big) \]
induced by the forgetful map.
When $k = l$, the stable cohomology of the right hand side has been computed completely in celebrated work of Galatius--Randal-Williams \cite{GR}, but even in that case it is mostly unknown what the graph cohomology classes map to.
Results in this direction have the potential to yield very interesting information about both the cohomology of the diffeomorphism groups as well as the homology of the graph complex and thus the homology of $\Out(F_g)$.

\paragraph{Structure of this paper.}
In \cref{sec:preliminaries}, we fix some conventions, recall various definitions, and prove basic lemmas needed throughout the rest of the paper.
In \cref{sec:directed_gc}, we use coinvariant theory to prove \cref{thm:intro_dgc}.
In \cref{sec:undirected_gc}, we prove \cref{thm:intro_ugc} using a spectral sequence argument.
Finally, \cref{sec:surfaces} contains the topological part, in which we combine everything to obtain \cref{thm:intro_main}.

\paragraph{Acknowledgments.}
First and foremost, I would like to thank my PhD advisor, Alexander Berglund, for suggesting the topic of this article to me, and for guiding me through the process of writing it.
Moreover, I am grateful to him and Tomáš Zeman for providing me with a preliminary version of their paper \cite{BZ}, as well as to Thomas Willwacher for a preliminary version of the computational results of \cite{BW}.
I would also like to thank Fabian Hebestreit and Manuel Krannich for useful discussions, and the latter for making me aware of the article \cite{LS19}.
Last but not least, I would like to thank Alexander Berglund, Manuel Krannich, and the anonymous referee for many useful comments on earlier versions of this paper.

\section{Preliminaries} \label{sec:preliminaries}

In this section we collect a number of basic conventions, notations, definitions, and lemmas which we will use throughout the rest of this paper.

\subsection{Graded vector spaces}

\begin{convention}
  The base field is $\QQ$.
\end{convention}

\begin{notation}
  We denote by $\Vect$ the category of vector spaces and linear maps.
\end{notation}

\begin{convention}
  All gradings are by $\ZZ$.
\end{convention}

\begin{notation}
  We denote by $\GrVect$ the category of graded vector spaces and grading preserving linear maps.
  Sometimes we will implicitly consider an ungraded vector space (such as $\QQ$) as a graded vector space concentrated in degree $0$.
\end{notation}

\begin{notation}
  For $k \in \ZZ$, we denote by $\shift[k](\blank)$ a $k$-fold degree shift.
  For $V$ a graded vector space, the graded vector space $\shift[k] V$ is given by $(\shift[k] V)_n \defeq V_{n-k}$.
  When $k = 1$, we just write $\shift \defeq \shift[1]$.
\end{notation}

\begin{notation} \label{not:graded_symplectic}
  Let $m \in \ZZ$.
  A \emph{graded symplectic form} of degree $-m$ on a graded vector spaces $V$ is a non-degenerate bilinear form of degree $-m$
  \[ \iprod \blank \blank  \colon  V \tensor V  \longto  \QQ \]
  such that $\iprod v w = (-1)^{\deg v \deg w + 1} \iprod w v$ (i.e.\ it is graded anti-symmetric).
  
  We denote by $\Sp[m]$ the category of finite-dimensional graded vector spaces equipped with a graded symplectic form of degree $-m$ with morphisms those linear maps of degree $0$ that preserve the bilinear form.
\end{notation}

\begin{remark}
  Note that a morphism of $\Sp[m]$ is automatically injective.
\end{remark}

\subsection{Double complexes}

\begin{convention}
  A \emph{double complex} is a bigraded vector space $C_{*,*}$ equipped with a differential $d^1$ of bidegree $(-1, 0)$ and a differential $d^2$ of bidegree $(0, -1)$ such that $d^1 \after d^2 = - d^2 \after d^1$.
\end{convention}

\begin{notation}
  Let $(C_{*,*}, d^1, d^2)$ be a double complex.
  We denote by $\Tot C$ its \emph{total complex}, i.e.\ the chain complex with underlying graded vector space
  \[ (\Tot C)_p  \defeq  \Dirsum_{k +l = p} C_{k,l} \]
  and differential $d = d^1 + d^2$.
\end{notation}

\begin{lemma} \label{lemma:double_complex_ss_collapse}
  Let $(C_{*,*}, d^1, d^2)$ be a double complex concentrated in non-negative degrees such that $\Ho * (C_{p,*}, d^2)$ is concentrated in degree $0$ for all $p$.
  Denote by
  \[ f \colon \Tot C \to \big( \Ho 0 (C_{\bullet,*}, d^2), d^1 \big) \]
  the map of chain complexes given by the canonical projections $C_{p, 0} \to \Ho 0 (C_{p,*}, d^2)$ and the trivial map on $C_{p, q}$ for $q > 0$.
  Then $f$ is a quasi-isomorphism.
\end{lemma}

\begin{proof}
  The map $f$ canonically lifts to a map $g \colon C_{*,*} \to \big( \Ho 0 (C_{\bullet,*}, d^2), d_1 \big)$ of double complexes, where the target is equipped with the double complex structure concentrated in bidegrees $(p, 0)$.
  The map $g$ induces an isomorphism between the $E^1$-pages of the spectral sequences associated to these double complexes whose zeroth differential is given by $d^2$ (see e.g.\ Weibel \cite[§5.6]{Wei}).
  By the Comparison Theorem for spectral sequences, this implies that $f$ is a quasi-isomorphism (see e.g.\ \cite[Theorem~5.2.12]{Wei}).
\end{proof}

\subsection{Group actions}

\begin{definition}
  Let $\cat G$ be a groupoid and $\cat C$ a category.
  A \emph{left $\cat G$-module in $\cat C$} is a functor $M \colon \cat G \to \cat C$, and a \emph{right $\cat G$-module} is a functor $M \colon \opcat{\cat G} \to \cat C$.
\end{definition}

\begin{remark}
  Identifying a group $G$ with its associated one-object groupoid, a left $G$-module in $\cat C$ is an object of $\cat C$ together with a left $G$-action, and a right $G$-module in $\cat C$ is an object of $\cat C$ together with a right $G$-action.
\end{remark}

\begin{convention}
  Unless otherwise stated ``module in $\cat C$'' will mean ``left module in $\cat C$''.
  In particular group actions will be from the left.
  If we omit the category $\cat C$, we mean a (left or right) module in the category $\GrVect$.
\end{convention}

\begin{definition}
  For $\cat G$ a groupoid, $\cat C$ a category, and $M$ a (left) $\cat G$-module in $\cat C$, we will denote by $\opmod M$ the right $\cat G$-module in $\cat C$ given by the composition
  \[ \opcat{\cat G} \xlongto{\iso} \cat G \xlongto{M} \cat C \]
  where the first map is the isomorphism given by the identity on objects and by taking the inverses of morphisms.
\end{definition}

\begin{remark}
  If $\cat G = G$ is a group, this means that $\opmod M$ is the right $G$-module with the opposite action, i.e.\ $\opmod m \act g \defeq \opmod {(\inv g \act m)}$.
\end{remark}

\subsection{Kan extensions} \label{sec:Kan_extensions}

The following lemmas will be useful later.

\begin{lemma} \label{lemma:functors_from_gc}
  Let $G$ be a group and $S$ a $G$-set, i.e.\ a functor $S \colon G \to \Set$.
  Then a functor $X$ from the Grothendieck construction $\gc[G] S$ to some category $\cat C$ can be equivalently described as a family $(X_s)_{s \in S}$ of objects of $\cat C$ equipped with morphisms $g_s \colon X_s \to X_{g \act s}$ for every $g \in G$ and $s \in S$ such that they are functorial in the sense that $h_{g \act s} \after g_s = (h g)_s$.
  
  Moreover, let $H \trianglelefteq G$ be a normal subgroup.
  If $\cat C$ is cocomplete, then there is an isomorphism
  \[ \coinv {\big( \coprod_{s \in S} X_s \big)} H  \iso  \Lan{\gc[G] S \to \quot{G}{H}} X \]
  of $(\quot{G}{H})$-modules in $\cat C$.
  On the left hand side $g \in G$ acts on $\coprod_{s \in S} X_s$ by mapping $X_s$ to $X_{g \act s}$ via $g_s$; the action of $\quot{G}{H}$ is the induced action on the coinvariants.
  On the right hand side we have the left Kan extension of $X$ along the composite $\gc[G] S \to G \to \quot{G}{H}$ of the canonical projection and the quotient map.
\end{lemma}

\begin{proof}
  The first claim follows directly from the definition of the Grothendieck construction.
  For the second claim note that the $G$-module $\coprod_{s \in S} X_s$ is isomorphic to the left Kan extension of $X$ along the canonical projection $\gc[G] S \to G$.
  Taking $H$-coinvariants of this corresponds to left Kan extending further along the quotient map $G \to \quot{G}{H}$.
  Since Kan extensions compose, the result is isomorphic to the left Kan extension of $X$ along the composite $\gc[G] S \to \quot{G}{H}$.
\end{proof}

We will now provide an explicit identification of the left Kan extension in the lemma above.
This takes the form of the following (more general) lemmas.

\begin{lemma} \label{lemma:groupoid_comma_cat}
  Let $\cat G$ be a groupoid, $H$ a group, and $F \colon \cat G \to H$ a functor.
  Denote by $\cat G^F \subseteq \cat G$ the wide subcategory of those morphisms that $F$ maps to the neutral element $e_H \in H$.
  Also denote by $*$ the unique object of $H$ considered as a groupoid, and by $F \comma *$ the corresponding comma category.
  Then there is a fully faithful functor
  \[ \Phi_F \colon \cat G^F  \longto  F \comma * \]
  given by sending an object $G \in \cat G^F$ to $(G, e_H \colon F(G) = * \to *)$ and a morphism $g \colon G \to G'$ to itself.
  
  Moreover assume that, for all $G \in \cat G$ and $h \in H$, there exists a morphism $g_{G,h} \colon G \to K_{G,h}$ in $\cat G$ such that $F(g_{G,h}) = h$.
  Then $\Phi_F$ is an equivalence of categories with an inverse (up to natural isomorphism) given by the functor
  \[ \Psi_F  \colon  F \comma *  \longto  \cat G^F \]
  given on objects by $(G, h) \mapsto K_{G,h}$ and on morphisms by $(g \colon (G, h) \to (G', h')) \mapsto g_{G',h'} g \inv{g_{G,h}}$.
  (Note that $\Psi_F$ depends on the choice of $g_{G,h}$ for all $G \in \cat G$ and $h \in H$.)
  The natural isomorphism $\id[F \comma *] \to \Phi_F \after \Psi_F$ is, at an object $(G, h)$, given by the map $g_{G,h} \colon G \to K_{G,h}$.
\end{lemma}

\begin{proof}
  This follows easily from the definitions.
\end{proof}

\begin{lemma} \label{lemma:groupoid_Kan_ext}
  In the situation of the second part of \cref{lemma:groupoid_comma_cat}, let $\cat C$ be a cocomplete category and $X \colon \cat G \to \cat C$ a functor.
  Furthermore, for an element $h \in H$, let $A_h$ denote the functor $\cat G^F \to \cat G^F$ that is given by $G \mapsto K_{G,h}$ on objects and that sends a morphism $g \colon G \to G'$ to $g_{G',h} g \inv{g_{G,h}}$.
  Then there is an isomorphism
  \[ \Lan{F \colon \cat G \to H} X  \iso  \colim[s]{\cat G^F} (X \after \inc[\cat G^F]) \]
  of $H$-modules in $\cat C$.
  Here an element $h \in H$ acts on the right hand side via the functor $A_h$ and the natural transformation $\inc[\cat G^F] \to {\inc[\cat G^F]} \after A_h$ given, at an object $G \in \cat G^F$, by $g_{G,h} \colon G \to K_{G,h}$.
\end{lemma}

\begin{proof}
  The left Kan extension in the statement is, at the unique object $*$ of $H$, isomorphic to $\colim{F \comma *} (X \after \pr)$, where $\pr \colon F \comma * \to \cat G$ denotes the projection.
  The action of an element $h \in H$ on this object is induced by the functor $F \comma * \to F \comma *$ given by $(G, h') \mapsto (G, h h')$ (note that this functor commutes with the projection).
  That this is isomorphic, as an $H$-module, to the description given in the statement follows from \cref{lemma:groupoid_comma_cat}.
\end{proof}

\subsection{Symmetric sequences}

\begin{definition}
  We denote by $\oSet$ the category with objects finite sets equipped with a linear order and morphisms the maps of sets (not necessarily respecting the order), by $\oBij \defeq \Core \oSet$ its maximal subgroupoid, and by $\PermGrpd$ the skeleton of $\oBij$ spanned by the objects $\finset n \defeq \set {1 < \dots < n}$ for $n \in \NN$.
\end{definition}

\begin{remark}
  There exist canonical equivalences of categories $\oSet \to \FinSet$, $\oBij \to \Core \FinSet$, and $\PermGrpd \to \Core \FinSet$, though neither of them admit a canonical inverse equivalence.
  The inclusion $\PermGrpd \to \oBij$, however, is an equivalence that does admit a canonical inverse equivalence, namely by identifying an object $S \in \oBij$ with an object in $\PermGrpd$ via the unique map that respects the linear order.
\end{remark}

\begin{definition}
  Let $\cat V$ be a symmetric monoidal category and $V \in \cat V$.
  Then there is a functor $V^{\tensor \blank} \colon \PermGrpd \to \cat V$ given on objects by $\finset n \mapsto V^{\tensor n}$ and defined on morphisms via the symmetrizer isomorphism.
  Pulling this back via the canonical equivalence $\oBij \to \PermGrpd$ we obtain a functor $\oBij \to \cat V$ which we also denote by $V^{\tensor \blank}$.
  
  Similarly, for $S \in \oBij$ and $(V_s)_{s \in S}$ a family of objects of $\cat V$, we write $\Tensor_{s \in S} V_s$ for $V_{s(1)} \tensor \dots \tensor V_{s(\card S)}$, where $s \colon \finset{\card S} \to S$ is the unique order-preserving bijection.
  We will also use the notation $\Tensor_{s \in S} v_s$ to denote elementary tensors of this tensor product.
\end{definition}

\begin{definition}
  For $\bijmod M$ a $\oBij$-module and $a \colon S \to T$ a morphism of $\oSet$, we write
  \[ \bijmod M(a)  \defeq  \Tensor_{t \in T} \bijmod M \big( \inv a(t) \big) \]
  where $\inv a(t) \subseteq S$ has the linear order induced from the one of $S$.
  This construction is functorial in $a \in \oBij \comma \oBij$, with the comma category being taken over $\oSet$.
\end{definition}

\begin{definition}
  Let $\bijmod P$ and $\bijmod Q$ be two $\oBij$-modules.
  Then we can form their \emph{composition product}, i.e.\ the $\oBij$-module $\bijmod P \compprod \bijmod Q$ given on objects by
  \[ S  \longmapsto  \colim{(a \colon S \to T) \in S \comma \oBij} \big( \bijmod P(T) \otimes \bijmod Q(a) \big) \]
  where the slice category $S \comma \oBij$ is taken over $\oSet$.
  A morphism $f \colon S \to S'$ of $\oBij$ acts via the map
  \[ \colim{(a \colon S \to T) \in S \comma \oBij} \big( \bijmod P(T) \otimes \bijmod Q(a) \big)  \longto  \colim{(a \colon S' \to T) \in S' \comma \oBij} \big( \bijmod P(T) \otimes \bijmod Q(a) \big) \]
  induced by the functor $\inv f \comma \oBij \colon S \comma \oBij \to S' \comma \oBij$ as well as the natural transformation $\bijmod P(T) \otimes \bijmod Q(a) \to \bijmod P(T) \otimes \bijmod Q(a \after \inv f)$ given by the identity of $\bijmod P(T)$ and the map $\bijmod Q(a) \to \bijmod Q(a \after \inv f)$ induced by the morphism $(f, \id[T]) \colon a \to a \after \inv f$ of $\oBij \comma \oBij$.
\end{definition}

\subsection{Schur functors}

\begin{notation}
  Let $\bijmod M$ a $\oBij$-module.
  We define a functor $\Schur{\bijmod M}{\blank} \colon \GrVect \to \GrVect$ by
  \[ V  \longmapsto  \Dirsum_{k \in \NN} \big( \opmod{\bijmod M(\finset k)} \tensor[\Symm{\finset k}] V^{\tensor \finset k} \big) \]
  and call it the \emph{Schur functor} associated to $\bijmod M$.
\end{notation}

\begin{remark} \label{rem:Schur_functor}
  There is a canonical, and natural in $V$, isomorphism
  \[ \Schur{\bijmod M}{V}  \iso  \colim{\finset k \in \PermGrpd} \big( \bijmod M(\finset k) \tensor V^{\tensor \finset k} \big) \]
  where $\PermGrpd$ acts diagonally on the tensor product.
  We can also replace $\PermGrpd$ by $\oBij$ since they are canonically equivalent.
\end{remark}

\begin{lemma} \label{lemma:compprod_and_Schur}
  Let $\bijmod P$ and $\bijmod Q$ be $\oBij$-modules.
  Then there is a natural isomorphism
  \[ \Schur {(\bijmod P \compprod \bijmod Q)} \blank  \xlongto{\iso}  \Schur {\bijmod P} {\Schur {\bijmod Q} \blank} \]
  with an explicit description as given in the proof.
\end{lemma}

\begin{proof}
  We define the isomorphism by sending the element of
  \[ \Schur {(\bijmod P \compprod \bijmod Q)} V = \colim{S \in \oBij} \Big( \colim{(a \colon S \to T) \in S \comma \oBij} \bijmod P(T) \otimes \bijmod Q(a) \Big) \tensor V^{\tensor S} \]
  represented by $S \in \oBij$ and $(a \colon S \to T) \in S \comma \oBij$ as well as the elements
  \[ p \tensor \eTensor_{t \in T} q_t \in \bijmod P(T) \tensor \bijmod Q(a)  \qquad\text{and}\qquad  \eTensor_{s \in S} v_s \in V^{\tensor S} \]
  to $\epsilon \in \set{\pm 1}$ times the element of
  \[ \Schur {\bijmod P} {\Schur {\bijmod Q} \blank}  =  \colim{T \in \oBij}\, \bijmod P(T) \tensor \Big( \colim{F \in \oBij} \bijmod Q(F) \tensor V^{\tensor F} \Big)^{\tensor T} \]
  represented by $T \in \oBij$ and $(F_t \defeq \inv a(t) \in \oBij)_{t \in T}$ as well as the elements
  \[ p \in \bijmod P(T)  \qquad\text{and}\qquad  q_t \tensor \eTensor_{s \in F_t} v_s \in \bijmod Q(F_t) \tensor V^{\tensor F_t} \]
  where $\epsilon$ is the sign incurred by permuting the two expressions
  \[ p \tensor \eTensor_{t \in T} q_t \tensor \eTensor_{s \in S} v_s  \quad\text{and}\quad  p \tensor \eTensor_{t \in T} \big( q_t \tensor \eTensor_{s \in F_t} v_s \big) \]
  into each other.
\end{proof}

\subsection{Differential graded Lie algebras}

\begin{definition} \label{def:pos_trunc}
  Let $L$ be a differential graded Lie algebra.
  Its \emph{positive truncation} is the differential graded Lie subalgebra $\trunc L \subseteq L$ given by
  \[ (\trunc L)_i  \defeq  \begin{cases} L_i, & \text{for } i \ge 2 \\ \ker (\delta \colon L_1 \to L_0), & \text{for } i = 1 \\ 0, & \text{for } i \le 0 \end{cases} \]
  where $\delta$ is the differential of $L$.
\end{definition}

\begin{definition} \label{not:CEchains}
  Let $L$ be a differential graded Lie algebra.
  We denote by $\CEchains * (L)$ its \emph{Chevalley--Eilenberg complex}, i.e.\ the underlying graded vector space of the free graded commutative algebra $\freegca(\shift L)$, equipped with the differential $d = d_0 + d_1$, where
  \begin{align*}
  d_0(\shift x_1 \wedge \dots \wedge \shift x_k) &\defeq \sum_{i = 1}^{k} (-1)^{b_i} \shift x_1 \wedge \dots \wedge \shift \delta(x_i) \wedge \dots \wedge \shift x_k \\
  b_i &\defeq 1 + \sum_{j = 1}^{i - 1} \deg{\shift x_j} \\
  d_1(\shift x_1 \wedge \dots \wedge \shift x_k) &\defeq \sum_{i < j} (-1)^{b_{i,j}} \shift \liebr{x_i}{x_j} \wedge \shift x_1 \wedge \dots \wedge \drop{\shift x_i} \wedge \dots \wedge \drop{\shift x_j} \wedge \dots \wedge \shift x_k \\
  b_{i,j} &\defeq \deg{\shift x_i} (1 + \deg{\shift x_j}) + \deg{\shift x_i} \sum_{l = 1}^{i-1} \deg{\shift x_l} + \deg{\shift x_j} \sum_{l = 1}^{j-1} \deg{\shift x_l}
  \end{align*}
  where $\delta$ is the differential of $L$.
  
  The Chevalley--Eilenberg complex comes equipped with the structure of a cocommutative differential graded coalgebra.
  Its counit $\epsilon \colon \CEchains * (L) \to \QQ$ is given by $1$ on the empty wedge and $0$ on all higher wedges.
  The comultiplication $\Delta \colon \CEchains * (L) \to \CEchains * (L) \tensor \CEchains * (L)$ is given by
  \[ \Delta(\shift x_1 \wedge \dots \wedge \shift x_k)  \defeq  \sum_{A \subseteq \finset k} \epsilon(A) (\Wedge_{a \in A} \shift x_a) \tensor (\Wedge_{b \in \finset k \setminus A} \shift x_b) \]
  where the sum runs over all subsets $A$ of the set $\finset k = \set{1, \dots, k}$.
  Here $\epsilon(A) \in \set{\pm 1}$ is the sign incurred by permuting
  $\shift x_1 \tensor \dots \tensor \shift x_k$ into $\eTensor_{a \in A} \shift x_a \tensor \eTensor_{b \in \finset k \setminus A} \shift x_b$.
  
  Moreover, we write $\CEcoho * (L)$ for the \emph{Chevalley--Eilenberg cohomology} of $L$, i.e.\ the cohomology of the linear dual of $\CEchains * (L)$, equipped with its graded commutative algebra structure.
\end{definition}

\begin{lemma} \label{lemma:CEchains_as_Schur}
  Let $\freegca \shift$ denote the $\oBij$-module given by $\freegca \shift(S) = (\shift \basefield)^{\tensor S}$.
  Then there is a canonical isomorphism of graded vector spaces, natural in $V \in \GrVect$,
  \begin{align*}
  \Schur {(\freegca \shift)} V  &\xlongto{\iso} \freegca (\shift V) \\
  \eqcl {(\shift 1)^{\tensor k} \tensor (v_1 \tensor \dots \tensor v_k)}  &\longmapsto (-1)^c \shift v_1 \wedge \dots \wedge \shift v_k
  \end{align*}
  where $c \defeq \sum_{i = 1}^{k} \deg{v_i} (k - i)$.
\end{lemma}

\begin{proof}
  The map lifts to an isomorphism $(\shift \basefield)^{\tensor k} \tensor V^{\tensor k}  \to  (\shift V)^{\tensor k}$ that is compatible with the $\Symm k$-actions.
  Passing to quotients again, we obtain the desired statement.
\end{proof}

\subsection{The convolution Lie algebra}

Recall that a \emph{cyclic operad} is an operad $\operad C$ such that the right action of $\Symm n$ on $\operad C(n)$ extends to a right action of $\Symm {n+1}$ in a way compatible with the composition operations (see Getzler--Kapranov \cite[§2]{GK95}).
Unless explicitly stated otherwise, we will always work in the category of graded vector spaces.

\begin{notation}
  For a cyclic operad $\operad C$, we write $\cyc{\operad C}{n}$ for $\operad C(n-1)$ equipped with its right $\Symm n$-action.
  We consider a cyclic operad as a $\PermGrpd$-module via (the opposite of) this extended right action.
  (Note that this is different from the $\PermGrpd$-module associated to the underlying operad.)
  In particular the Schur functor associated to $\operad C$ is $\Schur{\operad C}{V} = \Dirsum_{n \ge 1} \cyc{\operad C}{n} \tensor[\Symm n] V^{\tensor n}$.
\end{notation}

\begin{definition} \label{not:cyclic_operad_comp}
  For a cyclic operad $\operad C$, natural numbers $k, l \in \NNge{1}$, operations $p \in \cyc{\operad C} {k}$ and $q \in \cyc {\operad C} {l}$, and $1 \le i \le k$ and $1 \le j \le l$, we set
  \[ p \opcomp[i]{j} q  \defeq  p \opcomp{i} (q \act \rot[l]^j)  \in  \cyc {\operad C} {k+l-2} \]
  where $\rot[n] \defeq (1 2 \dots n) \in \Symm n$ is the cyclic permutation and $\opcomp{i}$ denotes the $i$-th partial composition of the underlying operad of $\operad C$.
  In the case that $i = k$, we set $p \opcomp{k} q \defeq ((p \act \rot[k]) \opcomp{k-1} q) \act \inv{\rot[k+l-2]}$.
\end{definition}

\begin{remark} \label{rem:cyclic_operad_comp}
  Via the canonical equivalence $\oBij \to \PermGrpd$, the operations $\opcomp[i]{j}$ of \cref{not:cyclic_operad_comp} generalize to operations
  \[ \opcomp[s]{t}  \colon  \cyc {\operad C} S \tensor \cyc {\operad C} T  \longto  \cyc {\operad C} {S \opcomp[s]{t} T} \]
  where $S, T \in \oBij$, and $s \in S$ and $t \in T$.
  Here
  \[ S \opcomp[s]{t} T  \defeq  (S \setminus \set s) \amalg (T \setminus \set t) \]
  equipped with the linear order such that
  \[ S_{< s} < T_{> t} < T_{< t} < S_{> s} \]
  with the linear order on each block inherited from $S$ respectively $T$.
\end{remark}

\begin{definition}
  Let $V \in \Sp[m]$, $k, l \in \NN$, $1 \le i \le k$, and $1 \le j \le l$.
  We denote by
  \[ \opcomp[i]{j} \colon V^{\tensor k} \tensor V^{\tensor l} \to V^{\tensor k + l - 2} \]
  the map given by
  \begin{align*}
  (v_1 \tensor \dots \tensor v_k) \opcomp[i]{j} (w_1 \tensor \dots \tensor w_l)  &\defeq  {(-1)}^c \mult v_{1,i-1} \tensor w_{j+1,l} \tensor w_{1,j-1} \tensor v_{i+1,k} \iprod {w_j} {v_i} \\
  c  &\defeq  (\deg {v_{i,k}} - m) (\deg {w_{1,l}} - m) + \deg {w_{1,j}} \deg {w_{j+1,l}} + 1
  \end{align*}
  where, for $a \le b$, we set $v_{a,b} \defeq v_a \tensor v_{a+1} \tensor \dots \tensor v_b$ and analogously for $w_{a,b}$.
\end{definition}

\begin{lemma} \label{lemma:Schur_cyclic_Lie}
  Let $\operad C$ be a cyclic operad in graded vector spaces that is concentrated in even degrees, and let $V \in \Sp[m]$.
  Then $\Schur {(\shift[-m] \operad C)} {V}$ becomes a graded Lie algebra by setting
  \[ \liebr {\shift[-m] \xi \tensor v} {\shift[-m] \zeta \tensor w}  \defeq  \sum_{i \in \finset k, j \in \finset l} \shift[-m] (\xi \opcomp[i]{j} \zeta) \tensor (v \opcomp[i]{j} w) \]
  for $\xi \in \cyc {\operad C} {k}$, $v \in V^{\tensor k}$, $\zeta \in \cyc {\operad C} {l}$, and $w \in V^{\tensor l}$.
  This is functorial, i.e.\ the restricted functor $\restrict {\Schur {(\shift[-m] \operad C)} {\blank}} {\Sp[m]}  \colon  \Sp[m] \to \GrVect$ lifts to a functor to graded Lie algebras, which we also denote by $\Schur {(\shift[-m] \operad C)} {\blank}$.
\end{lemma}

\begin{proof}
  Checking antisymmetry and the Jacobi identity are straightforward, though tedious, computations.
  The only trick one has to use is that $v \opcomp[i]{j} w$ being non-zero implies that $\deg {v_i} + \deg {w_j} = m$.
  We omit the details.
  The claim about the functoriality is clear from the definitions.
  
  In the case $\operad C = \Lie$, a different proof of this fact can be found in \cite[equation~(6.7)]{BM}.
\end{proof}

\begin{remark}
  The condition that $\operad C$ is concentrated in even degrees is purely for convenience, as it simplifies the signs occurring.
\end{remark}

%
%
%
%

\section{The directed graph complex} \label{sec:directed_gc}

In this section our goal is to prove a version of \cref{thm:intro_dgc} with coefficients.
That is, we want to identify (in a stable range)
\[ \coinv {\Big( \CEchains * \big( \trunc {\Schur[\big] {(\shift[-(k + l - 2)] \operad \Lie)} {\shift[-1] \rHo{*}(M^{k, l}_{g,1}; \QQ)}} \big) \tensor Q \Big)} {\Gamma_g^\QQ} \]
where $\Gamma_g^\QQ$ is the group of automorphisms of the graded vector space $\rHo{*}(M_{g,1}; \QQ)$ that respect the intersection pairing, $Q$ is some $\Gamma^\QQ_g$-module, and the graded Lie algebra structure on the Schur functor is obtained from \cref{lemma:Schur_cyclic_Lie}.

To this end, we begin, in the first two subsections, with studying $\coinv {(\Schur {\bijmod P} {H} \tensor Q)} {\Aut(H)}$ where $\bijmod P$ is a $\oBij$-module, $H$ is an object of $\Sp[m]$ that is concentrated in two degrees $\neq \frac{m}{2}$, and $Q$ is an $\Aut(H)$-representation.
From the results we obtain for this situation we will, via \cref{lemma:CEchains_as_Schur,lemma:compprod_and_Schur}, be able to deduce a description as desired.
The arguments we will use are similar to those of \cite[§9]{BM} for the case where $H$ is concentrated in degree $\frac m 2$.
We will restrict ourselves to the case that $Q = V^{\tensor I} \tensor W^{\tensor J}$ for some finite linearly ordered sets $I$ and $J$, where $V$ and $W$ are the two non-trivial homogeneous pieces of $H$.
Identifying the result in this case with its induced $(\Symm I \times \Symm J)$-action allows one to deduce the result for more general representations via Schur--Weyl duality (see e.g.\ \cite[§6.1]{FH} or \cite[§5.19]{Eti}).

Throughout this section, we fix two integers $n$ and $m$ such that $m \neq 2n$.
(They correspond to $k - 1$ and $k + l - 2$ above.)

\subsection{Coinvariants of monomial functors}

In this subsection, we let $H = V \dirsum W$ be a finite-dimensional graded vector space such that $V$ is concentrated in degree $n$ and $W$ is concentrated in degree $m - n \neq n$.
Moreover, we assume that $H$ is equipped with a graded symplectic form $\iprod \blank \blank$ of degree $-m$ (see \cref{not:graded_symplectic}).
We write $\Gamma \defeq \Aut[{\Sp[m]}](H, \iprod \blank \blank)$ for the group of degree $0$ symplectic automorphisms of $H$.

Let $S$, $I$, and $J$ be finite linearly ordered sets and $P$ a $\Symm S$-module.
Our goal for this subsection is to identify the coinvariants
\[ \coinv { \big( (\opmod P \tensor[\Symm S] H^{\tensor S}) \tensor V^{\tensor I} \tensor W^{\tensor J} \big) } \Gamma \]
as a $(\Symm I \times \Symm J)$-module only in terms of $P$, without reference to $H$.
We begin with some basic observations about the structures of $H$ and $\Gamma$.

\begin{lemma} \label{lemma:identify_with_dual}
  The map $W \to \shift[m] \dual V$ given by $w \mapsto \shift[m] \iprod w \blank$ is a degree $0$ isomorphism.
  After identifying $W$ with $\shift[m] \dual V$ via this isomorphism the inner product on $H = V \dirsum \shift[m] \dual V$ is given by $\iprod {(v, \shift[m] \phi)} {(v', \shift[m] \phi')} = \phi(v') - (-1)^{n(m-n)} \phi'(v)$.
\end{lemma}

\begin{proof}
  The composition
  \[
  \begin{tikzcd}[row sep = 0]
    V \dirsum W \rar{\iso} & \dual{(V \dirsum W)} \rar{\iso} & \dual V \dirsum \dual W \\
    h \rar[mapsto] & \iprod{h}{\blank} & \\
    & \phi \rar[mapsto] & (\restrict{\phi}{V}, \restrict{\phi}{W})
  \end{tikzcd}
  \]
  is a degree $-m$ isomorphism of the form
  \[
  \begin{pmatrix}
    0 & f \\
    f' & 0
  \end{pmatrix}
  \]
  since $\iprod v {v'} = 0 = \iprod w {w'}$ for all $v, v' \in V$ and $w, w' \in W$ by our assumption that $m \neq 2n$ (which also implies $m \neq 2(m - n)$).
  In particular $f$ and $f'$ are degree $-m$ isomorphisms.
  Noting that $f(w) = \iprod w \blank$, this implies the first claim.
  The second follows from $f(w)(v) = \iprod w v = - (-1)^{n(m-n)} \iprod v w$.
\end{proof}

Using the preceding lemma we will, in the following, implicitly identify $H$ with $V \dirsum \shift[m] \dual V$.
We now give a description of $\Gamma$ in these terms.

\begin{lemma} \label{lemma:aut_is_gl}
  The map $\GL(V) \to \Aut(V \dirsum \shift[m] \dual V, \iprod \blank \blank) = \Gamma$ given by $f \mapsto f \dirsum \dual{(\inv f)}$ is an isomorphism of groups.
  (Here we implicitly identify automorphisms of $\shift[m] \dual V$ with automorphisms of $\dual V$.)
\end{lemma}

\begin{proof}
  Let $f \in \GL(V)$.
  Then the computation
  \begin{align*}
     &\; \iprod {\paren{f(v), \dual{(\inv f)}(\shift[m] \phi)}} {\paren{f(v'), \dual{(\inv f)}(\shift[m] \phi')}} \\
    =&\; (\phi \after \inv f)(f(v')) - (-1)^{n(m-n)} (\phi' \after \inv f)(f(v)) \\
    =&\; \phi(v') - (-1)^{n(m-n)} \phi'(v) \\
    =&\; \iprod {(v, \shift[m] \phi)} {(v', \shift[m] \phi')}
  \end{align*}
  shows that the map in the statement is actually well-defined.
  It is clearly injective.
  We will now show that it is also surjective.
  For this purpose let $f \dirsum g \in \Gamma$, where $f \in \GL(V)$ and $g \in \GL(\shift[m] \dual V)$.
  By using that $\inv f \dirsum \inv g \in \Gamma$ as well, we obtain
  \[ g(\phi)(v) = \iprod {g(\shift[m] \phi)} v = \iprod {\shift[m] \phi} {\inv f(v)} = \phi(\inv f(v)) = \dual{(\inv f)}(\phi)(v) \]
  and hence $g = \dual{(\inv f)}$.
  This shows surjectivity.
\end{proof}

By basic linear algebra, there is a canonical isomorphism of graded vector spaces $V \tensor \shift[m] \dual V \iso \shift[m] \End{V}$ (here $\End V$ is concentrated in degree $0$ since $V$ is concentrated in a single degree).
We will now use this to identify $\coinv {(H^{\tensor S} \tensor V^{\tensor I} \tensor W^{\tensor J})} \Gamma$ in terms of $\End{V}$.
To make this precise, we need to introduce some notation.

\begin{definition} \label{not:oMatch_tensor}
  Let $S$ be a finite linearly ordered set, $A$ and $B$ two disjoint subsets of $S$ such that $A \disjunion B = S$ (i.e.\ an ordered partition of $S$ into two subsets), and $V_1$ and $V_2$ two graded vector spaces.
  Then we denote by $(V_1, V_2)^{\tensor (A, B)}$ the tensor product $\Tensor_{s \in S} V_s$ where $V_s \defeq V_1$ if $s \in A$ and $V_s \defeq V_2$ if $s \in B$.
  
  A morphism $f \colon S \to S'$ of $\oBij$ induces a map
  \[ f_* \colon (V_1, V_2)^{\tensor (A, B)} \to (V_1, V_2)^{\tensor (f(A), f(B))} \]
  by permuting the factors.
  When $V_1$ is concentrated in degree $n$ and $V_2$ in degree $m-n$, we denote by $\sgn_{A,B}(f) \in \set{\pm 1}$ the sign incurred by this permutation, i.e.\ it is chosen such that
  \[ f_*(\eTensor_{s \in S} v_s) = \sgn_{A,B}(f) \mult \eTensor_{s' \in S'} v_{\inv f(s')} \]
  holds.
\end{definition}

\begin{definition}
  Let $S$ be a finite linearly ordered set.
  We denote by $\oMatch(S)$ the set of \emph{ordered matchings} of $S$, i.e.\ (ordered) pairs $(A, B)$ such that $A$ and $B$ are disjoint subsets of $S$ with $\card A = \card B$ and $A \disjunion B = S$.
  (In particular $\oMatch(S)$ is empty if $\card S$ is odd.)
  This assembles into a functor $\oMatch \colon \oBij \to \Set$ by letting a morphism $f \colon S \to S'$ act via $(A, B) \mapsto (f(A), f(B))$.

  Also note that the linear order of $S$ restricts to a linear order on $A$ and a linear order on $B$.
  In particular there is a unique order-preserving bijection $A \to B$.
  We will denote this bijection by $\mu_{A,B}$.
\end{definition}

\begin{definition}
  Let $S$, $I$, and $J$ be finite linearly ordered sets.
  We denote by $\oMatch(S, I, J) \subseteq \oMatch(S \cop I \cop J)$ the subset of those ordered matchings $(A, B)$ such that $I \subseteq A$ and $J \subseteq B$.
  This assembles into a functor $\oMatch \colon \oBij \times \oBij \times \oBij \to \Set$ by letting a morphism $(f, g, h)$ act via $f \cop g \cop h$.
\end{definition}

\begin{remark}
  Note that there is a canonical bijection $\oMatch(S, I, J) \iso \oMatch(S)$, which is natural in $\oBij \times \oBij \times \oBij$.
\end{remark}

\begin{lemma} \label{lemma:reduce_to_end}
  Let $S$, $I$, and $J$ be finite linearly ordered sets.
  Then there is an isomorphism of $(\Symm S \times \Symm I \times \Symm J)$-modules
  \[ \Theta_1 \colon \Dirsum_{(A, B) \in \oMatch(S, I, J)} \shift[m \card A] \coinv {\big(\End{V}^{\tensor A}\big)} \Gamma  \xlongto{\iso}  \coinv {\big(H^{\tensor S} \tensor V^{\tensor I} \tensor W^{\tensor J}\big)} \Gamma \]
  where $g \in \GL(V) \iso \Gamma$ acts on $f \in \End V$ via the conjugation $g \act f = g f \inv g$, and diagonally on the tensor products.
  The $(\Symm S \times \Symm I \times \Symm J)$-module structure on the left hand side is given as follows: a morphism $(f, g, h)$ of $\Symm S \times \Symm I \times \Symm J$ maps the factor corresponding to $(A, B) \in \oMatch(S, I, J)$ into the factor corresponding to $(c(A), c(B))$ via the map given by
  \[ \shift[m \card A] \eqcl { \eTensor_{a \in A} \Ematrix {i_a} {j_a} }  \longmapsto  \shift[m \card A] \eqcl {\sgn_{A,B}(c) \mult \eTensor_{a' \in c(A)} \Ematrix {i_{\inv\sigma(a')}} {j_{\inv\tau(a')}} } \]
  where $c \defeq f \cop g \cop h$, $\sigma \defeq \restrict c A$, and $\tau \defeq \inv{(\mu_{c(A), c(B)})} \after \restrict c B \after \mu_{A, B}$, and the $\Ematrix i j$ are single-entry matrices with respect to some fixed basis of $V$ (the resulting map is independent of this choice).
  
  Moreover $\Theta_1$ can be chosen such that, for any basis $e$ of $V$, it fulfills, on the factor corresponding to $(A, B) \in \oMatch(S, I, J)$,
  \begin{align*}
    \Theta_1 ( \shift[m \card A] \eqcl { \eTensor_{a \in A} \Ematrix {i_a} {j_a} } )  &=  \eqcl {\eTensor_{s \in S \cop I \cop J} h_s} \\
    h_s &\defeq \begin{cases}e_{i_s} \in V, & \text{if } s \in A \\ \shift[m] \dual{e_{j_{\inv{\mu_{A,B}}(s)}}} \in \shift[m] \dual V, & \text{if } s \in B\end{cases}
  \end{align*}
  where the $E_{i,j}$ are with respect to $e$.
\end{lemma}

\begin{proof}
  There is a canonical isomorphism
  \[ H^{\tensor S} \tensor V^{\tensor I} \tensor W^{\tensor J}  \iso  \Dirsum_{(A,B)} (V, \shift[m] \dual V)^{\tensor (A, B)} \]
  where $(A,B)$ runs over all (ordered) partitions of $S \cop I \cop J$ into two subsets such that $I \subseteq A$ and $J \subseteq B$.
  This is an isomorphism of $(\Symm S \times \Symm I \times \Symm J \times \Gamma)$-modules where a morphism $(f, g, h)$ of $\Symm S \times \Symm I \times \Symm J$ acts on both sides by permuting the factors (on the right hand side this maps the $(A, B)$ summand into the $(c(A), c(B))$ summand, where $c \defeq f \cop g \cop h$) and an element $g \in \GL(V) \iso \Gamma$ acts diagonally (acting on $\dual V$ by $\dual{(\inv g)}$).

  Now note that the $\Gamma$-coinvariants of $(V, \shift[m] \dual V)^{\tensor (A, B)}$ are trivial when $\card A \neq \card B$ since $2 \in \GL(V)$ acts via multiplication by $\nicefrac{2^{\card A}}{2^{\card B}}$.
  Hence the inclusion
  \[ \Dirsum_{(A,B) \in \oMatch(S, I, J)} (V, \shift[m] \dual V)^{\tensor (A, B)} \longincl H^{\tensor S} \tensor V^{\tensor I} \tensor W^{\tensor J} \]
  becomes an isomorphism after taking $\Gamma$-coinvariants.

  Now let $T$ be a finite linearly ordered set.
  Then there is, for any $(A, B) \in \oMatch(T)$, a canonical isomorphism of $\Gamma$-modules
  \[ \Psi_{A, B} \colon (V, \shift[m] \dual V)^{\tensor (A, B)} \xlongto{\iso} (\shift[m] \dual V \tensor V)^{\tensor A} \]
  given by the following permutation: we consider the tensor product on the right hand side to be indexed over the set $A \times \set{0, 1}$ equipped with the lexicographical order and permute according to the bijection $\psi_{A, B} \colon T \to A \times \set{0, 1}$ that sends $a \in A$ to $(a, 1)$ and $b \in B$ to $(\inv{(\mu_{A, B})}(b), 0)$.

  The isomorphisms $\sgn_{A, B}(\psi_{A, B}) \mult \Psi_{A, B}$ assemble into an isomorphism of $(\Symm T \times \Gamma)$-modules
  \[ \Psi \colon \Dirsum_{(A,B) \in \oMatch(T)} (V, \shift[m] \dual V)^{\tensor (A, B)} \xlongto{\iso} \Dirsum_{(A,B) \in \oMatch(T)} (\shift[m] \dual V \tensor V)^{\tensor A} \]
  where an element $c \in \Symm T$ acts on the right hand side by sending the factor corresponding to $(A, B)$ into the factor corresponding to $(c(A), c(B))$ via
  \[ c_*^{A,B} \big( \eTensor_{a \in A} (\shift[m] \phi_a \tensor v_a) \big)  =  \sgn_{A, B}(c) \mult \big( \eTensor_{a \in A} (\shift[m] \phi_{\inv\tau(a)} \tensor v_{\inv\sigma(a)}) \big) \]
  where $\sigma$ and $\tau$ are as in the statement of the lemma.
  Said differently we have
  \[ c_*^{A,B} = \sgn_{A, B}(\psi_{A, B}) \sgn_{c(A), c(B)}(\psi_{c(A), c(B)}) \mult ( \Psi_{c(A), c(B)} \after c_* \after \inv{\Psi_{A, B}} ) \]
  which maybe makes it clearer why $\Psi$ is compatible with the $\Symm T$-action.
  
  Now consider the case that $T = S \cop I \cop J$.
  Then it is clear that the isomorphism $\Psi$ restricts to an isomorphism
  \[ \Dirsum_{(A,B) \in \oMatch(S, I, J)} (V, \shift[m] \dual V)^{\tensor (A, B)} \xlongto{\iso} \Dirsum_{(A,B) \in \oMatch(S, I, J)} (\shift[m] \dual V \tensor V)^{\tensor A} \]
  of $(\Symm S \times \Symm I \times \Symm J \times \Gamma)$-modules.

  Lastly note that there is a $\Gamma$-module isomorphism
  \begin{align*}
    \Phi \colon \shift[m] \dual V \tensor V  &\longto  \shift[m] \End V \\
    \shift[m] \phi \tensor v  &\longmapsto  \shift[m] (v' \mapsto \phi(v') v)
  \end{align*}
  where $\Gamma \iso \GL(V)$ acts on $\End V$ by conjugation.
  Together with the canonical isomorphism $(\shift[m] \End V)^{\tensor A} \iso \shift[m \card A] (\End V)^{\tensor A}$ these isomorphisms combine into an isomorphism
  \[ \Theta_1 \colon \Dirsum_{(A, B) \in \oMatch(S, I, J)} \shift[m \card A] \coinv {\big(\End V^{\tensor A}\big)} \Gamma \xlongto{\iso} \coinv {\big(H^{\tensor S} \tensor V^{\tensor I} \tensor W^{\tensor J}\big)} \Gamma \]
  as desired.
  For the claims made about the induced $(\Symm S \times \Symm I \times \Symm J)$-module structure on the left hand side and the explicit description of $\Theta_1$, note that, when $e \defeq (e_1, \dots, e_k)$ is a basis of $V$, then $\Phi(\shift[m] \dual {e_j} \tensor e_i) = \shift[m] \Ematrix i j$ where $\Ematrix i j$ is with respect to $e$.
\end{proof}

We have now reduced finding a description of $\coinv {\big(H^{\tensor S} \tensor V^{\tensor I} \tensor W^{\tensor J}\big)} \Gamma$ independent of $H$ to finding a description of $\coinv {\big({\End V}^{\tensor A}\big)} \Gamma$ that is independent of $V$.
Such a description is provided by the classical coinvariant theory for the general linear group.

\begin{proposition}[Fundamental theorem of coinvariant theory] \label{prop:fund_thm_coinv}
  Let $k \in \NN$ such that $k \le \dim V$.
  Then there is an isomorphism of vector spaces
  \[ \coinv {\big({\End V}^{\tensor k}\big)} \Gamma \xlongto{\iso} \grpring {\Symm k} \]
  such that $\eqcl {\Ematrix {\omega(1)} {\sigma(1)} \tensor \dots \tensor \Ematrix {\omega(k)} {\sigma(k)}} \mapsto \inv\omega \sigma$ for all $\omega, \sigma \in \Symm k$ and any basis of $V$.
\end{proposition}

\begin{proof}
  This can, for example, be found in \cite[9.2.5, 9.2.8, and 9.2.10]{Lod}\footnote{Note that in the proof of 9.2.10 there is a small mistake, leading to the result being stated as $\sigma \inv\omega$ instead of the correct $\inv\omega \sigma$}.
\end{proof}

Applying this to our situation, we obtain the following corollary.

\begin{corollary} \label{cor:from_end_to_grpring}
  Let $S$, $I$, and $J$ be finite linearly ordered sets such that $\card S + \card I + \card J \le 2 (\dim V)$.
  Then there is an isomorphism of $(\Symm S \times \Symm I \times \Symm J)$-modules
  \[ \Theta_2 \colon \Dirsum_{(A, B) \in \oMatch(S, I, J)} \shift[m \card A] \coinv {\big(\End V^{\tensor A}\big)} \Gamma \xlongto{\iso}  \Dirsum_{(A, B) \in \oMatch(S, I, J)} \shift[m \card A] \grpring{\Symm A} \]
  where $\Symm S \times \Symm I \times \Symm J$ acts on the left hand side as in \Cref{lemma:reduce_to_end} and on the right hand side as follows: an element $(f, g, h)$ maps the factor corresponding to $(A, B) \in \oMatch(S, I, J)$ into the factor corresponding to $(c(A), c(B))$ via the map
  \begin{align*}
    \shift[m \card A] \grpring{\Symm A}  &\longto  \shift[m \card A] \grpring{\Symm {c(A)}} \\
    \shift[m \card A] \omega  &\longmapsto  \shift[m \card A] (\sgn_{A, B}(c) \mult \sigma \omega \inv\tau)
  \end{align*}
  where $c \defeq f \cop g \cop h$, and $\sigma$ and $\tau$ are as in \Cref{lemma:reduce_to_end}.

  Moreover $\Theta_2$ can be chosen such that, for any basis of $V$ and $\omega \in \Symm A$, it fulfills, on the factor corresponding to $(A, B) \in \oMatch(S, I, J)$, that
  \[ \Theta_2 \big( \shift[m \card A] \eqcl {\Ematrix 1 {\omega(1)} \tensor \dots \tensor \Ematrix a {\omega(a)}} \big) = \shift[m \card A] \omega \]
  where we identify $A$ with $\set{1, \dots, a \defeq \card A}$ via the unique order-preserving bijection.
\end{corollary}

\begin{proof}
  On each summand, $\Theta_2$ is given by a shift of the isomorphism in \Cref{prop:fund_thm_coinv}.
  Now let $(A, B) \in \oMatch(S, I, J)$ and $\omega \in \Symm A$, and fix a basis of $V$.
  We will also use the notation from \Cref{lemma:reduce_to_end}.
  First note that by \Cref{prop:fund_thm_coinv} the elements $\eqcl {\Ematrix 1 {\omega(1)} \tensor \dots \tensor \Ematrix a {\omega(a)}}$ of $\coinv {\big(\End V^{\tensor A}\big)} \Gamma$, with $\omega$ running over $\Symm A$, form a basis.
  Hence it is enough to check the compatibility of $\Theta_2$ with the $(\Symm S \times \Symm I \times \Symm J)$-action on those elements.
  To this end we note that $\Theta_2 \big( \shift[m \card A] \eqcl {\Ematrix 1 {\omega(1)} \tensor \dots \tensor \Ematrix a {\omega(a)}} \big) = \shift[m \card A] \omega$ and
  \begin{align*}
     &\; \Theta_2 \big( {\sgn_{A, B}(c)} \mult \shift[m \card A] \eqcl {\Ematrix {\inv\sigma(1)} {\omega(\inv\tau(1))} \tensor \dots \tensor \Ematrix {\inv\sigma(a)} {\omega(\inv\tau(a))}} \big) \\
    =&\; \shift[m \card A] (\sgn_{A, B}(c) \mult \sigma \omega \inv\tau)
  \end{align*}
  which is what we wanted to show.
\end{proof}

\begin{remark} \label{rem:from_end_to_grpring}
  Setting $T = S \cop I \cop J$, we note that $\Theta_2$ restricts to an isomorphism of $(\Symm S \times \Symm I \times \Symm J)$-modules on the summands indexed by $\oMatch(S, I, J) \subseteq \oMatch(T)$.
\end{remark}

We have now almost achieved what we set out do to in this subsection, as we have obtained, when $\card S + \card I + \card J \le 2 (\dim V)$, an isomorphism
\begin{equation} \label{eq:H_coinv}
  \coinv {\big(H^{\tensor S} \tensor V^{\tensor I} \tensor W^{\tensor J}\big)} \Gamma  \iso  \Dirsum_{(A, B) \in \oMatch(S, I, J)} \shift[m \card A] \grpring{\Symm A}
\end{equation}
with the right hand side independent of $H$ as desired.
To generalize this to expressions of the form $\coinv {((\opmod P \tensor[\Symm S] H^{\tensor S}) \tensor V^{\tensor I} \tensor W^{\tensor J})} \Gamma$, we will use the abstract lemmas of \cref{sec:Kan_extensions}.
To state the result in its simplest form, we need to introduce a bit of notation.

\begin{definition}
  Let $I$ and $J$ be finite linearly ordered sets.
  We denote by $\Match{I,J}$ the following groupoid:
  \begin{itemize}
    \item Objects are tuples $(S, A, B, \mu)$ with $S$ a finite linearly ordered set, $(A, B) \in \oMatch(S, I, J)$ an ordered matching, and $\mu \colon A \to B$ a bijection.
    \item A morphism $(S, A, B, \mu) \to (S', A', B', \mu')$ is a tuple $(f, g, h)$ of bijections $f \colon S \to S'$, $g \in \Symm I$, and $h \in \Symm J$, such that the map $c \defeq f \cop g \cop h$ fulfills $c(A) = A'$, $c(B) = B'$, and $\restrict c B \after \mu = \mu' \after \restrict c A$.
  \end{itemize}
  Moreover, for $S \in \oBij$, we denote by $\Match[S]{I,J} \subseteq \Match{I,J}$ the full subgroupoid spanned by objects of the form $(S', A', B', \mu')$ with $S' = S$.
  Also note that there is a canonical functor $\pr[I,J] \colon \Match{I,J} \to \Symm I \times \Symm J$ given on morphisms by sending $(f, g, h)$ to $(g, h)$.
  We denote by $\Match{\id[I],\id[J]} \subseteq \Match{I,J}$ the wide subgroupoid of those morphisms that are sent to $(\id[I], \id[J])$ by $\pr[I,J]$, and analogously for $\Match[S]{\id[I],\id[J]} \subseteq \Match[S]{I,J}$.
\end{definition}

\begin{definition}
  We write $\sgn_{n, m}$ for the $\Match{I,J}$-module given by
  \[ (S, A, B, \mu)  \longmapsto  (\shift[n] \basefield, \shift[m-n] \basefield)^{\tensor (A, B)} \]
  (see \Cref{not:oMatch_tensor}).
  Said differently, it is $\shift[m \card A] \basefield$ on which a morphism $(f, g, h)$ acts by $\sgn_{A, B}(c)$, where $c \defeq f \cop g \cop h$.
\end{definition}

With only a little bit more work, the lemmas of \cref{sec:Kan_extensions} now specialize to the following.
(See \cite[Corollary~9.8]{BM} for a similar statement in the case where $H$ is concentrated in a single degree and $I = J = \emptyset$.)

\begin{proposition} \label{prop:coinv_monomial}
  Let $S$, $I$, and $J$ be finite linearly ordered sets such that $\card S + \card I + \card J \le 2 (\dim V) = \dim H$, and let $P$ be a $\Symm S$-module.
  Then there is an isomorphism of $(\Symm I \times \Symm J)$-modules
  \[ \colim{\Match[S]{\id[I],\id[J]}} (P \tensor \sgn_{n, m})  \xlongto{\iso}  \coinv { \big( (\opmod P \tensor[\Symm S] H^{\tensor S}) \tensor V^{\tensor I} \tensor W^{\tensor J} \big) } \Gamma \]
  where a morphism $(f, \id[I], \id[J])$ of $\Match[S]{\id[I],\id[J]}$ acts on $P$ by $f \in \Symm S$, and $\Gamma$ acts trivially on $P$ and diagonally on the tensor product.
  The action of an element $(g, h) \in \Symm I \times \Symm J$ on the left hand side is given by sending an element represented by $(A, B, \mu) \in \Match[S]{\id[I],\id[J]}$ and $x \in P \tensor \sgn_{n, m}$ to the element represented by $(A, B, \restrict c B \after \mu \after \inv {(\restrict c A)})$ and $\sgn_{A, B}(c) x$, where $c \defeq {\id[S]} \cop g \cop h$.
  
  Moreover, the isomorphism above can be chosen such that, for any basis $e$ of V, it maps the element of the colimit represented by $(A, B, \mu) \in \Match[S]{\id[I],\id[J]}$ and $p \tensor \shift[m \card A] 1 \in P \tensor \sgn_{n, m}$ to $\eqcl {p \tensor \eTensor_{t \in S \cop I \cop J} h_t}$ where
  \[ h_t  \defeq  \begin{cases} e_{\inv a(t)} \in V, & \text{if } t \in A \\ \shift[m] \dual{e_{\inv a(\inv \mu(t))}} \in \shift[m] \dual V, & \text{if } t \in B \end{cases} \]
  (here $a \colon \finset{\card A} \to A$ is the unique order-preserving bijection).
\end{proposition}

\begin{proof}
  We start by rewriting the expression yielded by \cref{cor:from_end_to_grpring,rem:from_end_to_grpring}.
  We have isomorphisms of $(\Symm S \times \Symm I \times \Symm J)$-modules
  \begin{align*}
    \Dirsum_{(A, B) \in \oMatch(S, I, J)} \shift[m \card A] \grpring{\Symm A}  &\iso  \Dirsum_{(A, B) \in \oMatch(S, I, J)} \shift[m \card A] \linspan{\Bij{A}{B}} \\
      &\iso  \Dirsum_{(A, B, \mu) \in \oMatch_\mu(S, I, J)} \shift[m \card A] \linspan{\mu}
  \end{align*}
  where, for the first isomorphism, we identify $\Symm A$ with the set of (not necessarily order preserving) bijections $\Bij{A}{B}$ via the map $\omega \mapsto \mu_{A,B} \after \inv \omega$.
  For the second isomorphism, we set $\oMatch_\mu(S, I, J) \defeq \Ob(\Match[S]{I,J})$.
  An element $(f, g, h) \in \Symm S \times \Symm I \times \Symm J$ acts on $\oMatch_\mu(S, I, J)$ via
  \[ (A, B, \mu) \longmapsto \big( c(A), c(B), c \act \mu \defeq \restrict c B \after \mu \after \inv {(\restrict c A)} \big) \]
  where $c \defeq f \cop g \cop h$, and on $\linspan \mu$ by sending $\mu$ to $\sgn_{A,B}(c) \mult (c \act \mu)$.
  
  Now we have isomorphisms of $(\Symm I \times \Symm J)$-modules
  \begin{align*}
    \coinv {\big( (\opmod P \tensor[\Symm S] H^{\tensor S}) \tensor V^{\tensor I} \tensor W^{\tensor J} \big)} \Gamma  &=  \coinv {\big( \coinv {(P \tensor H^{\tensor S})} {\Symm S} \tensor V^{\tensor I} \tensor W^{\tensor J} \big)} \Gamma \\
      &\iso  \coinv {\big( \coinv {(P \tensor H^{\tensor S} \tensor V^{\tensor I} \tensor W^{\tensor J})} \Gamma \big)} {\Symm S} \\
      &\iso  \coinv {\big( P \tensor \coinv {(H^{\tensor S} \tensor V^{\tensor I} \tensor W^{\tensor J})} \Gamma \big)} {\Symm S}
  \end{align*}
  since the actions of $\Symm S \times \Symm I \times \Symm J$ and $\Gamma$ commute and since $\Gamma$ acts trivially on $P$.
  By \cref{lemma:reduce_to_end,cor:from_end_to_grpring,rem:from_end_to_grpring}, as well as the preceding discussion, the rightmost expression above is isomorphic to
  \begin{equation} \label{eq:pf_coinv_monomial}
    \coinv {\Big( P \tensor \Dirsum_{(A, B) \in \oMatch(S, I, J)} \shift[m \card A] \grpring{\Symm A} \Big)} {\Symm S}  \iso  \coinv {\Big( \Dirsum_{(A, B, \mu) \in \oMatch_\mu(S, I, J)} \big( P \tensor \shift[m \card A] \linspan{\mu} \big) \Big)} {\Symm S}
  \end{equation}
  as $(\Symm I \times \Symm J)$-modules.
  To describe the action on the sum on the right hand side, let $(f, g, h) \in \Symm S \times \Symm I \times \Symm J$ and set $c \defeq f \cop g \cop h$.
  Then it sends the summand corresponding to $(A, B, \mu)$ to the summand corresponding to $(c(A), c(B), c \act \mu)$ by acting on $P$ via $f$ and on $\linspan \mu$ as described above.
  The quotient is taken with respect to the action of $\Symm S \iso \Symm S \times \set{\id[I]} \times \set{\id[J]}$.
  
  Now we can apply \cref{lemma:functors_from_gc}.
  Noting that there is a canonical isomorphism
  \[ \gc[\Symm S \times \Symm I \times \Symm J] \oMatch_\mu(S, I, J)  \iso  \Match[S]{I,J} \]
  of categories over $\Symm I \times \Symm J$, we obtain an isomorphism of the right hand side of \eqref{eq:pf_coinv_monomial} with
  \[ \Lan{\Match[S]{I,J} \to \Symm I \times \Symm J} (P \tensor \sgn_{n, m}) \]
  as $(\Symm I \times \Symm J)$-modules, where a morphism $(f, g, h)$ of $\Match[S]{I,J}$ acts on $P$ by $f \in \Symm S$.
  To this we can, in turn, apply \cref{lemma:groupoid_Kan_ext} to obtain the desired description in terms of a colimit: given an object $(A, B, \mu) \in \Match[S]{I,J}$ and an element $(g, h) \in \Symm I \times \Symm J$, we pick the lift $(\id[S], g, h) \colon (A, B, \mu) \to (A, B, c \act \mu)$ where $c \defeq {\id[S]} \cop g \cop h$.
  
  The explicit description of the isomorphism is obtained by chasing through the explicit descriptions of its constituent parts.
\end{proof}

\subsection{Coinvariants of Schur functors} \label{sec:coinv_Schur}

In this subsection we want to compute
\[ \coinv {\big( \Schur {\bijmod P} {H} \tensor V^{\tensor I} \tensor W^{\tensor J} \big)} {\Gamma} \]
as a $(\Symm I \times \Symm J)$-module, where $\bijmod P$ is some $\oBij$-module, and $H$, $V$, $W$, and $\Gamma$ are as in the preceding subsection.
Noting that $\Schur {\bijmod P} {H}$ is a sum of monomial functors in $H$, \cref{prop:coinv_monomial} yields an identification as desired in a range depending on the dimension of $H$.
To obtain a description of the whole thing we thus need to let the dimension of $H$ go to infinity.

More precisely we fix the following notation for the rest of this section:
\[ H_1 \to H_2 \to H_3 \to \dots \]
is a sequence of maps of $\Sp[m]$ (see \cref{not:graded_symplectic}) such that $H_i$ is concentrated in dimensions $n$ and $m-n$ and such that the sequence $\dim{H_i}$ diverges.
Moreover we write $\Gamma_i \defeq \Aut[{\Sp[m]}](H_i)$, and denote by $V_i$ and $W_i$ the degree $n$ and $m - n$ parts of $H_i$, respectively.

Note that there are maps $\Gamma_i \to \Gamma_{i+1}$ given by extending an automorphism of $H_i$ by the identity on the symplectic complement of $H_i$ in $H_{i+1}$.
Moreover the maps $H_i \to H_{i+1}$ become $\Gamma_i$-equivariant when $H_{i+1}$ is equipped with the restricted action.
In particular, for any functor $F$ from $\Sp[m]$ to some cocomplete category $\cat C$, we obtain a sequence $\coinv {F(H_1)} {\Gamma_1} \to \coinv {F(H_2)} {\Gamma_2} \to \dots$ and thus can make sense of the expression $\colim{i \in \NNo} \coinv {F(H_i)} {\Gamma_i}$.

\begin{definition}
  A \emph{compatible basis} of the $V_i$ is a sequence $e_1, e_2, \dots$ of elements of $\colim{i} V_i$ such that, for each $i \ge 1$, each of the elements $e_1, \dots, e_{\dim {V_i}}$ is represented by an (automatically unique) element of $V_i$, and the set of these representatives forms a basis of $V_i$.
\end{definition}

\begin{proposition} \label{lemma:coinv_Schur}
  Let $I$ and $J$ be finite linearly ordered sets and $\bijmod P$ a $\oBij$-module.
  Then there is an isomorphism of $(\Symm I \times \Symm J)$-modules
  \[ \colim{(S, A, B, \mu) \in \Match{\id[I],\id[J]}} \big( \bijmod P(S) \tensor \sgn_{n, m} \big)  \xlongto{\iso}  \colim{i \in \NNo} \coinv {\big( \Schur {\bijmod P} {H_i} \tensor V_i^{\tensor I} \tensor W_i^{\tensor J} \big)} {\Gamma_i} \]
  where a morphism $(f, \id[I], \id[J])$ of $\Match{\id[I],\id[J]}$ acts on $P$ by $f \in \Symm S$.
  The action of an element $(g, h) \in \Symm I \times \Symm J$ on the left hand side is given by sending an element represented by $(S, A, B, \mu) \in \Match{\id[I],\id[J]}$ and $x \in \bijmod P(S) \tensor \sgn_{n, m}$ to the element represented by $(S, A, B, \restrict c B \after \mu \after \inv {(\restrict c A)})$ and $\sgn_{A, B}(c) x$, where $c \defeq {\id[S]} \cop g \cop h$.
  
  Moreover, this isomorphism can be chosen such that, for any compatible basis $(e_j)_{j \in \NNpos}$ of the $V_i$, it maps an element represented by $(S, A, B, \mu) \in \Match{\id[I],\id[J]}$ and $p \tensor \shift[m \card A] 1 \in \bijmod P(S) \tensor \sgn_{n, m}$ to the element $\eqcl {p \tensor \eTensor_{t \in S \cop I \cop J} h_t}$ of
  \[ \colim{i \in \NNo} \coinv {\big( (\opmod {\bijmod P(S)} \tensor[\Symm S] H_i^{\tensor S}) \tensor V_i^{\tensor I} \tensor W_i^{\tensor J} \big)} {\Gamma_i}  \subseteq  \colim{i \in \NNo} \coinv {\big( \Schur {\bijmod P} {H_i} \tensor V_i^{\tensor I} \tensor W_i^{\tensor J} \big)} {\Gamma_i} \]
  where $h_t$ is as in \cref{prop:coinv_monomial} (and we use the definition of $\Schur {\bijmod P} {H_i}$ of \cref{rem:Schur_functor}).
\end{proposition}

\begin{proof}
  We have
  \[ \coinv {\big( \Schur {\bijmod P} {H_i} \tensor V_i^{\tensor I} \tensor W_i^{\tensor J} \big)} {\Gamma_i}  \iso  \Dirsum_{k \in \NN} \coinv {\big( (\opmod {\bijmod P(\finset k)} \tensor[\Symm{\finset k}] {H_i}^{\tensor \finset k}) \tensor V_i^{\tensor I} \tensor W_i^{\tensor J} \big)} {\Gamma_i}\]
  by definition of $\Schur{\bijmod P}{\blank}$.
  When $k \le \dim {H_i} - \card I - \card J$, there is an isomorphism
  \begin{equation} \label{eq:coinv_Schur_single}
    \colim{\Match[\finset k]{\id[I],\id[J]}} \big( \bijmod P(\finset k) \tensor \sgn_{n, m} \big)  \xlongto{\iso}  \coinv {\big( (\opmod {\bijmod P(\finset k)} \tensor[\Symm {\finset k}] {H_i}^{\tensor \finset k}) \tensor V_i^{\tensor I} \tensor W_i^{\tensor J} \big)} {\Gamma_i}
  \end{equation}
  provided by \Cref{prop:coinv_monomial}.
  It is compatible with the maps $H_i \to H_{i+1}$ (i.e.\ it is a natural isomorphism of sequential diagrams, where the left hand side is considered as a constant functor).
  To see this, we chose a basis of $V_i$ (the degree $n$ part of $H_i$), extend it to a basis of $V_{i+1}$, and use the explicit description of the isomorphism \eqref{eq:coinv_Schur_single}.
  Hence we have, for each $k \in \NN$, an isomorphism
  \[ \colim{\Match[\finset k]{\id[I],\id[J]}} \big( \bijmod P(\finset k) \tensor \sgn_{n, m} \big)  \xlongto{\iso}  \colim{i \in \NNo} \coinv {\big( (\opmod {\bijmod P(\finset k)} \tensor[\Symm {\finset k}] {H_i}^{\tensor \finset k}) \tensor V_i^{\tensor I} \tensor W_i^{\tensor J} \big)} {\Gamma_i} \]
  which, by setting $\MatchPermGrpd{\id[I],\id[J]}  \defeq  \coprod_{k \in \NN} \Match[\finset k]{\id[I],\id[J]}  \subset  \Match{\id[I],\id[J]}$, assemble into an isomorphism
  \begin{multline*}
    \colim{\MatchPermGrpd{\id[I],\id[J]}} \big( \bijmod P \tensor \sgn_{n, m} \big)  \iso  \Dirsum_{k \in \NN} \colim{\Match[\finset k]{\id[I],\id[J]}} \big( \bijmod P(\finset k) \tensor \sgn_{n, m} \big) \\
      \xlongto{\iso}  \colim{i \in \NNo} \coinv {\big( \Schur {\bijmod P} {H_i} \tensor V_i^{\tensor I} \tensor W_i^{\tensor J} \big)} {\Gamma_i}
  \end{multline*}
  \[  \]
  since colimits commute with direct sums.
  Since the inclusion $\MatchPermGrpd{\id[I],\id[J]} \to \Match{\id[I],\id[J]}$ is an equivalence of categories, it induces a canonical isomorphism
  \[ \colim{\MatchPermGrpd{\id[I],\id[J]}} \big( \bijmod P \tensor \sgn_{n, m} \big)  \xlongto{\iso}  \colim{\Match{\id[I],\id[J]}} \big( \bijmod P \tensor \sgn_{n, m} \big) \]
  which finishes the construction.
  
  We will now prove the claimed explicit description of the composed isomorphism.
  To this end, consider the element represented by $(S, A, B, \mu) \in \Match{\id[I],\id[J]}$ and $p \tensor \shift[m \card A] 1 \in \bijmod P(S) \tensor \sgn_{n, m}$.
  It is equal, for any choice of bijection $f \colon S \to \finset{\card S}$, to the element represented by $\big( \finset{\card S}, c(A), c(B), \restrict c B \after \mu \after \inv {(\restrict c A)} \big)$ and $\sgn_{A, B}(c) \mult \big( f_*(p) \tensor \shift[m \card A] 1 \big)$, where $c \defeq f \cop {\id[I]} \cop {\id[J]}$.
  On such an element the isomorphism is, by construction, given by the isomorphism of \Cref{prop:coinv_monomial}.
  The representative one obtains using the explicit description of that isomorphism represents the same element as the representative we provided in the statement of this proposition (by translating it back to $S$ via $\inv f$).
  In the same way one proves that the action of $\Symm I \times \Symm J$ is as claimed.
\end{proof}

\subsection{Coinvariants of double Schur functors}

In this subsection our goal is to compute
\[ \colim{i \in \NNo} \coinv {\big( \Schur {\bijmod P} {\Schur {\bijmod Q} {H_i}} \tensor V_i^{\tensor I} \tensor W_i^{\tensor J} \big)} {\Gamma_i} \]
as a $(\Symm I \times \Symm J)$-module, where $\bijmod P$ and $\bijmod Q$ are $\oBij$-modules, and $H_i$, $V_i$, $W_i$, and $\Gamma_i$ are as in the preceding subsection.
By \cref{lemma:compprod_and_Schur}, we have an isomorphism $\Schur {\bijmod P} {\Schur {\bijmod Q} {H_i}} \iso \Schur {(\bijmod P \compprod \bijmod Q)} {H_i}$.
Hence we can use \cref{lemma:coinv_Schur} to obtain a description as desired.
We now introduce some notation to state the result in a nice way.

\begin{definition} \label{def:dirgraph}
  Let $I$ and $J$ be finite linearly ordered sets.
  We denote by $\DirGraph[{\id[I],\id[J]}]$ the following groupoid:
  \begin{itemize}
    \item Objects are \emph{directed $(I,J)$-graphs}, that is, tuples $(F, N, S, T, \mu, a)$ with $F$ and $N$ finite linearly ordered sets, $(S, T) \in \oMatch(F, I, J)$ an ordered matching, $\mu \colon S \to T$ a bijection, and $a \colon F \to N$ a map of sets.
    \item A morphism from $(F, N, S, T, \mu, a)$ to $(F', N', S', T', \mu', a')$ is a tuple of bijections $(f \colon F \to F'$, $k \colon N \to N')$ such that $c(S) = S'$, $c(T) = T'$, $\restrict c T \after \mu = \mu' \after \restrict c S$, and $k \after a = a' \after f$, where we set $c \defeq f \cop {\id[I]} \cop {\id[J]}$.
  \end{itemize}
  When $I = J = \emptyset$, we just write $\DirGraph$ and call its objects \emph{directed graphs}.
\end{definition}

\begin{remark} \label{rem:DirGraph_intuition}
  We think of a directed $(I,J)$-graph as a directed graph with $I$ inputs and $J$ outputs.
  Under this viewpoint we have: $N$ is the set of \emph{vertices} (or nodes), $F$ is the set of \emph{internal flags} (or half-edges), $a$ describes to which vertex a flag is attached, $S$ is the set of those flags which are \emph{sources} (i.e.\ outgoing) which includes the inputs, similarly $T$ is the set of those flags which are \emph{targets} (i.e.\ incoming) which includes the outputs, and $\mu$ describes the (directed) edges (i.e.\ $s \in S$ and $\mu(s) \in T$ are connected and form an edge).
\end{remark}

\begin{lemma} \label{lemma:Match_sgn_alt}
  There is an isomorphism from $\sgn_{n, m}$ to the $\Match{I,J}$-module given by
  \[ (S, A, B, \mu)  \longmapsto  (\shift[m] \basefield)^{\tensor A} \]
  where a morphism $(f, g, h) \colon (S, A, B, \mu) \to (S', A', B', \mu')$ of $\Match{I,J}$ acts by permuting the factors according to $\restrict c A \colon A \to A'$, where $c \defeq f \cop g \cop h$.
  This isomorphism maps $\shift[m \card A] 1 \in \sgn_{n, m}$ to $\epsilon \mult \shift[m \card A] 1 \in (\shift[m] \basefield)^{\tensor A}$ where $\epsilon$ is the sign incurred by permuting
  \[ \Tensor_{t \in S \cop I \cop J} h_t  \quad \text{into} \quad  \Tensor_{a \in A} (h_a \tensor h_{\mu(a)}) \]
  where $h_t \defeq \shift[n] 1$ if $t \in A$ and $h_t \defeq \shift[m-n] 1$ if $t \in B$.
\end{lemma}

\begin{proof}
  Let $p, q \colon \Match{I,J} \to \Match{\emptyset,\emptyset}$ be the canonical functors given, on objects, by
  \begin{align*}
    p (S, A, B, \mu) &= (S \cop I \cop J, A, B, \mu) \\
    q (S, A, B, \mu) &= (T, A, B, \mu)
  \end{align*}
  where $T$ is equal to $S \cop I \cop J$ as a set but is equipped with the linear order uniquely determined by the following two properties:
  \begin{itemize}
    \item The element $a \in A$ is the predecessor of $\mu(a) \in B$.
    \item Its restriction to $A$ agrees with the linear order on $A$ obtained by restricting the one of $S \cop I \cop J$.
  \end{itemize}
  The identity of the underlying sets determines a natural isomorphism $p \to q$.
  The map in the statement is obtained by applying $\sgn_{n, m}$ to this natural transformation.
\end{proof}

The following proposition contains the desired description of the $(\Symm I \times \Symm J)$-module $\colim{i \in \NNo} \coinv {( \Schur {\bijmod P} {\Schur {\bijmod Q} {H_i}} \tensor V_i^{\tensor I} \tensor W_i^{\tensor J} )} {\Gamma_i}$.
(See \cite[Theorem~9.12]{BM} for a similar statement in the case where $H$ is concentrated in a single degree and $I = J = \emptyset$.)

\begin{proposition} \label{lemma:coinv_double_Schur}
  Let $I$ and $J$ be finite linearly ordered sets, and let $\bijmod P$ and $\bijmod Q$ be $\oBij$-modules.
  Then there is an isomorphism of $(\Symm I \times \Symm J)$-modules
  \[ \colim{(F, N, S, T, \mu, a) \in \DirGraph[{\id[I],\id[J]}]} \big( \bijmod P(N) \tensor \bijmod Q(a) \tensor (\shift[m] \basefield)^{\tensor S} \big)  \xlongto{\iso}  \colim{i \in \NNo} \coinv {\big( \Schur {\bijmod P} {\Schur {\bijmod Q} {H_i}} \tensor V_i^{\tensor I} \tensor W_i^{\tensor J} \big)} {\Gamma_i} \]
  where a morphism $(f, k)$ of $\DirGraph[{\id[I],\id[J]}]$ acts on $\bijmod P(N)$ by $k$, on $\bijmod Q(a)$ by $(f, k)$ considered as a morphism of $\oBij \comma \oBij$, and on $(\shift[m] \basefield)^{\tensor S}$ by permuting the factors according to $\restrict{c}{S}$ where $c \defeq f \cop {\id[I]} \cop {\id[J]}$.
  The action of an element $(g, h) \in \Symm I \times \Symm J$ on the left hand side is given by sending an element represented by $(F, N, S, T, \mu, a) \in \DirGraph[{\id[I],\id[J]}]$ and $x \in \bijmod P(N) \tensor \bijmod Q(a) \tensor (\shift[m] \basefield)^{\tensor S}$ to the element represented by $(F, N, S, T, \restrict c T \after \mu \after \inv {(\restrict c S)}, a)$ and $\sgn(g)^m x$, where $c \defeq {\id[F]} \cop g \cop h$.
  
  Moreover, this isomorphism can be chosen such that, for any compatible basis $(e_j)_{j \in \NNpos}$ of the $V_i$, it maps an element represented by $(F, N, S, T, \mu, a) \in \DirGraph[{\id[I],\id[J]}]$ and
  \[ p \tensor \eTensor_{v \in N} q_v \tensor \shift[m \card S] 1  \in  \bijmod P(N) \tensor \bijmod Q(a) \tensor (\shift[m] \basefield)^{\tensor S} \]
  to the element
  \[ \epsilon \mult \eqcl[\big] { p \tensor \eTensor_{v \in N} \big( q_v \tensor \eTensor_{f \in \inv a(v)} h_f \big) \tensor \eTensor_{f \in I \cop J} h_f }  \in  \colim{i \in \NNo} \coinv {\big( \Schur {\bijmod P} {\Schur {\bijmod Q} {H_i}} \tensor V_i^{\tensor I} \tensor W_i^{\tensor J} \big)} {\Gamma_i} \]
  where
  \[ h_f \defeq
       \begin{cases}
         e_{\inv s(f)} \in V, & \text{if } f \in S \\
         \dual{\shift[m] e_{\inv s(\inv \mu (f))}} \in \shift[m] \dual V, & \text{if } f \in T
       \end{cases}
  \]
  and $s \colon \finset{\card S} \to S$ is the unique order-preserving bijection.
  The sign $\epsilon$ is the sign incurred by permuting
  \[ \eTensor_{v \in N} q_v \tensor \eTensor_{s \in S} (h_s \tensor h_{\mu(s)})  \quad \text{into} \quad  \eTensor_{v \in N} \big( q_v \tensor \eTensor_{f \in \inv a(v)} h_f \big) \tensor \eTensor_{f \in I \cop J} h_f \]
  according to the Koszul sign rule.
\end{proposition}

\begin{proof}
  We have isomorphisms
  \begin{align*}
    \colim{(F, S, T, \mu) \in \Match{\id[I],\id[J]}} \big( (\bijmod P \compprod \bijmod Q)(F) \tensor (\shift[m] \basefield)^{\tensor S} \big)  &\iso  \colim{i \in \NNo} \coinv {\big( \Schur {(\bijmod P \compprod \bijmod Q)} {H_i} \tensor V_i^{\tensor I} \tensor W_i^{\tensor J} \big)} {\Gamma_i} \\
      &\iso  \colim{i \in \NNo} \coinv {\big( \Schur {\bijmod P} {\Schur {\bijmod Q} {H_i}} \tensor V_i^{\tensor I} \tensor W_i^{\tensor J} \big)} {\Gamma_i}
  \end{align*}
  by \Cref{lemma:coinv_Schur,lemma:compprod_and_Schur,lemma:Match_sgn_alt}.
  Moreover, by definition,
  \[ (\bijmod P \compprod \bijmod Q)(F) = \colim{(a \colon\! F \to N) \in F \comma \oBij} \big( \bijmod P(N) \tensor \bijmod Q(a) \big) \]
  so that the leftmost colimit above is isomorphic to
  \[ \colim{(F, S, T, \mu) \in \Match{\id[I],\id[J]}} \; \colim{(a \colon\! F \to N) \in F \comma \oBij} \big( \bijmod P(N) \tensor \bijmod Q(a) \tensor (\shift[m] \basefield)^{\tensor S} \big) \]
  and hence it is enough to show that the Grothendieck construction of the covariant functor $U \colon \Match{\id[I],\id[J]} \to \Cat$ given by $(F, S, T, \mu) \mapsto F \comma \oBij$ (where a map acts by precomposition with its inverse) is canonically isomorphic to $\DirGraph[{\id[I],\id[J]}]$.
  For this we define a functor $\gc U \to \DirGraph[{\id[I],\id[J]}]$, given, on objects respectively morphisms, by
  \begin{align*}
    \big( (F, S, T, \mu), a \colon F \to N \big) &\longmapsto \big( F, N, S, T, \mu, a \big) \\
    \big( (f \colon F \to F', \id[I], \id[J]), k \colon a \after \inv f \to a' \big) &\longmapsto \big( f, \pr(k) \big)
  \end{align*}
  where $\pr \colon F \comma \oBij \to \oBij$ is the projection.
  This is clearly an isomorphism.
  
  The explicit description of the $(\Symm I \times \Symm J)$-action and the explicit description of the isomorphism follow from the explicit descriptions in \Cref{lemma:coinv_Schur,lemma:compprod_and_Schur,lemma:Match_sgn_alt}.
\end{proof}

\subsection{Coinvariants of CE chains of convolution Lie algebras} \label{sec:coinv_conv}

In this subsection our goal is to compute the $(\Symm I \times \Symm J)$-module
\begin{equation} \label{eq:stable_CE_conv}
  \colim{i \in \NNo} \coinv {\Big( \CEchains * \big( \Schur {(\shift[-m] \operad C)} {H_i} \big) \tensor V_i^{\tensor I} \tensor W_i^{\tensor J} \Big)} {\Gamma_i}
\end{equation}
where $\operad C$ is a cyclic operad, and $H_i$, $V_i$, $W_i$, and $\Gamma_i$ are as in the preceding subsections.
(For ease of dealing with signs, we will only consider the case where $\operad C$ is concentrated in even degrees.)
The cocommutative coalgebra structure on the Chevalley--Eilenberg chains induces a cocommutative coalgebra structure on \eqref{eq:stable_CE_conv} when $I = J = \emptyset$, and the structure of a comodule over this coalgebra for general $I$ and $J$.
We will also identify this coalgebra (resp.\ comodule) structure explicitly.

By \cref{lemma:CEchains_as_Schur}, there is an isomorphism
\[ \CEchains * \big( \Schur {(\shift[-m] \operad C)} {H_i} \big)  \iso  \Schur {(\freegca \shift)} {\Schur {(\shift[-m] \operad C)} {H_i}} \]
of graded vector spaces.
Hence we can apply \cref{lemma:coinv_double_Schur} to obtain a description of the underlying graded vector space of \eqref{eq:stable_CE_conv}.
The rest of this subsection is concerned with also identifying its differential in these terms.
We begin with some useful auxiliary definitions.

\begin{definition}
  Let $I$ and $J$ be finite linearly ordered sets and $(F, N, S, T, \mu, a)$ a directed $(I,J)$-graph.
  For a vertex $v \in N$ we set:
  \begin{itemize}
    \item $\deg v \defeq \card{\inv a(v)}$ its \emph{degree}.
    \item $\outdeg v \defeq \card{\inv a(v) \intersect S}$ its \emph{out-degree}.
    \item $\indeg v \defeq \card{\inv a(v) \intersect T}$ its \emph{in-degree}.
    \item $\weight(v) \defeq n \outdeg{v} + (m - n) \indeg{v}$ its \emph{weight}.
  \end{itemize}
\end{definition}

\begin{definition} \label{def:edge_contraction}
  Let $I$ and $J$ be finite linearly ordered sets, $G = (F, N, S, T, \mu, a)$ a directed $(I,J)$-graph, and $s \in F \intersect S$ an internal outgoing flag.
  Moreover we set $v \defeq a(s)$ and $v' \defeq a(\mu(s))$ and assume that $v \neq v'$.
  In this situation, we denote by $\contract{s}(G)$ the directed $(I,J)$-graph $(F', N', S', T', \mu', a')$ with
  \begin{itemize}
    \item $F' \defeq F \setminus \set{s, \mu(s)}$ and $N' \defeq \quot N {(v \sim v')}$,
    \item $S' \defeq S \setminus \set s$ and $T' \defeq T \setminus \set{\mu (s)}$,
    \item $\mu' \defeq \restrict \mu {S'}$ and $a' \defeq {\pr} \after (\restrict a {F'})$, where $\pr \colon N \to N'$ is the quotient map.
  \end{itemize}
  Here $N'$ has the linear order such that $\pr(v) = \pr(v')$ is the smallest element and such that $\pr$ is order-preserving on all other elements.
  The set $F'$ is equipped with the linear order determined as follows: the elements of $\inv a({\set{v, v'}})$ are smaller than all other elements; these other elements have the order restricted from $F$; on $\inv a({\set{v, v'}})$ the order is given by
  \[ \inv a(v)_{< s} < \inv a(v')_{> \mu(s)} < \inv a(v')_{< \mu(s)} < \inv a(v)_{> s} \]
  where on each block the order is the one restricted from $F$.
  (This convention is chosen such that it interacts nicely with the composition operations of a cyclic operad of \cref{rem:cyclic_operad_comp}.)
\end{definition}

\begin{definition}
  Let $I$ and $J$ be finite linearly ordered sets and $G = (F, N, S, T, \mu, a)$ a directed $(I,J)$-graph.
  A \emph{neighbor-closed vertex set} of $G$ is a subset $N' \subseteq N$ such that $\mu(S \intersect \inv a(N')) = T \intersect \inv a(N')$.
  A \emph{closed connected component} of $G$ is a neighbor-closed vertex set that is inclusion minimal among non-empty neighbor-closed vertex sets.
  
  A neighbor-closed vertex set $N'$ of $G$ determines a directed $(\emptyset,\emptyset)$-graph and a directed $(I,J)$-graph
  \begin{align*}
    \induce[G](N')  &\defeq  \big(F', N', S', T', \restrict \mu {S'}, \restrict a {F'} \big) \in \DirGraph \\
    \coinduce[G](N')  &\defeq  \big(F \setminus F', N \setminus N', S \setminus S', T \setminus T', \restrict \mu {S \setminus S'}, \restrict a {F \setminus F'} \big) \in \DirGraph[{\id[I],\id[J]}]
  \end{align*}
  where $F' \defeq \inv a(N')$, $S' \defeq S \intersect \inv a(N')$, and $T' \defeq T \intersect \inv a(N')$.
\end{definition}

\begin{remark} \label{rem:neighbor-closed}
  Any neighbor-closed vertex set of $G$ is a disjoint union of a unique set of closed connected components of $G$.
  Moreover, when $I = J = \emptyset$ and $N' \subseteq N$ is a neighbor-closed vertex set of a directed graph $G = (F, N, S, T, \mu, a)$, then $N \setminus N'$ is one as well, and we have $\induce[G](N \setminus N') = \coinduce[G](N')$.
\end{remark}

\begin{definition} \label{def:dgc}
  Let $I$ and $J$ be linearly ordered sets.
  For a $\oBij$-module $\bijmod C$, we define the $(\Symm I \times \Symm J)$-module of \emph{$\bijmod C$-decorated directed $(I,J)$-graphs} to be
  \[ \DGC[I,J] m {\bijmod C}  \defeq  \colim{\DirGraph[{\id[I],\id[J]}]} \Dec m {\bijmod C; I, J} \]
  (when $I = J = \emptyset$, then we omit them from the notation).
  Here $\Dec m {\bijmod C; I, J}$ is the functor $\DirGraph[{\id[I],\id[J]}] \to \GrVect$ given, on a directed $(I,J)$-graph $G \defeq (F, N, S, T, \mu, a)$, by
  \[ \Dec m {\bijmod C; I,J} (G)  \defeq  (\shift \basefield)^{\tensor N} \tensor (\shift[-m] \bijmod C)(a) \tensor (\shift[m] \basefield)^{\tensor S} \]
  where a morphism $(f, k)$ of $\DirGraph[{\id[I],\id[J]}]$ acts on $(\shift[-m] \bijmod C)(a)$ by $(f, k)$ considered as a morphism of $\oBij \comma \oBij$, on $(\shift \basefield)^{\tensor N}$ by permuting the factors according to $k$, and on $(\shift[m] \basefield)^{\tensor S}$ by permuting the factors according to $\restrict {(f \cop {\id[I]} \cop {\id[J]})} S$.
  The action of an element $(g, h) \in \Symm I \times \Symm J$ on $\DGC[I,J] m {\bijmod C}$ is given by sending an element represented by $(F, N, S, T, \mu, a) \in \DirGraph[{\id[I],\id[J]}]$ and $x \in \Dec m {\bijmod C; I, J}$ to the element represented by $(F, N, S, T, \restrict c T \after \mu \after \inv {(\restrict c S)}, a)$ and $\sgn(g)^m x$, where $c \defeq {\id[F]} \cop g \cop h$.
  
  Given a family $\big( \xi_v \in \bijmod C(\inv a(v)) \big)_{v \in N}$ of elements of $\bijmod C$, we call
  \[ \shift[\card N] 1 \tensor \eTensor_{v \in N} \shift[-m] \xi_v \tensor \shift[m \card S] 1 \in \Dec m {\bijmod C} (G) \]
  the element \emph{represented} by ($G$ and) the decorations $\xi_v$.
  
  If $\bijmod C$ is concentrated in even degrees and equipped with the structure of a cyclic operad, then $\DGC[I,J] m {\bijmod C}$ becomes a chain complex, which we will call the \emph{directed $(I,J)$-graph complex} associated to $\bijmod C$.
  We now describe its differential $d$.
  For every directed $(I,J)$-graph $G \defeq (F, N, S, T, \mu, a)$, every $y \in \Dec m {\bijmod C; I, J} (G)$, and every internal outgoing flag $s \in S \intersect F$ (or, equivalently, every internal edge $e \defeq (s, \mu(s))$) we define an element $d^G_s(y) \in \DGC[I,J] m {\bijmod C}$ by specifying the following properties:
  \begin{itemize}
    \item For every isomorphism $\chi \colon G \to G'$ of directed $(I,J)$-graphs holds that $d^G_s(y) = d^{G'}_{\chi(s)}(\chi \act y)$.
    \item If $a(s) = a(\mu(s))$, i.e.\ if $e$ is a loop, then $d^G_s(y) = 0$.
    \item Assume that $v \defeq a(s)$ and $v' \defeq a(\mu(s))$ are, in this order, the first two elements of $N$.
    Moreover assume that $a$ is order-preserving, and that $s$ is the last element of $\inv a(v)$ and that $\mu(s)$ is the first element of $\inv a(v')$.
    Then $d^G_s$ applied to the element represented by some decorations $\xi_v$ is $(-1)^{m (\outdeg{v} - 1)}$ times the element represented by the directed $(I,J)$-graph $\contract{s}(G)$ decorated by $\xi_v \opcomp[s]{\mu(s)} \xi_{v'}$ at the collapsed vertex $\pr(v) = \pr(v')$ and by $\xi_w$ at all other vertices $w$.
  \end{itemize}
  For $x \in \DGC[I,J] m {\bijmod C}$ the element represented by $G$ and $y$, we set $d(x) \defeq \sum_{s \in S \intersect F} d^G_s(y)$.
  
  When $I = J = \emptyset$, we equip $\DGC m {\bijmod C}$ with the structure of a cocommutative differential graded coalgebra.
  The counit $\epsilon \colon \DGC m {\bijmod C} \to \basefield$ sends the element represented by the empty graph $\emptyset$ and $1 \in \QQ = \Dec m {\bijmod C} (\emptyset)$ to $1$ and an element represented by any other graph to $0$.
  The comultiplication $\Delta \colon \DGC m {\bijmod C} \to \DGC m {\bijmod C} \tensor \DGC m {\bijmod C}$ is given as follows: for an element $x \in \DGC m {\bijmod C}$ represented by $G$ and $(\xi_v)_{v \in N}$, we set
  \[ \Delta(x)  \defeq  \sum_{N'} (-1)^{m \card {N'} \card{N \setminus N'}} \epsilon(N') x_{N'} \tensor x_{N \setminus N'} \]
  where the sum runs over all neighbor-closed vertex subsets $N'$ of $G$ and $x_A$ is the element of $\DGC m {\bijmod C}$ represented by $\induce[G](A)$ and $(\xi_v)_{v \in A}$.
  The sign $\epsilon(N') \in \set {\pm 1}$ is the sign incurred by permuting
  \begin{gather*}
    (\shift 1)^{\tensor N} \tensor \Tensor_{v \in N} \shift[-m] \xi_v \tensor (\shift[m] 1)^{\tensor S} \\
    \shortintertext{into}
    (\shift 1)^{\tensor N'} \tensor \Tensor_{v \in N'} \shift[-m] \xi_v \tensor (\shift[m] 1)^{\tensor S'} \tensor (\shift 1)^{\tensor N \setminus N'} \tensor \Tensor_{v \in N \setminus N'} \shift[-m] \xi_v \tensor (\shift[m] 1)^{\tensor S \setminus S'}
  \end{gather*}
  where $S' \defeq S \intersect \inv a(N')$.
  
  For general $I$ and $J$, we equip $\DGC[I,J] m {\bijmod C}$ with the structure of a differential graded left comodule over $\DGC m {\bijmod C}$.
  Its structure map $\rho \colon \DGC[I,J] m {\bijmod C} \to \DGC m {\bijmod C} \tensor \DGC[I,J] m {\bijmod C}$ is given as follows: for an element $x \in \DGC[I,J] m {\bijmod C}$ represented by $G$ and $(\xi_v)_{v \in N}$, we set
  \[ \rho(x)  \defeq  \sum_{N'} (-1)^{m \card {N'} \card{N \setminus N'}} \epsilon(N') x_{N'} \tensor x'_{N'} \]
  where the sum runs over all neighbor-closed vertex subsets $N'$ of $G$, $x_{N'}$ is the element of $\DGC m {\bijmod C}$ represented by $\induce[G](N')$ and $(\xi_v)_{v \in N'}$, and $x'_{N'}$ is the element of $\DGC[I,J] m {\bijmod C}$ represented by $\coinduce[G](N')$ and $(\xi_v)_{v \in N \setminus N'}$.
  The sign $\epsilon(N') \in \set {\pm 1}$ is the same as above.
\end{definition}

\begin{remark}
  It will follow from \Cref{thm:graph_complex_identification} that the differential $d$ of $\DGC m {\bijmod C}_{I,J}$ is well-defined and actually a differential, that, in the case $I = J = \emptyset$, the maps $\epsilon$ and $\Delta$ actually equip it with the structure of a cocommutative differential graded coalgebra, and, for general $I$ and $J$, that $\rho$ actually equips it with the structure of a differential graded comodule.
  
  Note that, for defining the counit and comultiplication (respectively the comodule structure map), it is not necessary for $\operad C$ to be a cyclic operad.
  However, our proof below that they yield the structure of a cocommutative graded coalgebra (respectively comodule) only applies in that case; it is certainly also true without this assumption, though.
  The condition that $\operad C$ is concentrated in even degrees is purely for convenience; it avoids having to deal with even more signs.
\end{remark}

\begin{remark}
  We think of elements of $\DGC m {\bijmod C}_{I,J}$ as being represented by a directed $(I,J)$-graph $(F, N, S, T, \mu, a) \in \DirGraph[{\id[I],\id[J]}]$ together with a decoration $\xi_v \in \bijmod C(\inv a(v))$ for every vertex $v \in N$.
  To understand the action of an $(I,J)$-graph isomorphism, we should additionally think of such an element as being equipped with an orientation on the set of vertices if $m$ is even or an orientation on the set of edges if $m$ is odd; when an $(I,J)$-graph isomorphism flips the orientation of the respective set, it induces an extra sign.
\end{remark}

\begin{remark}
  The sign $(-1)^{m (\outdeg{v} - 1)}$ in the definition of the differential above corresponds to moving the to-be-contracted outgoing flag $s$ to the beginning of $S$.
\end{remark}

\begin{remark}
  In the case of the cyclic commutative operad $\operad C = \Com$, the directed graph complex $\DGC m {\bijmod C}$ (or versions of it) has appeared in the work of, for example, Willwacher \cite[Appendix~K]{Wil}, Živković \cite{Ziv}, and Dolgushev--Rogers \cite{DR}.
\end{remark}

We are now ready to state and prove the main results of this subsection.
(See \cite[Theorem~9.14]{BM} for a similar statement in the case where $H$ is concentrated in a single degree.)

\begin{theorem} \label{thm:graph_complex_identification}
  Let $I$ and $J$ be finite linearly ordered sets, and let $\operad C$ be a cyclic operad concentrated in even degrees.
  Then there is an isomorphism of differential graded $(\Symm I \times \Symm J)$-modules
  \[ \Phi_{I,J} \colon \DGC m {\operad C}_{I,J}  \xlongto{\iso}  \colim{i \in \NNo} \coinv {\Big( \CEchains * \big( \Schur {(\shift[-m] \operad C)} {H_i} \big) \tensor V_i^{\tensor I} \tensor W_i^{\tensor J} \Big) } {\Gamma_i} \]
  where $\Schur {(\shift[-m] \operad C)} {H_i}$ is equipped with the graded Lie algebra structure of \Cref{lemma:Schur_cyclic_Lie}.
  When $I = J = \emptyset$, this is an isomorphism of differential graded coalgebras, and, for general $I$ and $J$, it is compatible with the comodule structures over these algebras.
\end{theorem}

\begin{remark} \label{rem:directed_gc_unstable}
  Tracing the dimension restriction of \cref{prop:coinv_monomial} through our arguments, the preceding theorem can be strengthened to an unstable isomorphism in a certain stable range that is compatible with the stabilization maps.
  More precisely we obtain an isomorphism from the differential graded subcoalgebra (resp.\ subcomodule) of $\DGC m {\operad C}_{I,J}$ spanned by the decorated graphs with at most $\dim H_i - \card I - \card J$ internal flags to the differential graded subcoalgebra (resp.\ subcomodule) of $\coinv {\CEchains * \big( \Schur {(\shift[-m] \operad C)} {H_i} \tensor V_i^{\tensor I} \tensor W_i^{\tensor J} \big) } {\Gamma_i}$ spanned by the elementary tensors involving at most $\dim H_i$ elements of $H_i$ (including $V_i$ and $W_i$).
\end{remark}

\begin{proof}[Proof of \cref{thm:graph_complex_identification}]
  \Cref{lemma:CEchains_as_Schur,lemma:coinv_double_Schur} provide us with an isomorphism $\Phi'_{I,J}$ on the underlying $(\Symm I \times \Symm J)$-modules (in graded vector spaces).
  Let $G = (F, N, S, T, \mu, a) \in \DirGraph[{\id[I],\id[J]}]$ be a directed $(I,J)$-graph and
  \begin{equation} \label{eq:pf_identification_element}
    x \defeq (\shift 1)^{\tensor N} \tensor \eTensor_{v \in N} \shift[-m] \xi_v \tensor \shift[m \card S] 1  \quad\in\quad  (\shift \basefield)^{\tensor N} \tensor (\shift[-m] \operad C)(a) \tensor (\shift[m] \basefield)^{\tensor S}
  \end{equation}
  an element.
  Denoting by $\eqcl x$ the element represented by $G$ and $x$, we set
  \begin{align*}
    \Phi_{I,J}(\eqcl x) &\defeq \gamma(G) \Phi'_{I,J}(\eqcl x) \\
    \gamma(G) &\defeq (-1)^{m \frac{1}{2} \card N (\card N + 1) + n \card N}
  \end{align*}
  (the sign $\gamma(G)$ serves to cancel out a sign encountered below).
  We have $\Phi_{I,J}(\eqcl x) = \eqcl {\Phi^G_{I,J}(x)}$ where we set $\Phi^G_{I,J}(x) \defeq \gamma(G) \epsilon(x) y$ with $y$ given by
  \begin{equation} \label{eq:pf_identification_image}
    \begin{aligned}
      y &\defeq \tilde y \tensor \eTensor_{f \in I \cop J} h_f \\
      \tilde y &\defeq \Wedge_{v \in N} \shift y_v \\
      y_v &\defeq \shift[-m] \xi_v \tensor \eTensor_{f \in \inv a(v)} h_f \\
      h_f &\defeq
        \begin{cases}
          e_{\inv s(f)} \in V, & \text{if } f \in S \\
          \dual{\shift[m] e_{\inv s(\inv \mu (f))}} \in \shift[m] \dual V, & \text{if } f \in T
        \end{cases}
    \end{aligned}
  \end{equation}
  (where $s \colon \finset{\card S} \to S$ is the unique order-preserving bijection) and the sign $\epsilon(x) \in \set{\pm 1}$ is the sign incurred by permuting
  \[ (\shift 1)^{\tensor N} \tensor \Tensor_{v \in N} \shift[-m] \xi_v \tensor \Tensor_{s \in S} (h_s \tensor h_{\mu(s)})  \quad \text{into} \quad  \Tensor_{v \in N} \shift \bigg( \shift[-m] \xi_v \tensor \Tensor_{f \in \inv a(v)} h_f \bigg) \tensor \Tensor_{f \in I \cop J} h_f \]
  according to the Koszul sign rule.
  
  We will now identify the result of pulling the Chevalley--Eilenberg differential $d = d_0 + d_1$ (cf.\ \Cref{not:CEchains}) back to $\DGC m {\operad C}_{I,J}$ along the isomorphism $\Phi_{I,J}$.
  To this end, we first evaluate it on $\tilde y$.
  We have $d_0 = 0$ (since the differential of $\Schur {(\shift[-m] \operad C)} {H_i}$ is trivial) and
  \begin{align}
    d_1(\tilde y) &= \sum_{\substack{v, v' \in N \\ v < v'}} (-1)^{b_{v,v'}} \shift \liebr{y_v}{y_{v'}} \wedge \Wedge_{w \in V \setminus \set{v, v'}} \shift y_w \nonumber\\
    \liebr{y_v}{y_{v'}} &= \sum_{\substack{i \in \inv a(v) \\ j \in \inv a(v')}} \shift[-m] (\xi_v \opcomp[i]{j} \xi_{v'}) \tensor (h_v \opcomp[i]{j} h_{v'}) \label{eq:lie_br_ys}\\
    b_{v,v'} &\defeq \deg{\shift y_v} (1 + \deg{\shift y_{v'}}) + \deg{\shift y_v} \sum_{\substack{w \in N \\ w < v}} \deg{\shift y_w} + \deg{\shift y_{v'}} \sum_{\substack{w \in N \\ w < v'}} \deg{\shift y_w} \nonumber
  \end{align}
  where $h_v \defeq \eTensor_{f \in \inv a(v)} h_f$.
  In the following, we will, for ease of notation, identify each $\inv a(v)$ with $\set{1, \dots, \deg v}$ via the unique order-preserving bijection, and set $h_{v;i} \defeq h_{f(v,i)}$ where $f(v,i)$ is the $i$-th element of $\inv a(v)$.
  We recall
  \begin{align*}
    h_v \opcomp[i]{j} h_{v'} &\defeq {(-1)}^{a^{i,j}_{v,v'}} h_{v;1,i-1} \tensor h_{v';j+1,\deg{v'}} \tensor h_{v';1,j-1} \tensor h_{v;i+1,\deg v} \iprod {h_{v';j}} {h_{v;i}} \\
    a^{i,j}_{v,v'} &\defeq (\deg {h_{v;i,\deg v}} - m) (\deg {h_{v'}} - m) + \deg {h_{v';1,j}} \deg {h_{v';j+1,\deg{v'}}} + 1 \\
    h_{w;k,l} &\defeq \eTensor_{i = k}^l h_{w;i}
  \end{align*}
  and note that, by definition of $h_f$, the expression $\iprod {h_{v';j}} {h_{v;i}}$ is zero unless $\set{i, j} = \set{s, \mu(s)}$ for some $s \in S \intersect F$.
  Hence the same is true for the summand in \eqref{eq:lie_br_ys} corresponding to $(i, j)$.
  We thus obtain the following description
  \begin{align*}
    d(y)  &=  \sum_{\substack{s \in S \intersect F \\ a(s) \neq a(\mu(s))}} d^G_s(y) \\
    d^G_s(y)  &\defeq  \delta_s \mult \big( \shift (y_{a(s)} \opcomp[s]{\mu(s)} y_{a(\mu(s))}) \wedge \Wedge_{v \in V \setminus \set{a(s), a(\mu(s))}} \shift y_v \big) \tensor \eTensor_{f \in I \cop J} h_f \\
    y_v \opcomp[i]{j} y_{v'}  &\defeq  \shift[-m] (\xi_v \opcomp[i]{j} \xi_{v'}) \tensor (h_{v} \opcomp[i]{j} h_{v'})
  \end{align*}
  where $\delta_s \in \set{\pm 1}$ is some sign.
  We note that $d^G_s(\Phi^G_{I,J}(x)) \defeq \gamma(G) \epsilon(x) d^G_s(y)$ is equal to some sign $\delta'_s \in \set{\pm 1}$ times the image under $\Phi_{I,J}$ of the element $\eqcl z$ represented by the directed $(I,J)$-graph $\contract{s}(G)$ decorated by $\xi_{a(s)} \opcomp[s]{\mu(s)} \xi_{a(\mu(s))}$ at the collapsed vertex $\pr(a(s)) = \pr(a(\mu(s)))$ and by $\xi_v$ at all other vertices $v$.
  (Note that the family of basis elements $h_\bullet$ obtained for $\contract{s}(G)$ does not agree with the family obtained for the non-contracted flags of $G$. However they are equivalent under the $\Gamma_i$-action.)
  
  We will now determine the sign $\delta'_s$ in a special case.
  To this end, let $s \in S \intersect F$ and $t \defeq \mu(s)$ such that $a(s) \neq a(t)$.
  Assume that $v \defeq a(s)$ and $v' \defeq a(t)$ are, in this order, the first two elements of $N$.
  Moreover assume that $a \colon F \to N$ is order-preserving and that $s$ is the last element of $\inv a(v)$ and that $t$ is the first element of $\inv a(v')$.
  (Note that any directed $(I,J)$-graph is isomorphic to one of this form.)
  We have
  \begin{align*}
    \delta'_s  &=  \gamma(G) \epsilon(x) \iprod {h_{v';t}} {h_{v;s}} (-1)^c \gamma(\contract{s}(G)) \epsilon(z) \\
    c  &\defeq  b_{v,v'}(x) + a_{v,v'}^{s,t}(x)
  \end{align*}
  by definition.
  We first note that $\iprod {h_{v';t}} {h_{v;s}} = 1$, that $\gamma(G) \gamma(\contract{s}(G)) = (-1)^{m \card N + n}$, and that
  \begin{align*}
    c &= \deg{\shift y_v} + \big( \deg{h_{v;s}} - m \big) \big( \deg{h_{v'}} - m \big) + \deg{h_{v';t}} \deg{h_{v';t+1,\deg{v'}}} + 1 \\
      &= \big( 1 - m + \weight(v) \big) + (n - m) \big( {\weight(v')} - m \big) + (m - n) \big( {\weight(v')} - (m - n) \big) + 1
  \end{align*}
  which reduces the question to determining $\epsilon(x) \epsilon(z)$.
  To this end we factor the permutation defining $\epsilon$ into three parts: firstly permuting the $h_f$ among each other, secondly interleaving the $\shift[-m] \xi_v$ and $h_v$, and thirdly interleaving the $\shift 1$ and $y_v$.
  The deviation between $\epsilon(x)$ and $\epsilon(z)$ for the first part is given by $-1$ to the power of
  \[ m \big( \outdeg{v} - 1 \big) + m \big( {\weight(v)} - n \big) \]
  for the second part it is $-1$ to the power of
  \[ \big( \deg{h_s} + \deg{h_t} \big) \big( \card N - 2 \big) (-m) + \deg{h_v} (-m) \equiv \card N m + \weight(v) m  \pmod 2 \]
  and for the the third part it is $-1$ to the power of
  \[ \big( {- m} + \deg{h_s} + \deg{h_t} \big) \card N + \deg{y_v} = -m + \weight(v) \]
  so that, putting everything together, we obtain $\delta'_s = (-1)^{m (\outdeg{v} - 1)}$.
  
  Lastly we show that, for any isomorphism of directed $(I,J)$-graphs $\chi \colon G \to G'$, we have
  \[ \eqcl[\big]{ d^G_s(\Phi^G_{I,J}(x)) }  =  \eqcl[\big]{ d^{G'}_{\chi(s)}(\Phi^{G'}_{I,J}(\chi \act x)) } \]
  which will finish the identification of the differential.
  The statement is clear up to sign, so it is enough to show that those agree as well.
  We write $\chi \act x = \alpha x'$ where $\alpha \in \set{\pm 1}$ and $x'$ is the element represented by $G'$ and the labels $(\chi \act \xi_{\inv \chi(w)})_{w \in N'}$.
  Lastly we define $y'$ analogously to \eqref{eq:pf_identification_image}, so that $\Phi^{G'}_{I,J}(x') = \gamma(G') \epsilon(x') y'$.
  
  By definition, we have $\eqcl {\gamma(G) \epsilon(x) y} = \Phi_{I,J}(\eqcl x) = \Phi_{I,J}(\eqcl {\chi \act x}) = \eqcl {\alpha \gamma(G') \epsilon(x') y'}$.
  Noting that $\gamma(G) = \gamma(G')$, it follows from the definitions that $\alpha \gamma(G) \epsilon(x) \gamma(G') \epsilon(x') = \alpha \epsilon(x) \epsilon(x')$ is the sign induced by permuting
  \[ \Tensor_{v \in N} \shift \bigg( \shift[-m] \xi_v \tensor \Tensor_{f \in \inv a(v)} h_f \bigg)  \quad \text{into} \quad  \Tensor_{w \in N'} \shift \bigg( \shift[-m] \xi_{\inv \chi(w)} \tensor \Tensor_{g \in \inv{a'}(w)} h_{\inv \chi(g)} \bigg)  \]
  (here we use that the families $(h_f)_{f \in F}$ and $(h'_{\chi(f)})_{f \in F}$ are equivalent under the action of $\Gamma_i$ for large enough $i$; hence we can assume that they are equal, which we will do from now on).
  This sign is the product of the sign $\alpha'$ incurred by permuting $\Tensor_{v \in N} \shift y_v$ into $\Tensor_{w \in N'} \shift y_{\inv \chi(w)}$ and the signs $\alpha_v$ incurred by permuting $\Tensor_{f \in \inv a(v)} h_f$ into $\Tensor_{g \in \inv{a'}(\chi(v))} h_{\inv \chi(g)}$.
  
  It remains to evaluate the sign difference between $d^G_s(y)$ and $d^{G'}_{\chi(s)}(y')$.
  For this, write $t \defeq \mu(s)$, and $\set {v, v'} \defeq \set{a(s), a(t)}$ such that $v < v'$.
  Also let $\set {w, w'} \defeq \set {\chi(v), \chi(v')}$ such that $w < w'$.
  Then the relevant summands of $d_1(\tilde y)$ and $d_1(\tilde y')$ are the ones indexed by $(v, v')$ and $(w, w')$.
  Note that $-1$ to the power of $b_{v,v'} + b'_{w,w'}$, the sign difference between $\Wedge_{\tilde v \in V \setminus \set{v, v'}} \shift y_{\tilde v}$ and $\Wedge_{\tilde w \in V' \setminus \set{w, w'}} \shift y'_{\tilde w}$, and $\alpha'$ and $\alpha_{\tilde v}$ for $\tilde v \in V \setminus \set{v, v'}$ cancel each other out, except for, if $\chi$ swaps the order of $v$ and $v'$, a sign of $-1$ to the power of $1 + \deg{y'_w} \deg{y'_{w'}}$.
  Now note that, in the case that $\chi$ swaps the order of $v$ and $v'$, we have $\liebr{y'_{w'}}{y'_w} = (-1)^{1 + \deg{y'_w} \deg{y'_{w'}}} \liebr{y'_w}{y'_{w'}}$ in the summand of $d_1(\tilde y')$.
  Hence the only sign that has not yet canceled is $\alpha_v \alpha_{v'}$.
  
  Thus, it remains to show that the summand of $\liebr{y_v}{y_{v'}}$ corresponding to $(s, t)$ equals the summand of $\liebr{y'_{\chi(v)}}{y'_{\chi(v')}}$ corresponding to $(\chi(s), \chi(t))$, up to a sign of $\alpha_v \alpha_{v'}$.
  (For ease of notation we assume from now on that $a(s) = v$, and we set $w \defeq \chi(v)$, $w' \defeq \chi(v')$, $s' \defeq \chi(s)$, and $t' \defeq \chi(t)$, possibly changing notation from before.)
  These two summands agree in $\Schur {(\shift[-m] \operad C)} {H_i}$ up to sign.
  The difference in signs is given by $-1$ to the power of $a_{v,v'}^{s,t} + a_{w,w'}^{s',t'}$ times the sign incurred by permuting
  \[ h_v \tildeopcomp[s]{t} h_{v'}  \defeq  h_{v;1,s-1} \tensor h_{v';t+1,\deg{v'}} \tensor h_{v';1,t-1} \tensor h_{v;s+1,\deg v} \]
  into $h'_w \tildeopcomp[s']{t'} h'_{w'}$.
  Now we note that $-1$ to the power of $a_{v,v'}^{s,t} + 1 + m + m (\deg {h_{v;s}} + \deg {h_{v'}})$ is precisely the sign incurred by permuting $h_v \tensor h_{v'}$ into $(h_v \tildeopcomp[s]{t} h_{v'}) \tensor h_{v';t} \tensor h_{v;s}$ and analogously for $a_{w,w'}^{s',t'}$.
  Noting that $\deg {h_{v;s}} = \deg {h'_{w;s'}}$ and $\deg {h_{v'}} = \deg {h'_{w'}}$, this implies the claim since $\alpha_v \alpha_{v'}$ is the sign incurred by permuting $h_v \tensor h_{v'}$ into $h'_w \tensor h'_{w'}$.
  
  Now we prove that $\Phi_{\emptyset,\emptyset}$ is a map of graded coalgebras, and that $\Phi_{I,J}$ is compatible with the comodule structures over these coalgebras.
  It is clear that $\Phi_{\emptyset,\emptyset}$ is compatible with the counits.
  We will now identify the comodule structure map $\rho$ of $\DGC m {\operad C}_{I,J}$ obtained by pulling back the one of the Chevalley--Eilenberg chains.
  When $I = J = \emptyset$, this also deals with the comultiplication $\Delta$ of the coalgebra structure since, in this case, we have $\Delta = \rho$ on both sides (see \cref{rem:neighbor-closed}).
  We continue to use the notation of \eqref{eq:pf_identification_element} and \eqref{eq:pf_identification_image}.
  We have
  \begin{equation} \label{eq:pf_identification_CE_coproduct}
    \rho(y)  \defeq  \sum_{N' \subseteq N} \delta(N') \mult \big( {\Wedge_{v \in N'} \shift y_v} \big) \tensor \big( (\Wedge_{w \in N \setminus N'} \shift y_w) \tensor \eTensor_{f \in I \cop J} h_f \big)
  \end{equation}
  where the sum runs over all subsets $N'$ of $N$ and $\delta(N') \in \set{\pm 1}$ is the sign incurred by permuting
  $\eTensor_{v \in N} \shift y_v$ into $\eTensor_{v \in N'} \shift y_v \tensor \eTensor_{w \in N \setminus N'} \shift y_w$.
  Note that (the equivalence class of) $\Wedge_{v \in N'} \shift y_v$ is trivial except when $\mu(S \intersect \inv a(N')) = T \intersect \inv a(N')$.
  Hence we can also consider the sum in \eqref{eq:pf_identification_CE_coproduct} to run only over all neighbor-closed vertex sets $N'$ of $G$.
  For such an $N'$, we set
  \begin{align*}
    x_{N'}  &\defeq  (\shift 1)^{\tensor N'} \tensor \eTensor_{v \in N'} \shift[-m] \xi_v \tensor \shift[m \card {S'}] 1 \\
    x'_{N'}  &\defeq  (\shift 1)^{\tensor N \setminus N'} \tensor \eTensor_{v \in N \setminus N'} \shift[-m] \xi_v \tensor \shift[m \card {S \setminus S'}] 1
  \end{align*}
  where $S' \defeq S \intersect \inv a(N')$.
  We denote the element of $\DGC m {\operad C}$ represented by $\induce[G](N')$ and $x_{N'}$ by $\eqcl {x_{N'}}$, and the element of $\DGC m {\operad C}_{I,J}$ represented by $\coinduce[G](N')$ and $x'_{N'}$ by $\eqcl {x'_{N'}}$.
  We have
  \begin{align*}
    \Phi_{\emptyset,\emptyset}(\eqcl{x_{N'}}) &= \gamma \big( {\induce[G](N')} \big) \epsilon(x_{N'}) \mult \Wedge_{v \in {N'}} \shift y_v \\
    \Phi_{I,J}(\eqcl{x'_{N'}}) &= \gamma \big( {\coinduce[G](N')} \big) \epsilon(x'_{N'}) \mult \big( { \Wedge_{v \in {N \setminus N'}} \shift y_v } \tensor \eTensor_{f \in I \cop J} h_f \big)
  \end{align*}
  so that it only remains to see that
  \[ \gamma(G) \epsilon(x) \delta(N') \gamma \big( {\induce[G](N')} \big) \epsilon(x_{N'}) \gamma \big( {\coinduce[G](N')} \big) \epsilon(x'_{N'}) \]
  is equal to the sign in the definition of the comodule structure of $\DGC m {\operad C}_{I,J}$.
  Noting that
  \[ \gamma(G) \gamma \big( {\induce[G](N')} \big) \gamma \big( {\coinduce[G](N')} \big) = (-1)^{m \card {N'} \card{N \setminus N'}} \]
  this is clear from the definitions.
\end{proof}

\subsection{Coinvariants of the truncation}

In this subsection our goal is to obtain a description of the differential graded $(\Symm I \times \Symm J)$-module
\[ \colim{i \in \NNo} \coinv {\Big( \CEchains * \big( \trunc{ \Schur {(\shift[-m] \operad C)} {H_i} } \big) \tensor V_i^{\tensor I} \tensor W_i^{\tensor J} \Big) } {\Gamma_i} \]
in terms of a graph complex (including the coalgebra/comodule structures, as before).
Here $\operad C$ is a cyclic operad, $H_i$, $V_i$, $W_i$, and $\Gamma_i$ are as in the preceding subsections, and $\trunc {(\blank)}$ denotes positive truncation (see \cref{def:pos_trunc}).
We begin with the following useful definition.

\begin{definition} \label{def:trunc_dgc}
  Let $I$ and $J$ be finite linearly ordered sets.
  A directed $(I,J)$-graph $G$ is \emph{truncated} if it has no vertices of valence zero or one, and each vertex of valence two has only incoming edges.
  We denote by
  \[ \DirGraphtrunc[{\id[I],\id[J]}] \subset \DirGraph[{\id[I],\id[J]}] \]
  the full subgroupoid spanned by the truncated directed $(I,J)$-graphs.
  
  We also denote by $\DGCtrunc m {\operad C}_{I,J} \subseteq \DGC m {\operad C}_{I,J}$ the differential graded $(\Symm I \times \Symm J)$-submodule spanned by the elements represented by truncated decorated $(I,J)$-graphs.
  When $I = J = \emptyset$, we omit them from the notation.
\end{definition}

\begin{remark}
  When $I = J = \emptyset$, we have that $\DGCtrunc m {\operad C} \subseteq \DGC m {\operad C}$ is a subcoalgebra.
  For general $I$ and $J$, the $\DGC m {\operad C}$-comodule structure of $\DGC m {\operad C}_{I,J}$ restricts to a $\DGCtrunc m {\operad C}$-comodule structure on $\DGCtrunc m {\operad C}_{I,J}$.
\end{remark}

\begin{corollary} \label{lemma:trunc_dgc}
  Let $I$ and $J$ be finite linearly ordered sets, and let $\operad C$ be a cyclic operad concentrated in non-negative even degrees such that $\cyc {\operad C} 0$, $\cyc {\operad C} 1$, and $\cyc {\operad C} 2$ are concentrated in degree $0$.
  Assume that $2n < m < 3n$ and set
  \[ \alpha  \defeq  \begin{cases} \frac 2 3, & \text{if $m = 3n - 1$.} \\ 1, & \text{otherwise.} \end{cases} \]
  Then, for all $i \in \NNge{1}$, there is, in homological degrees $\le \alpha \dim H_i$, an isomorphism of differential graded $(\Symm I \times \Symm J)$-modules
  \begin{equation} \label{eq:unstable_iso_trunc}
    \DGCtrunc m {\operad C}_{I,J}  \iso  \coinv {\Big( \CEchains * \big( \trunc {\Schur {(\shift[-m] \operad C)} {H_i}} \big) \tensor V_i^{\tensor I} \tensor W_i^{\tensor J} \Big)} {\Gamma_i}
  \end{equation}
  that is compatible with the stabilization maps on the right hand side.
  When $I = J = \emptyset$, it is an isomorphism of differential graded coalgebras.
  For general $I$ and $J$, it is compatible with the comodule structures over these coalgebras.
  
  In particular, after stabilizing, we obtain an isomorphism
  \[ \DGCtrunc m {\operad C}_{I,J}  \iso  \colim{i \in \NNo} \coinv {\Big( \CEchains * \big( \trunc {\Schur {(\shift[-m] \operad C)} {H_i}} \big) \tensor V_i^{\tensor I} \tensor W_i^{\tensor J} \Big)} {\Gamma_i} \]
  of differential graded $(\Symm I \times \Symm J)$-modules.
  When $I = J = \emptyset$, it is an isomorphism of differential graded coalgebras.
  For general $I$ and $J$, it is compatible with the comodule structures over these coalgebras.
\end{corollary}

\begin{proof}
  The condition $2n < m < 3n$ implies that
  \[ \trunc {\Schur {(\shift[-m] \operad C)} {H_i}}  =  \big( \shift[-m] \cyc {\operad C} 2 \tensor[\Symm{2}] (W_i \tensor W_i) \big) \dirsum \Dirsum_{k \ge 3} \shift[-m] \cyc {\operad C} k \tensor[\Symm{k}] H_i^{\tensor k}  \subseteq  \Schur {(\shift[-m] \operad C)} {H_i} \]
  where $W_i$ is the degree $m - n$ part of $H_i$.
  Then the explicit description of the isomorphism of \cref{thm:graph_complex_identification}, together with \cref{rem:directed_gc_unstable}, implies that there is an isomorphism from the differential graded $(\Symm I \times \Symm J)$-submodule $D'_*$ of $\DGCtrunc m {\operad C}_{I,J}$ spanned by the decorated graphs with at most $\dim H_i - \card I - \card J$ internal flags to the differential graded $(\Symm I \times \Symm J)$-submodule $C'_*$ of the right hand side of \eqref{eq:unstable_iso_trunc} spanned by the elementary tensors involving at most $\dim H_i$ elements of $H_i$ (including $V_i$ and $W_i$).
  
  We will now prove that $C'_p$ is equal to degree $p$ of the right hand side of \eqref{eq:unstable_iso_trunc} when $p \le \alpha \dim H_i$.
  To this end, we take elements
  \[
  \begin{array}{ll}
    y \defeq  x \tensor v \tensor w  &\in  \CEchains * \big( \trunc {\Schur {(\shift[-m] \operad C)} {H_i}} \big) \tensor V_i^{\tensor I} \tensor W_i^{\tensor J} \\
    x \defeq  \Wedge_{v \in N} \shift \big( \shift[-m] \xi_v \tensor \eTensor_{j = 1}^{k_v} h_{v,j} \big)  &\in  \CEchains * \big( \trunc {\Schur {(\shift[-m] \operad C)} {H_i}} \big)
  \end{array}
  \]
  and note that $\deg x = \sum_{v \in N} \big( 1 - m + \deg{\xi_v} + \weight(v) \big)$ where $\weight(v) \defeq \sum_{j = 1}^{k_v} \deg{h_{v,j}}$.
  Due to the truncation we have that
  \[ \weight(v) \begin{cases} = 2 (m - n), & \text{if } k_v = 2 \\ \ge k_v n, & \text{if } k_v \ge 3 \end{cases} \]
  and that $k_v \ge 2$.
  Setting $N_2 \defeq \set{v \in N \mid k_v = 2}$ and similarly for $N_{\ge 3}$, we thus obtain
  \[ \deg x  \ge  \card{N_2} (1 - m + 2 (m-n)) + \sum_{v \in N_{\ge 3}} (1 - m + k_v n) \]
  since $\deg{\xi_v} \ge 0$ by assumption.
  We note $1 - m + 2(m - n) \ge 2$ and, when $k_v \ge 3$, that
  \[ \frac{3}{1 - m + 3n} (1 - m + k_v n) \ge k_v \]
  since $\frac{3n}{1 - m + 3n} \ge 1$ as $-3n < 1 - m < 0$.
  Setting
  \[ \beta  \defeq  {\max} \Big( 1, \frac{3}{1 - m + 3n} \Big) = \inv \alpha \]
  we thus obtain
  \[ \beta \deg x  \ge  \card {N_2} 2 + \sum_{v \in N_{\ge 3}} k_v  =  \sum_{v \in N} k_v \]
  and hence that $\beta \deg x + \card I + \card J$ is larger or equal to the number of elements of $H_i$ involved in $y$.
  Furthermore
  \[ \beta \deg x + \card I + \card J  =  \beta \deg y + (1 - \beta n) \card I + (1 - \beta (m-n)) \card J  \le  \beta \card y \]
  since $\deg v = n \card I$ and $\deg w = (m - n) \card J$.
  Hence $\beta \deg y = \inv \alpha \deg y \le \dim H_i$ implies $\eqcl y \in C'_*$.
  
  Using the same argument we can show that $D'_p = \DGCtrunc m {\operad C}_p$ for $p \le \alpha \dim H_i$.
  (A combinatorial argument using graphs would produce a better bound here; however, we will not need this.)
\end{proof}

\section{The undirected graph complex} \label{sec:undirected_gc}

Throughout this section, we let $m \in \NN$ be a natural number and $\operad C$ a cyclic operad concentrated in non-negative even degrees such that $\cyc{\operad C}{2} = \gen \basefield \id$ with the trivial $\Symm{2}$-action.
(As in the previous section, the condition that $\operad C$ is concentrated in even degrees is purely for the convenience of avoiding more signs.)

Our goal is to obtain a simpler description of the homology of the truncated directed graph complex $\DGCtrunc m {\operad C}$ of \cref{def:trunc_dgc}.
In fact, we will obtain a description in terms of \emph{undirected} graphs.
The rough idea is as follows.
Restricting the differential of $\DGCtrunc m {\operad C}$ to the two edges incident to a vertex of valence two (which must be labeled by the identity operation by our assumption on $\operad C$) yields the following picture.
\begin{center}
  \begin{tikzpicture} [scale = 1.5]
  \tikzstyle{vertex}=[circle, draw, inner sep = 0.05cm]
  \tikzstyle{bigvertex}=[ellipse, draw, inner sep = 0.05cm, minimum width = 1cm]
  
  \node[vertex] (l1) at (-1,0) {$\xi_1$};
  \node[vertex] (m1) at (0,0) {$\id$};
  \node[vertex] (r1) at (1,0) {$\xi_2$};
  
  \draw (l1.north west) -- ($2*(l1.north west) - (l1.center)$);
  \draw (l1.west) -- ($2*(l1.west) - (l1.center)$);
  \draw (l1.south west) -- ($2*(l1.south west) - (l1.center)$);
  \draw (r1.north east) -- ($2*(r1.north east) - (r1.center)$);
  \draw (r1.east) -- ($2*(r1.east) - (r1.center)$);
  \draw (r1.south east) -- ($2*(r1.south east) - (r1.center)$);
  
  \node at (2,0) {$\longmapsto$};
  
  \node[vertex] (l2) at (3,0) {$\xi_1$};
  \node[bigvertex] (r2) at (4,0) {$\xi_2$};
  
  \draw (l2.north west) -- ($2*(l2.north west) - (l2.center)$);
  \draw (l2.west) -- ($2*(l2.west) - (l2.center)$);
  \draw (l2.south west) -- ($2*(l2.south west) - (l2.center)$);
  \draw (r2.north east) -- ($1.75*(r2.north east) - 0.75*(r2.center)$);
  \draw (r2.east) -- ($1.5*(r2.east) - 0.5*(r2.center)$);
  \draw (r2.south east) -- ($1.75*(r2.south east) - 0.75*(r2.center)$);
  
  \node at (4.75,0) {$+$};
  
  \node[bigvertex] (l3) at (5.5,0) {$\xi_1$};
  \node[vertex] (r3) at (6.5,0) {$\xi_2$};
  
  \draw (l3.north west) -- ($1.75*(l3.north west) - 0.75*(l3.center)$);
  \draw (l3.west) -- ($1.5*(l3.west) - 0.5*(l3.center)$);
  \draw (l3.south west) -- ($1.75*(l3.south west) - 0.75*(l3.center)$);
  \draw (r3.north east) -- ($2*(r3.north east) - (r3.center)$);
  \draw (r3.east) -- ($2*(r3.east) - (r3.center)$);
  \draw (r3.south east) -- ($2*(r3.south east) - (r3.center)$);
  
  \begin{scope}[decoration = {markings, mark = at position 0.6 with {\arrow{>}}}]
  \draw[postaction={decorate}] (l1.east) -- (m1.west);
  \draw[postaction={decorate}] (r1.west) -- (m1.east);
  
  \draw[postaction={decorate}] (l2.east) -- (r2.west);
  
  \draw[postaction={decorate}] (r3.west) -- (l3.east);
  \end{scope}
  \end{tikzpicture}
\end{center}
Thus, intuitively, taking homology should kill the difference between the two orientations of an edge.
We formalize this using a spectral sequence argument: filtering $\DGCtrunc m {\Lie}$ by the number of vertices of valence two yields a spectral sequence, for which we prove that its first page is concentrated in a single row.
We moreover show that this row is isomorphic as a chain complex to a (well-known) undirected version of the graph complex, which we denote by $\UGC m {\operad C}$.
The spectral sequence then implies that there is an isomorphism $\Ho * (\DGCtrunc m {\Lie}) \iso \Ho * (\UGC m {\operad C})$.
(This argument is similar to a proof sketched by Willwacher \cite[Appendix~K]{Wil}.)

\begin{remark} \label{rem:ugc_coeff}
  There should also be a version of everything we do in this section in the more general situation of the truncated directed $(I,J)$-graph complex $\DGCtrunc m {\operad C}_{I,J}$.
  We will not elaborate on this here, though.
  It might appear in future work.
\end{remark}

We begin by noting that the degree $k$ part of the graded vector space $\DGCtrunc m {\operad C}$ splits as a direct sum
\[ \DGCtrunc m {\operad C}_k  =  \Dirsum_{l \in \NN} \DGCtrunc m {\operad C}_{k,l} \]
where $\DGCtrunc m {\operad C}_{k,l}$ is the subspace generated by elements represented by truncated decorated graphs with exactly $l$ vertices of valence two.
Note that a truncated decorated graph with $N$ vertices and $E$ edges represents an element of homological degree $k \ge N (1 - m) + E m$ (since $\operad C$ is concentrated in non-negative degrees).
Since all of its vertices are at least bivalent we have $E \ge N$ and hence $k \ge N$ (since $m \ge 0$).
Thus, in the above decomposition, a summand is trivial if $l > k$ and hence the sum is finite.

The differential restricts to a map
\[ d  \colon  \DGCtrunc m {\operad C}_{k,l}  \xlongto{(d^1, d^2)}  \DGCtrunc m {\operad C}_{k-1,l} \dirsum \DGCtrunc m {\operad C}_{k-1,l-1} \]
so that regrading via $C_{p,q} \defeq \DGCtrunc m {\operad C}_{p+q,q}$ yields a double complex $C_{*,*}$ with a differential $d^1$ of bidegree $(-1,0)$ and a differential $d^2$ of bidegree $(0, -1)$; that $d^1 \circ d^2 = - d^2 \circ d^1$ follows immediately from $d \circ d = 0$.
The total complex of $C_{*,*}$ is canonically isomorphic to $\DGCtrunc m {\operad C}$.
Moreover $C_{p,q}$ can only be non-trivial if $0 \le q \le p + q$.
Hence we obtain a convergent spectral sequence of the following form
\begin{equation} \label{ss:DGCtrunc}
  E^1_{p,q} = \Ho{q} \big( C_{p,*}, d^2 \big)  \abuts  \Ho{p+q} \big( \DGCtrunc m {\operad C} \big)
\end{equation}
(see e.g.\ \cite[§5.6]{Wei}).

We will use this spectral sequence to identify the homology of $\DGCtrunc m {\operad C}$.
However we will need a few preliminaries before we can do so, beginning with an undirected version of \cref{def:dirgraph}.

\begin{definition}
  We denote by $\Graph$ the following groupoid:
  \begin{itemize}
    \item Objects are tuples $(F, N, \mu, a)$ with $F$ and $N$ finite linearly ordered sets, $a \colon F \to N$ a map of sets, and $\mu$ a matching of $F$ (i.e.\ $\mu$ is a fixed-point-free bijection $F \to F$ of order two).
    \item A morphism $(F, N, \mu, a) \to (F', N', \mu', a')$ is a pair of bijections $(f \colon F \to F', g \colon N \to N')$ such that $f \after \mu = \mu' \after f$ and $g \after a = a' \after f$.
  \end{itemize}
  Moreover we write $\Graph[\ge 3]$ for the full subgroupoid of $\Graph$ spanned by those objects $(F, N, \mu, a)$ such that $\card {\inv a(v)} \ge 3$ for all $v \in N$.
\end{definition}

\begin{remark}
  Analogously to \cref{rem:DirGraph_intuition}, we can think of $\Graph$ as the groupoid of \emph{undirected graphs} and of $\Graph[\ge 3]$ as the groupoid of undirected graphs such that each vertex has at least valence $3$.
  
  Given $G = (F, N, \mu, a) \in \Graph$, the set $F$ specifies the set of flags, $N$ is the set of vertices (or nodes), the map $a$ specifies to which vertex a flag is incident, and the matching $\mu$ specifies how the flags are connected to form edges.
  In particular an \emph{edge} of $G$ is an orbit of $\mu$, i.e.\ a subset $\set {f, f'} \subseteq F$ such that $\mu(f) = f'$ and $\mu(f') = f$.
\end{remark}

\begin{definition}
  For $G = (F, N, \mu, a) \in \Graph$, we write $\Edge(G)$ for the set of edges of $G$.
  We equip it with the linear order pulled back along the injection $\Edge(G) \to F$ given by $\set {f, f'} \mapsto \min(f, f')$.
  This construction yields a functor $\Edge \colon \Graph \to \oBij$.
\end{definition}

\begin{definition}
  Let $D = (F, N, S, T, \mu, a)$ be a directed graph.
  Its \emph{underlying undirected graph} is the undirected graph $(F, N, \mu', a)$, where $\mu' \colon F \to F$ is the matching given by $\mu'(s) = \mu(s)$ for $s \in S \subseteq F$ and $\mu'(t) = \inv\mu(t)$ for $t \in T \subseteq F$.
\end{definition}

\begin{definition}
  For a truncated directed graph $D$ we define its \emph{underlying graph} $\Underlying (D) \in \Graph[\ge 3]$ to be the underlying undirected graph of $D$, except that all vertices of valence two together with their two incident edges are replaced by a single edge.
\end{definition}

\begin{remark} \label{rem:underlying_unprecise}
  To be entirely precise we would need to specify a linear order on the set of vertices and the set of flags of $\Underlying (D)$.
  This choice won't matter, however, and so we can make it arbitrarily.
\end{remark}

We will now study (the homology of) the chain complexes $C^2_{p,*} \defeq (C_{p,*}, d^2)$ occurring in the spectral sequence \eqref{ss:DGCtrunc}.
To this end, let $D$ be a truncated directed graph.
Note that applying $d^2$ to an element represented by $D$ yields a sum of elements each of which is represented by a directed graph $D'$ such that $\Underlying(D') \iso \Underlying(D)$.
Hence there is a splitting of chain complexes
\[ C^2_{p,*}  \iso  \Dirsum_{\eqcl G} C^G_{p,*} \]
where $\eqcl G$ runs over all isomorphism classes of $\Graph[\ge 3]$ and $C^G_{p,*}$ is the subcomplex of $C^2_{p,*}$ generated by the elements represented by directed graphs $D$ with $\Underlying(D) \iso G$.
Thus we can restrict our attention to the chain complexes $C^G_{p,*}$.

To understand (the homology of) $C^G_{p,*}$ we will construct a functor $\widetilde C_{\bullet,*}$ from $\Graph[\ge 3]$ to chain complexes in graded vector spaces (here the first grading $\bullet$ is the one of the graded vector space, and the second grading $*$ is the one of the chain complex) such that
\begin{equation} \label{eq:covering_condition}
\coinv{\big( \widetilde C_{\bullet,*}(G) \big)}{\Aut(G)}  \iso  C^G_{\bullet,*}  \defeq  \Dirsum_{p} C^G_{p,*}
\end{equation}
for all $G$.
Since $\Aut(G)$ is finite and, in characteristic $0$, quotienting by a finite group commutes with taking homology, we will then be able to deduce the homology of $C^G_{\bullet,*}$ from the homology of $\widetilde C_{\bullet,*}(G)$ together with its $\Aut(G)$-action.

More precisely, we will take a ``coordinate-free'' approach to the isomorphism \eqref{eq:covering_condition} and actually construct an isomorphism
\[ \colim{\Graph[\ge 3]} \widetilde C_{\bullet,*} \iso C^2_{\bullet,*} \defeq \big( {\Dirsum_p C_{p,*}}, d^2 \big) \]
of chain complexes in graded vector spaces.
This will be the more convenient approach as it doesn't require us to chose representatives of isomorphism classes of $\Graph[\ge 3]$.
We need a few preliminaries before we can construct $\widetilde C_{\bullet,*}$.

\begin{definition}
  Let $S$ be a set.
  We write $\dir(S)$ for the set of functions $S \to \set{-, 0, +}$.
  For an element $o \in \dir(S)$, we set its \emph{degree} to be $\deg o \defeq \card{\inv o(0)}$.
  
  Moreover, for $\? \in \set{+, -}$, $o \in \dir(S)$, and $s \in S$ such that $o(s) = 0$, we write $o_s^\? \in \dir(S)$ for the function with $o_s^\?(s) = \?$ and $o_s^\?(s') = o(s')$ for all $s' \in S \setminus \set s$.
\end{definition}

\begin{definition}
  When $G$ is an undirected graph, we write $\dir(G) \defeq \dir(\Edge(G))$.
  This forms a (covariant) functor $\dir$ from $\Graph$ to the category of sets by letting an isomorphism $\chi \colon G \to G'$ act on $o \in \dir(G)$ such that
  \[ (\chi \act o)(e)  \defeq  \begin{cases} o(\inv\chi(e)), & \text{if } \inv\chi(f_1) < \inv\chi(f_2) \\ -o(\inv\chi(e)), & \text{if } \inv\chi(f_1) > \inv\chi(f_2) \end{cases} \]
  for each edge $e = \set{f_1 < f_2}$ of $G'$.
  (Here we set $-- = +$, $-0 = 0$, and $-+ = -$, as expected.)
\end{definition}

\begin{remark}
  We think of the elements of $\dir(G)$ as all possible ways to, for each edge of $G$, either chose an orientation or leave it undirected.
\end{remark}

\begin{definition}
  Let $e = \set{f < f'}$ be an edge of an undirected graph $G$.
  We set the vector space $\Or(e)$ of orientations of $e$ to be $\basefield$ (in degree $0$).
  An isomorphism $\chi \colon G \to G'$ of undirected graphs induces a map $\Or(e) \to \Or(\chi(e))$ by setting it to be $\id[\basefield]$ if $\chi(f) < \chi(f')$ and $- {\id[\basefield]}$ otherwise.
\end{definition}

\begin{definition} \label{def:Direct}
  Let $G \defeq (F, N, \mu, a) \in \Graph[\ge 3]$ and $o \in \dir(G)$.
  We write $\Direct(G, o)$ for the truncated directed graph obtained from $G$ by, for each edge $e = \set{f < f'}$, doing the following
  \begin{itemize}
    \item if $o(e) = +$, we orient it from $f$ to $f'$ (i.e.\ $f \in S$ and $f' \in T$).
    \item if $o(e) = -$, we orient it from $f'$ to $f$ (i.e.\ $f' \in S$ and $f \in T$).
    \item if $o(e) = 0$, we replace $e$ by a new vertex $v_e$ with two incoming edges $(s_e^+, t_e^+)$ and $(s_e^-, t_e^-)$ such that the flag $s_e^+$ is incident to $a(f')$, the flag $s_e^-$ is incident to $a(f)$, and the flags $t_e^+$ and $t_e^-$ are incident to $v_e$.
  \end{itemize}
  The linear orders on the set of vertices and the set of flags of $\Direct(G, o)$ are inherited from $G$ such that, for $e = \set{f < f'}$ with $o(e) = 0$, the new vertex $v_e$ is inserted as the successor of $a(f)$ (if there are multiple such vertices, they are inserted according to the order of the edges $e$), the two new flags $s_e^+$ and $t_e^+$ are successively inserted in this order instead of $f'$, and similarly for $s_e^-$, $t_e^-$, and $f$.
  
  This construction yields an equivalence of categories $\Direct$ from the Grothendieck construction $\gc[{\Graph[\ge 3]}] \dir$ to $\DirGraphtrunc$.
\end{definition}

\begin{remark}
  Note that $\Underlying(\Direct(G, o))$ is canonically isomorphic to $G$ for all $o$.
\end{remark}

\begin{definition}
  We say that a collection of decorations $\xi_v$ of the vertices $v$ of $G$ \emph{represents} the element of $\Dec{m}{\operad C}(\Direct(G, o))$ that is represented by the $\xi_v$ and $\id \in \cyc {\operad C} 2$ at all of the added vertices of valence two.
\end{definition}

We are now ready to construct the functor $\widetilde C_{\bullet,*}$ of \eqref{eq:covering_condition}.
As a functor to bigraded vector spaces we set
\begin{equation} \label{eq:def_covering_complex}
  \widetilde C(G)  \defeq  \Dirsum_{o \in \dir(G)} \Dec{m}{\operad C}(\Direct(G, o))
\end{equation}
bigraded by $(k - \deg o, \deg o)$, where $k$ is the homological degree.
An isomorphism $\chi \colon G \to G'$ acts on this in the canonical way, i.e.\ it maps $\Dec{m}{\operad C}(\Direct(G, o))$ to $\Dec{m}{\operad C}(\Direct(G, \chi \act o))$ via the map induced by $\Direct(\chi, o)$.
This is precisely the following left Kan extension
\[
\begin{tikzcd}
  \gc[{\Graph[\ge 3]}] \dir \dar[swap]{\pr} \rar{\Direct}[swap]{\eq} & \DirGraphtrunc \dar{\Dec{m}{\operad C}} \\
  \Graph[\ge 3] \rar[dashed]{\widetilde C} & \BiGrVect
\end{tikzcd}
\]
of $\Dec{m}{\operad C} \circ \Direct$ along the projection $\pr$.

Since $\Direct$ is an equivalence of categories and since, by definition of $C^2$, there is an isomorphism of bigraded vector spaces $\colim{\DirGraphtrunc} \Dec{m}{\operad C} \iso C^2$, we obtain an isomorphism $\colim{\Graph[\ge 3]} \widetilde C \iso C^2$ of bigraded vector spaces.
(This uses that the colimit of a left Kan extension is isomorphic to the colimit of the original functor.)

To construct the differential of $\widetilde C(G)$ we will use the technical lemma following the next definition.

\begin{definition}
  Our standing assumption that $\cyc {\operad C} 2 = \gen \basefield \id$ implies that the vector spaces $\Dec{m}{\operad C}(\Direct(G, o))$ are independent of $o$ up to isomorphism (when $o(e) = 0$, the extra vertex $v_e$ is labeled by an element of $\cyc {\operad C} 2 \iso \QQ$).
  We denote by
  \[ c_{G,e}  \colon  \Dec{m}{\operad C}(\Direct(G, o_e^+))  \xlongto{\iso}  \Dec{m}{\operad C}(\Direct(G, o_e^-)) \]
  the canonical such isomorphism (by a slight abuse, we do not include $o$ in the notation).
  It is induced by the canonical bijection $\sigma$ between the set of initial flags of $\Direct(G, o_e^+)$ and the set of initial flags of $\Direct(G, o_e^-)$, i.e.\ it is given by $\id \tensor (\shift[m] \basefield)^{\tensor \sigma}$.
\end{definition}

\begin{lemma} \label{lemma:props_lifted_partial_diffs}
  There exist, for all $G \in \Graph[\ge 3]$, $e \in \Edge(G)$, and $o \in \dir(G)$ such that $o(e) = 0$, two isomorphisms
  \begin{align*}
    d_{G,e}^+ &\colon \Dec{m}{\operad C}(\Direct(G, o)) \longto \Dec{m}{\operad C}(\Direct(G, o_e^+)) \\
    d_{G,e}^- &\colon \Dec{m}{\operad C}(\Direct(G, o)) \longto \Dec{m}{\operad C}(\Direct(G, o_e^-))
  \end{align*}
  such that the following properties hold for all elements $\? \in \set{\pm}$ and $x \in \Dec{m}{\operad C}(\Direct(G, o))$:
  \begin{enumerate}
    \item Let $\chi \colon G \to G'$ be an isomorphism of $\Graph[\ge 3]$ and set $\?' \defeq (\chi \act o_e^\?)(\chi(e))$.
    Then $\chi \act d_{G,e}^\?(x) = d_{G',\chi(e)}^{\?'}(\chi \act x)$.
    \item Let $e' \neq e$ be another edge of $G$ such that $o(e') = 0$, and let $\?' \in \set{\pm}$.
    Then $d_{G,e'}^{\?'} (d_{G,e}^\? (x))= - d_{G,e}^\? (d_{G,e'}^{\?'} (x))$.
    \item We have $r (d_{G,e}^\? (x)) = d^{\Direct(G,o)}_{s_e^\?} (r (x))$, where $r$ denotes the canonical map from $\Dec{m}{\operad C}(\Direct(G, o))$ respectively $\Dec{m}{\operad C}(\Direct(G, o_e^\?))$ to $C^2$.
    \item We have that $d_{G,e}^- \circ \inv{(d_{G,e}^+)} \colon \Dec{m}{\operad C}(\Direct(G, o_e^+)) \to \Dec{m}{\operad C}(\Direct(G, o_e^-))$ is equal to $(-1)^m c_{G,e}$.
  \end{enumerate}
\end{lemma}

Assuming this lemma, we can define a differential on $\widetilde C(G)$ of bidegree $(0, -1)$ by
\begin{equation} \label{eq:ugc_ss_diff}
  d \defeq \sum_{e \in \Edge(G)} (d_{G,e}^+ + d_{G,e}^-)
\end{equation}
(here $d_{G,e}^\?$ is understood to be zero on $\Dec{m}{\operad C}(\Direct(G, o))$ when $o(e) \neq 0$).
By the second property above we have $d \circ d = 0$ and hence that $\widetilde C_{\bullet,*}(G) \defeq (\widetilde C(G), d)$ is a chain complex.
Moreover the first property implies functoriality of $\widetilde C_{\bullet,*}$, and the third property implies that the isomorphism $\colim{\Graph[\ge 3]} \widetilde C \iso C^2$ is compatible with the differentials.

\begin{proof}[Proof of \Cref{lemma:props_lifted_partial_diffs}]
  It will be convenient to construct the isomorphisms $d_{G,e}^+$ and $d_{G,e}^-$ simultaneously for all elements $(G, e)$ of some fixed isomorphism class of the category of undirected graphs with a specified edge.
  To this end we fix some pair $(G_0 = (F_0, N_0, \mu_0, a_0), e_0)$ in this isomorphism class such that the one or two vertices incident to $e_0$ are the first elements of $N_0$, such that $a_0$ is order preserving, and such that the two flags making up $e_0$ come first in the linear orders of their respective vertices (if they are incident to the same vertex, we ask these two flags to be the first two elements).
  
  Now let $o \in \dir(G_0)$ such that $o(e_0) = 0$ and let $x \in \Dec{m}{\operad C}(\Direct(G_0, o))$ be the element represented by some decorations $(\xi_v)_{v \in N_0}$.
  Then we set $d_{G_0,e_0}^\?(x)$ to be $(-1)^{\card N m + n + b_\?}$ times the element of $\Dec{m}{\operad C}(\Direct(G_0, o_{e_0}^\?))$ represented by the $\xi_v$.
  Here
  \begin{align*}
    b_- &\defeq 0 \\
    b_+ &\defeq \begin{cases} m, & \text{if $e_0$ is a loop} \\ m \outdeg{v_{1,o}}, & \text{otherwise} \end{cases}
  \end{align*}
  where $v_{1,o}$ is the first element of $N_0$ considered as a vertex of $\Direct(G_0, o)$.
  
  We extend this to all pairs $(G, e)$ by choosing some isomorphism $\chi \colon G \to G_0$ such that $\chi(e) = e_0$, and setting for $o \in \dir(G)$, $x \in \Dec{m}{\operad C}(\Direct(G, o))$, and $\? \in \set{\pm}$,
  \begin{equation} \label{eq:def_lifted_partial_diffs}
    d_{G,e}^\?(x) \defeq \inv{\Direct(\chi, o_e^\?)} \act d_{G_0,e_0}^{\?'}(\Direct(\chi, o) \act x)
  \end{equation}
  where $\?' \defeq (\chi \act o_e^\?)(e_0)$.
  To see that this is well-defined we need to show that equation \eqref{eq:def_lifted_partial_diffs} holds when $(G, e) = (G_0, e_0)$ and $\chi$ is any automorphism of $G_0$ such that $\chi(e_0) = e_0$.
  This follows directly from comparing the signs induced by $\Direct(\chi, o)$ and $\Direct(\chi, o_{e_0}^\?)$ (distinguishing the cases of $e_0$ being a loop or not and $\chi$ flipping the two flags making up $e_0$ or not).
  The first property in the statement of the lemma is thus true by definition.
  
  To prove the third property, it is enough, by equation \eqref{eq:def_lifted_partial_diffs} and the fact that $r(\chi \act x) = r(x)$, to do so for the pair $(G_0, e_0)$.
  In this case, the equality of the signs involved can be deduced easily from \Cref{def:dgc,def:Direct}.
  
  The second property follows from similar arguments: by equation \eqref{eq:def_lifted_partial_diffs} it is enough to check this for one member of each of the isomorphism classes of the category of undirected graphs with two distinct specified edges.
  Distinguishing various cases, corresponding to the different ways the vertices incident to the two edges can coincide or not, we choose representatives of the isomorphism classes that behave nicely with our construction above.
  In these cases it is then straightforward (but tedious) to check the equality directly from the definition.
  
  To show the fourth property, we first note that, for all isomorphisms $\chi \colon G \to G'$, the following diagram of isomorphisms commutes
  \[
  \begin{tikzcd}[column sep = 40]
    \Dec{m}{\operad C}(\Direct(G, o_e^+)) \rar{c_{G,e}} \dar[swap]{\chi} & \Dec{m}{\operad C}(\Direct(G, o_e^-)) \dar{\chi} \\
    \Dec{m}{\operad C}(\Direct(G', o_{\chi(e)}^\?)) \rar[leftrightarrow]{c_{G',\chi(e)}} & \Dec{m}{\operad C}(\Direct(G', o_{\chi(e)}^{-\?}))
  \end{tikzcd}
  \]
  where $\? \defeq (\chi \act o)(\chi(e))$.
  Hence it is enough to prove this property in the cases of the pairs $(G_0, e_0)$.
  There it is true by definition since $(-1)^{b_+ + b_- + m}$ is (regardless of whether $e_0$ is a loop or not) precisely the $m$-th power of the sign of the canonical bijection between the set of initial flags of $\Direct(G_0, o_{e_0}^+)$ and the set of initial flags of $\Direct(G_0, o_{e_0}^-)$.
\end{proof}

We now claim that, for any $G$, the homology of $\widetilde C_{\bullet,*}(G)$ is concentrated in bidegrees $(p, 0)$ (which implies that the spectral sequence \eqref{ss:DGCtrunc} collapses on the second page).
This follows immediately from the following more general lemma.

\begin{lemma} \label{lemma:cube_homology}
  Let $S$ be a finite set and $(V_o)_{o \in \dir(S)}$ a family of vector spaces.
  Furthermore let there be, for each $s \in S$ and $o \in \dir(S)$ such that $o(s) = 0$, two isomorphisms $d_s^+ \colon V_o \to V_{o_s^+}$ and $d_s^- \colon V_o \to V_{o_s^-}$.
  Assume that, for all $\?, \?' \in \set{\pm}$ and $s \neq s' \in S$ such that $o(s) = 0 = o(s')$, we have an equality $d_s^? \circ d_{s'}^{\?'} = - d_{s'}^{\?'} \circ d_s^\?$ of maps $V_o \to V_{o_{s\,s'}^{\?\,\?'}} = V_{o_{s'\,s}^{\?'\,\?}}$.
  
  Let $C \defeq \Dirsum_{o \in \dir(S)} V_o$ be graded by the degree of $o$, and equip it with the differential given by $d \defeq \sum_{s \in S} (d_s^+ + d_s^-)$ (here we set $d_s^?(V_o) = 0$ for all $o$ with $o(s) \neq 0$).
  Then $C_*$ is a chain complex with homology concentrated in degree $0$.
\end{lemma}

\begin{proof}
  First note that the anticommutativity of the partial differentials implies that $d \circ d = 0$, so we only need to show the claim about the homology.
  Now choose some $s \in S$.
  Then $C$ can be bigraded via
  \[
  C_{p,0}  \defeq  \Dirsum_{\substack{o \in \dir(S) \\ \card{\inv o(0)} = p \\ o(s) \neq 0}} V_o
  \qquad \text{and} \qquad
  C_{p,1}  \defeq  \Dirsum_{\substack{o \in \dir(S) \\ \card{\inv o(0)} = p + 1 \\ o(s) = 0}} V_o
  \]
  and $C_{p,q} \defeq 0$ for $q \not\in \set{0, 1}$.
  The two differentials $d^1 \defeq \sum_{s' \in S \setminus \set {s}} (d_{s'}^+ + d_{s'}^-)$ and $d^2 \defeq d_{s}^+ + d_{s}^-$ give this the structure of a first-quadrant double complex with total complex isomorphic to $C$.
  Hence there is a convergent spectral sequence
  \begin{equation} \label{eq:cube_ss}
    E^1_{p,q} = \Ho{q} (C_{p,*}, d^2)  \abuts  \Ho{p+q} (C_*)
  \end{equation}
  with the differential on the first page induced by $d^1$ (see e.g.\ \cite[§5.6]{Wei}).
  
  Note that $C_{p,*}$ splits as $\Dirsum_o C_{p,*}^o$ where $o$ runs over those elements of $\dir(S)$ of degree $p+1$ such that $o(s) = 0$, and $C_{p,*}^o$ is the subcomplex
  \[ V_o  \xlongto{(d_s^+, d_s^-)}  V_{o_{s}^+} \dirsum V_{o_{s}^-} \]
  of $C_{p,*}$.
  Since $d_s^+$ and $d_s^-$ are both isomorphisms, the homology of $C_{p,*}^o$ is concentrated in degree 0.
  Moreover the inclusion $V_{o_s^+} \to C_{p,*}^o$ induces an isomorphism on homology (where the former is considered as a chain complex concentrated in degree 0).
  Hence, setting
  \[ C^+  \defeq  \Dirsum_{\substack{o \in \dir(S) \\ o(s) = +}} V_o \]
  grading it by the degree of $o$ and equipping it with $d^1$ as a differential, we obtain an isomorphism of chain complexes $C^+_\bullet \to \big( \Ho{0}(C_{\bullet,*}, d^2), d^1 \big)$.
  Since the $E_1$-page of the spectral sequence \eqref{eq:cube_ss} is concentrated in the row $q = 0$, there thus is an isomorphism $\Ho{*}(C) \iso \Ho{*}(C^+)$.
  Hence it is enough to show that the homology of $C^+$ is concentrated in degree $0$.
  
  To see this, we note that $C^+$ is a chain complex as described in the statement of this lemma, but with base set $S \setminus \set s$ instead of $S$.
  Thus we can, by induction, reduce to the case where $S$ is empty.
  In that case the statement is trivially true, as any such chain complex is concentrated in degree $0$.
\end{proof}

Let us recap what we did so far.
We constructed the spectral sequence \eqref{ss:DGCtrunc}
\[ E^1_{p,q} = \Ho{q} \big( C^2_{p,*} \big)  \abuts  \Ho{p+q} \big( \DGCtrunc m {\operad C} \big) \]
and, in \eqref{eq:def_covering_complex} and \eqref{eq:ugc_ss_diff}, a functor $\widetilde C_{\bullet,*}$ from $\Graph[\ge 3]$ to chain complexes in graded vector spaces such that there is an isomorphism $\colim{} \widetilde C_{\bullet,*} \iso C^2_{\bullet,*}$.
Then we proved, in \cref{lemma:cube_homology}, that the pointwise homology $\Ho * \after \widetilde C_{\bullet,*}$ is concentrated in bidegrees $(p, 0)$.

Now note that $\Graph[\ge 3]$ is a disjoint union of groupoids equivalent to finite groups, and hence that taking rational homology commutes with $\Graph[\ge 3]$-colimits (since rational homology commutes with direct sums and quotients by finite groups).
Hence there are isomorphisms
\begin{equation} \label{eq:ss_first_page}
  \Ho * \big( C^2_{\bullet,*} \big)  \iso  \Ho * \big( \colim {\Graph[\ge 3]} \widetilde C_{\bullet,*} \big)  \iso  \colim {\Graph[\ge 3]} \big( \Ho * \after \widetilde C_{\bullet,*} \big)
\end{equation}
of bigraded vector spaces.
In particular the homology of $C^2_{\bullet,*}$ is also concentrated in bidegrees $(p, 0)$.
The spectral sequence then implies that the differential $d^1$ of the double complex $C_{*,*}$ induces a differential on $\Ho 0 (C^2_{\bullet,*})$ such that there is an isomorphism
\begin{equation} \label{eq:ss_collapse}
  \Ho \bullet \big( \DGCtrunc m {\operad C} \big)  \iso  \Ho \bullet \big( \Ho 0 (C^2_{\bullet,*}), d^1 \big) 
\end{equation}
of graded vector spaces (\cref{lemma:double_complex_ss_collapse} provides an explicit description of this isomorphism, which we will use later).
Pulling back the differential $d^1$ along the isomorphism \eqref{eq:ss_first_page} hence yields a differential on the graded vector space $\colim{} (\Ho{0} \circ \widetilde C_{\bullet,*})$ such that the homology is isomorphic to the homology of $\DGCtrunc m {\operad C}$.
This colimit, equipped with the pulled-back differential, is the undirected graph complex we are looking for.
To prove this, we will first give an explicit description of the functor $\Ho{0} \circ \widetilde C_{\bullet,*}$, and then also identify the differential in these terms.

\begin{definition}
  For an undirected graph $G \defeq (F, N, \mu, a) \in \Graph[\ge 3]$ we set
  \[ \UDec{m}{\operad C}(G)  \defeq  (\shift \basefield)^{\tensor N} \tensor (\shift[-m] \operad C)(a) \tensor \Tensor_{e \in \Edge(G)} \shift[m] \big( \Or(e)^{\tensor m+1} \big) \]
  and call it the \emph{(graded) vector space of $\operad C$-decorations on $G$}.
  This yields a functor $\UDec{m}{\operad C} \colon \Graph[\ge 3] \to \GrVect$.
  Moreover we set
  \[ \UGC m {\operad C}  \defeq  \colim{\Graph[\ge 3]} \UDec{m}{\operad C} \]
  and call it the \emph{(graded) vector space of $\operad C$-decorated undirected graphs}.
  
  Given a family $\big( \xi_v \in \operad C(\inv a(v)) \big)_{v \in N}$ of elements of $\operad C$, we will call
  \[ \shift[\card N] 1 \tensor \Tensor_{v \in N} \shift[-m] \xi_v \tensor \shift[m \card{\Edge(G)}] 1 \in \UDec m {\operad C} (G) \]
  the element \emph{represented} by ($G$ and) the decorations $\xi_v$.
  By a slight abuse of terminology we will use the same name for its image in $\UGC{m}{\operad C}$.
\end{definition}

\begin{lemma} \label{lemma:UDec_iso}
  The following composition
  \[ \UDec{m}{\operad C}(G)  \xlongto{\iso}  \Dec{m}{\operad C}(\Direct(G, \const[+]))  \xlongto{\iso}  \Ho{0}(\widetilde C_{\bullet,*}(G)) \]
  is a natural isomorphism $\UDec{m}{\operad C} \to \Ho{0} \after \widetilde C_{\bullet,*}$ of functors $\Graph[\ge 3] \to \GrVect$.
  Here the left hand map is given by sending the element represented by $G$ and some decorations $\xi_v$ to the element represented by $\Direct(G, \const[+])$ and the same decorations $\xi_v$.
  The right hand map is the composite of the canonical inclusion of the middle term into $\widetilde C_{\bullet,0}(G)$ followed by the projection onto the zeroth homology.
\end{lemma}

\begin{proof}
  That the left hand map is an isomorphism is clear.
  Hence it is enough to show that the right hand map is one as well, and that their composition is natural.
  We begin with the former.
  
  By definition we have
  \[ \Ho{0}(\widetilde C_*(G))  =  \quot {\big( \Dirsum_{\substack{o \in \dir(G) \\ \deg o = 0}} \Dec{m}{\operad C}(\Direct(G, o)) \big)} {\sim} \]
  where $\sim$ is generated by the relations $d_{G,e}^+(x) \sim - d_{G,e}^-(x)$ for each edge $e$ of $G$, element $o \in \dir(S)$ such that $\inv o(0) = \set e$, and $x \in \Dec{m}{\operad C}(\Direct(G, o))$.
  The relation $\sim$ can be equivalently described as being generated by $x \sim - d_{G,e}^- \big( \inv{(d_{G,e}^+)}(x) \big)$ for each edge $e$, element $o \in \dir(S)$ with $\deg o = 0$ and $o(e) = +$, and $x \in \Dec{m}{\operad C}(\Direct(G, o))$.
  
  By the last property of \Cref{lemma:props_lifted_partial_diffs}, we have an equality of isomorphisms $- d_{G,e}^- \after \inv{(d_{G,e}^+)} = (-1)^{m+1} c_{G,e}$.
  We also have, by definition, that $c_{G,e} \after c_{G,e'} = c_{G,e'} \after c_{G,e}$.
  This implies that each of the canonical maps $\Dec{m}{\operad C}(\Direct(G, o)) \to \Ho{0}(\widetilde C_*(G))$ is an isomorphism.
  
  We will now identify the functorial structure of $\Ho{0} \circ \widetilde C_{\bullet,*}$ in these terms, which will imply the desired naturality.
  Let $\chi \colon G \to G'$ be an isomorphism and let $x \in \Dec{m}{\operad C}(\Direct(G, \const[+]))$ be represented by some decorations $(\xi_v)_{v \in N}$ where $N$ is the set of vertices of $G$.
  Then $\chi \act x$ is $\epsilon \in \set{\pm 1}$ times the element $y$ of $\Dec{m}{\operad C}(\Direct(G', \chi \act \const[+]))$ represented by the decorations $(\chi \act \xi_{\inv \chi(v)})_{v \in N'}$ where $N'$ is the set of vertices of $G'$.
  This element $y$ is equivalent under $\sim$ to
  \[ (-1)^{k(m+1)} \big( \inv {c_{G',e_k}} \circ \dots \circ \inv{c_{G',e_1}} \big) (y)  \in  \Dec{m}{\operad C}(\Direct(G', \const[+])) \]
  where $\set{e_1, \dots, e_k} \defeq \set{e \in \Edge(G') \mid (\chi \act \const[+])(e) \neq +}$.
  We hence obtain that $\chi \act x$ is equivalent under $\sim$ to $\epsilon (-1)^{k(m+1)} \sgn(\sigma)^m$ times the element of $\Dec{m}{\operad C}(\Direct(G', \const[+]))$ represented by the decorations $(\chi \act \xi_{\inv \chi(v)})_{v \in N'}$, where $\sigma$ is the canonical bijection between the set of initial flags of $\Direct(G', \chi \act \const[+])$ and the set of initial flags of $\Direct(G', \const[+])$.
  Finally we note that, under the canonical identification of $\Edge(G)$ with the initial flags $S$ of $\Direct(G, \const[+])$ and the analogous identification for $G'$ (both of which are order preserving), the composition $\sigma \circ \restrict{\Direct(\chi, \const[+])}{S}$ corresponds precisely to the map $\Edge(G) \to \Edge(G')$ induced by $\chi$.
  This finishes the proof.
\end{proof}

Combining \cref{lemma:UDec_iso} with the isomorphism \eqref{eq:ss_first_page}, we obtain an isomorphism
\begin{equation} \label{eq:gc_identification}
  \UGC m {\operad C}  =  \colim{\Graph[\ge 3]} \UDec{m}{\operad C}  \iso  \colim{\Graph[\ge 3]} (\Ho{0} \circ \widetilde C_{\bullet,*})  \iso  \Ho 0 (C^2_{\bullet,*})
\end{equation}
of graded vector spaces.
As explained after \eqref{eq:ss_collapse}, we will now identify the result of pulling back the differential $d^1$ of $\Ho 0 (C^2_{\bullet,*})$ to $\UGC m {\operad C}$ along the isomorphism \eqref{eq:gc_identification}.
Moreover, note that, via the isomorphisms \eqref{eq:ss_collapse} and \eqref{eq:gc_identification}, the graded coalgebra structure of $\Ho \bullet \big( \DGCtrunc m {\operad C} \big)$ determines a graded coalgebra structure on $\Ho \bullet \big( \UGC{m}{\operad C} \big)$.
We will also describe this graded coalgebra structure explicitly.

\begin{definition}
  Let $G = (F, N, \mu, a) \in \Graph[\ge 3]$ be an undirected graph.
  A \emph{neighbor-closed vertex set} of $G$ is a subset $N' \subseteq N$ such that $\mu(\inv a(N')) = \inv a(N')$.
  A \emph{connected component} of $G$ is a neighbor-closed vertex set that is inclusion minimal among non-empty neighbor-closed vertex sets.
  
  A neighbor-closed vertex set $N'$ of $G$ determines an undirected graph
  \[ \induce[G](N')  \defeq  \big(F', N', \restrict \mu {F'}, \restrict a {F'} \big) \in \Graph[\ge 3] \]
  where $F' \defeq \inv a(N')$.
\end{definition}

\begin{remark}
  When $N' \subseteq N$ is a neighbor-closed vertex set of an undirected graph $G = (F, N, \mu, a)$, then $N \setminus N'$ is one as well.
  Moreover, any neighbor-closed vertex set of $G$ is a disjoint union of a unique set of connected components of $G$.
\end{remark}

\begin{definition} \label{def:UGC_diff}
  We equip the graded vector space $\UGC m {\operad C}$ of $\operad C$-decorated undirected graphs with the structure of a cocommutative differential graded coalgebra in the following way.
  The differential is given by
  \[ d(\eqcl x)  \defeq  \sum_{e \in \Edge(G)} d^G_e(x) \]
  where $G = (F, N, \mu, a) \in \Graph[\ge 3]$ is an undirected graph and $x \in \UDec{m}{\operad C}(G)$.
  Here the $d^G_e(x)$ are elements of $\UGC m {\operad C}$ uniquely determined by the following properties:
  \begin{itemize}
    \item For every isomorphism $\chi \colon G \to G'$ of undirected graphs holds $d^G_e(x) = d^{G'}_{\chi(e)}(\chi \act x)$.
    \item If $e$ is a loop, then $d^G_e(x) = 0$.
    \item If $e$ is not a loop, let $f < f'$ be the two flags making up $e$ and assume that $v \defeq a(f)$ and $v' \defeq a(f')$ are, in this order, the first two elements of $N$.
    Moreover assume that $a$ is order-preserving, and that $f$ is the first element of $F$ (in particular $e$ is the first element of $\Edge(G)$).
    Then $d^G_e$ applied to the element represented by $G$ and some decorations $(\xi_w)_{w \in N}$ is the element represented by the graph
    \[ \contract{e}(G)  \defeq  \big( F \setminus e, N / (v \sim v'), \restrict \mu {F \setminus e}, a' \defeq \pr \after \restrict a {F \setminus e} \big) \]
    (obtained from $G$ by contracting $e$) decorated by $\omega \act (\xi_v \opcomp[f]{f'} \xi_{v'})$ at the collapsed vertex $\bar v \defeq \pr(v) = \pr(v')$ and by $\xi_w$ at all other vertices $w$.
    Here $\omega$ is the canonical bijection between the linearly ordered set $\inv{a}(v) \opcomp[f]{\mu(f)} \inv{a}(v')$ of \cref{rem:cyclic_operad_comp} and $\inv{a'}(\bar v) \subseteq F \setminus e$.
  \end{itemize}

  The counit $\epsilon \colon \UGC m {\operad C} \to \QQ$ sends the element represented by the empty graph to $1$ and an element represented by any other graph to $0$.
  The comultiplication $\Delta \colon \UGC m {\operad C} \to \UGC m {\operad C} \tensor \UGC m {\operad C}$ is given as follows: for an element $x \in \UGC m {\operad C}$ represented by $G = (F, N, \mu, a)$ and decorations $(\xi_v)_{v \in N}$, we set
  \[ \Delta(x)  \defeq  \sum_{N'} (-1)^{m \card {N'} \card{N \setminus N'}} \delta(N') x_{N'} \tensor x_{N \setminus N'} \]
  where the sum runs over all neighbor-closed vertex subsets $N'$ of $G$ and $x_A$ is the element of $\UGC m {\operad C}$ represented by $\induce[G](A)$ and $(\xi_v)_{v \in A}$.
  The sign $\delta(N') \in \set {\pm 1}$ is the sign incurred by permuting
  \begin{gather*}
  (\shift 1)^{\tensor N} \tensor \Tensor_{v \in N} \shift[-m] \xi_v \tensor (\shift[m] 1)^{\tensor \Edge(G)} \\
  \intertext{into}
  (\shift 1)^{\tensor N'} \tensor \Tensor_{v \in N'} \shift[-m] \xi_v \tensor (\shift[m] 1)^{\tensor \Edge(G_1)} \tensor (\shift 1)^{\tensor N \setminus N'} \tensor \Tensor_{v \in N \setminus N'} \shift[-m] \xi_v \tensor (\shift[m] 1)^{\tensor \Edge(G_2)}
  \end{gather*}
  where $G_1 \defeq \induce[G](N')$ and $G_2 \defeq \induce[G](N \setminus N')$.
\end{definition}

\begin{remark}
  Note that it is easy to see that $\Delta$ is well-defined, i.e.\ that $\Delta(\chi \act x) = \Delta(x)$ for any isomorphism $\chi \colon G \to G'$.
  In the proof of \cref{lemma:UGC_diff} we will show that the elements $d^G_e(x)$ actually exist, which in particular implies that $d$ is a well-defined map.
  Moreover, we will construct a surjection from a cocommutative differential graded coalgebra onto $\UGC m {\operad C}$ that is compatible with the differentials, the counits, and the comultiplications.
  This implies that the maps above indeed equip $\UGC m {\operad C}$ with the structure of a cocommutative differential graded coalgebra.
\end{remark}

\begin{remark}
  The \emph{(undirected) graph complex} $\UGC{m}{\operad C}$, equipped with the differential graded coalgebra structure of \cref{def:UGC_diff}, is a well-known object.
  Variants of it were first described by Kontsevich \cite{Kon93,Kon94} and have, since then, appeared in the works of many different authors.
  Using the language of modular operads due to Getzler--Kapranov \cite{GK98}, one alternative description is
  \[ \dual{\UGC{m}{\operad C}}  \iso  \mathsf{F}_{\mathfrak T^{1 - m}}(\Sigma^{1 - m} \operad C)  \iso  \Sigma^{m - 1} \mathsf{F}_{\mathrm{Det}^{1 - m}}(\operad C) \]
  where $\mathsf F$ denotes the ``Feynman transform'', $\mathfrak T$ and $\mathrm{Det}$ are certain ``hyperoperads'', and $\Sigma$ denotes a degree shift.
  Other descriptions have been given for example by Lazarev--Voronov \cite[§3.1]{LV} (though with a different grading convention) and, in the case $m = 0$, by Conant--Vogtmann \cite[§2.3]{CV}.
  
  The homology of $\UGC{m}{\operad C}$ is in general very complicated and in most cases of interest only partial information is known.
  See, however, the following remarks for a couple of basic observations.
  Also see \cref{rem:lie_graph_homology} for an overview of what is known about this homology in the case of the cyclic Lie operad $\operad C = \Lie$.
\end{remark}

\begin{remark}
  Note that, up to a regrading, the differential graded coalgebra $\UGC{m}{\operad C}$, and hence its homology, only depend on the parity of $m$.
\end{remark}

\begin{remark} \label{rem:graph_complex_free}
  It follows from the description of the coalgebra structure of $\UGC {m} {\operad C}$ that the cohomology $\Coho * \big( \dual{\UGC {m} {\operad C}} \big)$ of its linear dual is a free graded commutative algebra generated by the cohomology of the dual of the subcomplex $\UGC[\mathrm{conn}] {m} {\operad C} \subseteq \UGC {m} {\operad C}$ spanned by the connected graphs.
  This subcomplex splits as a direct sum
  \[ \UGC[\mathrm{conn}] {m} {\operad C}  \iso  \Dirsum_{g \in \NN} \UGC[\mathrm{conn}; g] {m} {\operad C} \]
  where $\UGC[\mathrm{conn}; g] {m} {\operad C}$ denotes the subcomplex spanned by connected graphs of genus $g$ (here the \emph{genus} of a graph is $1$ plus the number of edges minus the number of vertices).
  Hence, understanding the homology of $\UGC {m} {\operad C}$ (or the cohomology of its dual) is equivalent to understanding the homology of $\UGC[\mathrm{conn}; g] {m} {\operad C}$ for all $g$ (or the cohomologies of their duals).
\end{remark}

\begin{remark} \label{rem:graph_complex_concentrated}
  Let $G \in \Graph[\ge 3]$ be a connected graph with $N$ vertices, $E$ edges, and genus $g = 1 - N + E$.
  Since every vertex is at least trivalent, we have $E \ge \frac{3}{2} N$, which implies $E \le 3 (g - 1)$.
  Since $G$ is non-empty, we have $N \ge 1$, which implies $g \le E$.
  (In particular there are no such graphs with $g \le 1$.)
  Now assume that $\operad C$ is concentrated in degree $0$.
  Then an element of $\UGC {m} {\operad C}$ represented by $G$ has homological degree $p = (1 - m) N + m E = E + (m - 1) (g - 1)$.
  This implies that, in this case, the subcomplex $\UGC[\mathrm{conn}; g] {m} {\operad C}$ of \cref{rem:graph_complex_free} is concentrated in those homological degrees $p$ such that $m (g - 1) + 1 \le p \le (m + 2) (g - 1)$.
\end{remark}

\begin{theorem} \label{lemma:UGC_diff}
  Let $m \in \NN$ be a natural number and $\operad C$ a cyclic operad concentrated in non-negative even degrees such that $\cyc{\operad C}{2} = \gen \basefield \id$ with the trivial $\Symm{2}$-action.
  Then there is an isomorphism of graded coalgebras
  \[ \Ho * \big( \UGC m {\operad C} \big)  \iso  \Ho * \big( \DGCtrunc m {\operad C} \big) \]
  where $\UGC m {\operad C}$ is equipped with the differential graded coalgebra structure of \cref{def:UGC_diff}.
\end{theorem}

\begin{proof}
  We first show that there is an isomorphism $\UGC m {\operad C} \to \Ho 0 (C^2_{\bullet,*})$ of chain complexes, where the domain is equipped with the differential of \cref{def:UGC_diff} and the target with the differential induced by $d^1$.
  Together with \eqref{eq:ss_collapse}, this yields the desired isomorphism on the level of the underlying graded vector spaces.
  
  The isomorphism $\Phi \colon \UGC m {\operad C} \to \Ho 0 (C^2_{\bullet,*})$ of \eqref{eq:gc_identification} maps the element $\eqcl x \in \UGC m {\operad C}$ represented by $G = (F, N, \mu, a)$ and decorations $(\xi_v)_{v \in N}$ to the element $\eqcl y$ of $\Ho 0 (C^2_{\bullet,*})$ represented by the directed graph $D_+(G) \defeq \Direct(G, \const[+])$ and the decorations $\xi_v$.
  By \cref{def:dgc}, we have $d^1(\eqcl y) = \sum_{s \in S} d^{D_+(G)}_s(y)$ where $S$ is the set of initial flags of $D_+(G)$.
  There is a canonical bijection $\sigma^G \colon \Edge(G) \to S$ (which is order preserving).
  We now define, for $e \in \Edge(G)$,
  \[ d^G_e(x)  \defeq  \inv \Phi \Big( \eqcl[\big] { d^{D_+(G)}_{\sigma^G(e)}(\widetilde \Phi^G(x)) } \Big) \]
  where $\widetilde \Phi^G \colon \UDec{m}{\operad C}(G)  \to  \Dec{m}{\operad C}(D_+(G))$ is the isomorphism from \cref{lemma:UDec_iso}.
  We will now prove that these elements $d^G_e(x)$ fulfill the properties of \cref{def:UGC_diff}.
  This implies that the map $\Phi$ is an isomorphism of chain complexes as desired.
  
  First note that, if $e$ is a loop, then $d^{D_+(G)}_{\sigma^G(e)}(\widetilde \Phi(x)) = 0$ and hence $d^G_e(x) = 0$.
  Hence the second property of \cref{def:UGC_diff} is fulfilled.
  
  Now we show that $d^G_e(x) = d^{G'}_{\chi(e)}(\chi \act x)$ for any isomorphism $\chi \colon G \to G'$ and edge $e = \set {f_1 < f_2} \in \Edge(G)$.
  This is clear when $e$ is a loop.
  Hence it is enough to consider the case where $G$ fulfills that $a$ is order preserving, that $v_1 \defeq a(f_1)$ and $v_2 \defeq a(f_2)$ are (in this order) the first two elements of $N$, and that $f_1$ is the last element of $\inv a(v_1)$ and $f_2$ the first of $\inv a(v_2)$.
  (This is enough since any isomorphism of undirected graphs with a distinguished non-loop edge factors through a graph of this form.)
  We need to show that
  \begin{equation} \label{eq:undirected_diff_equiv}
    \eqcl[\big] { d^{D_+(G)}_{\sigma^G(e)} (\widetilde \Phi^G(x)) }  =  \eqcl[\big] { d^{D_+(G')}_{\sigma^{G'}(\chi(e))} (\widetilde \Phi^{G'}(\chi \act x)) }
  \end{equation}
  as elements of $\Ho 0 (C^2_{\bullet,*})$.
  We first note that the left hand side is equal to $(-1)^{m (\outdeg{v_1} - 1)}$ times the equivalence class of the element $y$ represented by $\contract{\sigma^G(e)}(D_+(G))$ and the decorations $\xi_{v_1} \opcomp[f_1]{f_2} \xi_{v_2}$ at the collapsed vertex $\pr(v_1) = \pr(v_2)$ and by $\xi_w$ at all other vertices $w$.
  Now, we set $D' \defeq \Direct(G, \inv \chi \act \const[+])$ and write $\psi \colon D' \to \Direct(G', \const[+]) = D_+(G')$ for the isomorphism induced by $\chi$.
  We further denote by $S'$ the linearly ordered set of initial flags of $D'$ and set $s' \defeq \inv \psi (\sigma^{G'}(\chi(e))) \in S'$.
  Then the right hand side of \eqref{eq:undirected_diff_equiv} is equal to $\eqcl { d^{D'}_{s'} (\inv \psi \act \widetilde \Phi^{G'}(\chi \act x)) }$.
  It follows from the definitions that $\inv \psi \act \widetilde \Phi^{G'}(\chi \act x)$ is $\sgn(\omega)^m (-1)^{k (1 + m)}$ times the element $x'$ of $\Dec{m}{\operad C}(D')$ represented by the decorations $(\xi_w)_{w \in N}$, where $k$ is the number of edges whose orientation is flipped by $\chi$, and $\omega \colon S \to S'$ is the canonical bijection.
  
  There are now two cases: either $\chi$ preserves the orientation of $e$, i.e.\ we have $\chi(f_1) < \chi(f_2)$, or it reverses it.
  In both cases, we have that $d^{D'}_{s'} (x')$ is some sign $\epsilon \in \set {\pm 1}$ times the element $y'$ represented by $\contract{s'}(D')$ and the decorations $\xi_{v_1} \opcomp[f_1]{f_2} \xi_{v_2}$ (in the first case) or $\xi_{v_2} \opcomp[f_2]{f_1} \xi_{v_1}$ (in the second case) at the collapsed vertex $\pr(v_1) = \pr(v_2)$ and $\xi_w$ at all other vertices $w$.
  In the first case, we have $\epsilon = (-1)^{m (\outdeg{v_1'} - 1)}$, where $v_1' \defeq v_1$ considered as a vertex of $D'$.
  In the second case, we have $\epsilon = (-1)^{1 + m + m \outdeg{v_1'}}$ (to see this, one transforms $D'$ into the form necessary to read off $d^{D'}_{s'}$, and then identifies the result of the contraction with $\contract{s'}(D')$).
  It follows from the proof of \cref{lemma:UDec_iso} that $\eqcl {y'} = \delta \eqcl {y} \in \Ho 0 (C^2_{\bullet,*})$, where $\delta = \sgn(\hat \omega)^m (-1)^{k(1 + m)}$ in the first case, and $\delta = \sgn (\hat \omega)^m (-1)^{(k - 1)(1 + m)}$ in the second.
  Here $\hat \omega$ is the canonical bijection $S \setminus \set {s} \to S' \setminus \set {s'}$, and $k$ is the number of edges flipped by $\chi$, as above.
  Lastly, we note that commutativity of
  \[
  \begin{tikzcd}
    \set {s} \cop (S \setminus \set {s}) \rar{\iso} \dar[swap]{\hat \omega} & S \dar{\omega} \\
    \set {s'} \cop (S' \setminus \set {s'}) \rar{\iso} & S'
  \end{tikzcd}
  \]
  implies that the product $\sgn(\omega) \sgn(\hat \omega)$ equals $(-1)^{\outdeg{v_1} - 1 + \outdeg{v_1'} - 1}$ in the first case, and $(-1)^{\outdeg{v_1} - 1 + \outdeg{v_1'}}$ in the second.
  This finishes the proof of \eqref{eq:undirected_diff_equiv}.
  
  We now prove that the elements $d^G_e(x)$ fulfill the third property of \cref{def:UGC_diff}.
  We assume that $e = \set {f_1 < f_2}$ is not a loop, that $v_1 \defeq a(f_1)$ and $v_2 \defeq a(f_2)$ are (in this order) the first two elements of $N$, that $a$ is order-preserving, and that $f_1$ is the first element of $F$.
  We also assume that $f_2$ is the first element of $\inv{a}(v_2)$ (we can arrange this without introducing a sign).
  Now, we let $D'$ be equal to $D_+(G)$ except that the linear order of $F$ has been changed such that $f_1$ is the last element of $\inv{a}(v_1)$ instead of the first.
  Moreover, we denote by $\chi \colon D_+(G) \to D'$ the canonical isomorphism.
  We have $d^{D_+(G)}_{f_1}(\widetilde \Phi(x)) = d^{D'}_{f_1}(\chi \act \widetilde \Phi(x))$ and that $\chi \act \widetilde \Phi(x)$ is $(-1)^{m (\outdeg {v_1} - 1)}$ times the element represented by $D'$ and the decorations $(\xi_w)_{w \in N}$.
  By definition we thus have that $d^{D'}_{f_1}(\chi \act \widetilde \Phi(x))$ is the element represented by $\contract{f_1}(D')$ and the decorations $(\restrict \sigma {\inv a(v_1)} \act \xi_{v_1}) \opcomp[f_1]{f_2} \xi_{v_2}$ at the collapsed vertex and $\xi_w$ at all other vertices $w$.
  Here $\sigma$ is the canonical bijection from $F$ to the linearly ordered set $F'$ of flags of $D'$.
  Lastly, we note that $\contract{f_1}(D') = \Direct(\contract{e}(G), \const[+])$ and that $(\restrict \sigma {\inv a(v_1)} \act \xi_{v_1}) \opcomp[f_1]{f_2} \xi_{v_2} = \omega \act (\xi_{v_1} \opcomp[f_1]{f_2} \xi_{v_2})$, where $\omega$ is as in \cref{def:UGC_diff}.
  Thus the third property is fulfilled as well.
  
  It remains to prove that the isomorphism is compatible with the graded coalgebra structures.
  By \cref{lemma:double_complex_ss_collapse}, the isomorphism \eqref{eq:ss_collapse} is induced by the surjection of chain complexes
  \[ \Psi \colon \DGCtrunc m {\operad C}  \longto  \big( \Ho 0 (C^2_{\bullet,*}), d^1 \big) \]
  that is given by the canonical projections $\DGCtrunc m {\operad C}_{k,0} \to \Ho 0 (C^2_{\bullet,*})$ and the trivial maps on $\DGCtrunc m {\operad C}_{k,l}$ for $l > 0$.
  Hence it is enough to show that the surjection
  \[ \Theta \defeq \inv {\Phi} \after \Psi \colon \DGCtrunc m {\operad C}  \longto  \UGC m {\operad C} \]
  is a map of graded coalgebras.
  
  It is clear that $\Theta$ is compatible with the counits.
  Now note that the comultiplication $\Delta$ of $\DGCtrunc m {\operad C}$ is compatible with its bigrading.
  In particular $(\Psi \tensor \Psi)(\Delta(x))$ is trivial for an element $\DGCtrunc m {\operad C}_{k,l}$ with $l > 0$.
  Hence it is enough to check that $\Theta$ is compatible with the comultiplication of $\DGCtrunc m {\operad C}_{*,0}$.
  Let $D = (F, N, S, T, \mu, a)$ be a truncated directed graph without vertices of valence two, and let $x \in \DGCtrunc m {\operad C}_{*,0}$ be the element represented by $D$ and some decorations $(\xi_v)_{v \in N}$.
  Let $o_D \in \dir(\Underlying(D))$ be defined by setting $o_D(\set{f_1 < f_2})$ to be $+$ if $f_1 \in S$ and $-$ otherwise.
  Then $\Direct(\Underlying(D), o_D)$ is equal to $D$.
  By the proof of \cref{lemma:UDec_iso}, this implies that $\Theta(x)$ is $\alpha(D) \defeq \sgn(\sigma_D)^m (-1)^{k_D(1+m)}$ times the element of $\UGC m {\operad C}$ represented by $\Underlying(D)$ and the decorations $\xi_v$.
  Here $k_D \defeq \card{\inv{o_D}(-)}$ and $\sigma_D \colon S \to \Edge(\Underlying(D))$ is the canonical bijection.
  
  Recall that
  \[ \Delta(x)  \defeq  \sum_{N'} (-1)^{m \card {N'} \card{N \setminus N'}} \epsilon(N') x_{N'} \tensor x_{N \setminus N'} \]
  where the sum runs over all neighbor-closed vertex subsets $N'$ of $D$ and $x_A$ is the element of $\DGC m {\operad C}$ represented by $\induce[D](A)$ and $(\xi_v)_{v \in A}$.
  The sign $\epsilon(N')$ is specified in \cref{def:dgc}.
  Applying $\Theta \tensor \Theta$ to the right hand side, we obtain
  \[ \sum_{N'} (-1)^{m \card {N'} \card{N \setminus N'}} \epsilon(N') y_{N'} \tensor y_{N \setminus N'} \]
  where $y_A$ is $\alpha(\induce[D](A))$ times the element of $\UGC m {\operad C}$ represented by $\Underlying(\induce[D](A)) = \induce[\Underlying(D)](A)$ and $(\xi_v)_{v \in A}$.
  Note that a subset $N' \subseteq N$ is a neighbor-closed vertex subset of $D$ if and only if it is a neighbor-closed vertex subset of $\Underlying(D)$.
  Lastly, we observe that
  \[ \alpha(D) \alpha \big( {\induce[D](N')} \big) \alpha \big( {\induce[D](N \setminus N')} \big) \epsilon(N') = \delta(N') \]
  where $\delta(N')$ is the sign of \cref{def:UGC_diff}.
  This finishes the proof.
\end{proof}

\begin{remark}
  A proof of \cref{lemma:UGC_diff} for $\DGC m {\operad C}$ instead of $\DGCtrunc m {\operad C}$ has been sketched by Willwacher \cite[Appendix~K]{Wil} (in the case of the commutative cyclic operad $\operad C = \Com$).
  Other related statements have been proven in detail by Živković \cite{Ziv} and Dolgushev--Rogers \cite{DR}.
\end{remark}

\begin{remark} \label{rem:lie_graph_homology}
  We quickly survey what is known about the homology of the undirected graph complex $\UGC {m} \Lie$ associated to the cyclic Lie operad.
  By work of Kontsevich \cite{Kon93} (see also Conant--Vogtmann \cite[Theorem~2 and §3.1]{CV}) and Lazarev--Voronov \cite[Corollary~3.13]{LV}, there are, for $g \ge 2$, isomorphisms
  \[ \Coho p \big( \dual{\UGC[\mathrm{conn}; g] {m} \Lie} \big)  \iso  \Ho {(m + 2)(g - 1) - p} \big( \Out(\freegrp g); \widetilde \QQ^{\tensor m} \big) \]
  where $\UGC[\mathrm{conn}; g] {m} \Lie$ is as in \cref{rem:graph_complex_free}.
  (Note that it follows from \cref{rem:graph_complex_concentrated} that $\UGC[\mathrm{conn}; g] {m} \Lie$ is trivial for $g \le 1$.)
  Here $\Out(\freegrp g)$ is the group of outer automorphisms of the free group on $g$ generators, and $\widetilde \QQ$ is the one-dimensional $\Out(\freegrp g)$ representation given by
  \[ \Out(\freegrp g)  \xlongto{\mathrm{ab}}  \Out(\mathrm F^{\mathrm{ab}}_g)  \iso  \GL[g](\ZZ)  \xlongto{\det}  \GL[1](\ZZ)  \longto  \GL[1](\QQ) \]
  (which is sometimes called the ``determinant representation'').
  
  The homology $\Ho * (\Out(\freegrp g); \QQ)$, occurring when $m$ is even, has been studied intensively.
  A recent state of the art is summarized in \cite[§1]{CHKV}, including certain stability results and computer calculations for small $g$.
  The only non-trivial homology (outside of degree $0$) with $g \le 4$ is $\Ho 4 ( \Out(\freegrp 4); \QQ ) \iso \QQ$.
  More recently, Borinsky--Vogtmann \cite{BV} obtained a formula and asymptotics for the Euler characteristic of $\Ho * (\Out(\freegrp g); \QQ)$.
  
  The homology $\Ho * (\Out(\freegrp g); \widetilde \QQ)$, occurring when $m$ is odd, is less well studied.
  The only structural results known to the author are Euler characteristic computations and asymptotics, again from \cite{BV}, as well as that
  \[ \Ho{0} \big( \Out(\freegrp g); \widetilde \QQ \big)  \iso  \coinv {\widetilde \QQ} {\Out(\freegrp g)}  \iso  0    \qquad\text{and}\qquad    \Ho{*} \big( \Out(\freegrp 2); \widetilde \QQ \big)  \iso  0 \]
  the latter of which follows from the fact that $\Out(\freegrp 2) \iso \GL[2](\ZZ)$ and the Eichler--Shimura isomorphism, see Conant--Hatcher--Kassabov--Vogtmann \cite[§3.5]{CHKV}.
  Apart from this, there are computer calculations for small $g$ by Brun--Willwacher \cite[Figure~9]{BW}.
  Their work shows that the only non-trivial homology with $g \le 4$ is $\Ho 3 (\Out(\freegrp 4); \widetilde \QQ) \iso \QQ$.
  Using \cref{rem:graph_complex_concentrated}, this implies that $\Ho p (\UGC m \Lie)$ is trivial for $0 < p \le \min(3m + 2, 4m)$ when $m$ is odd.
\end{remark}

\section{Connected sums of products of spheres} \label{sec:surfaces}

Throughout this section we let $3 \le k < l \le 2k - 2$ be two integers.
We set
\[ M^{k, l}_{g, r}  \defeq  \Connsum_g (\Sphere{k} \times \Sphere{l}) \setminus \textstyle\coprod\limits_r \openDisk{k + l} \]
i.e.\ a $g$-fold connected sum of $\Sphere{k} \times \Sphere{l}$ with $r$ disjoint $(k + l)$-dimensional open disks removed (we will often omit $k$ and $l$ from the notation and simply write $M_{g,r}$).
We are interested in the stable cohomology of
\[ \B \autbdry (M_{g,1}) \]
i.e.\ of the classifying space of the group-like topological monoid of homotopy equivalences from $M_{g,1}$ to itself that fix the boundary pointwise.
To make sense of this, we need to specify stabilization maps $\B \autbdry (M_{g,1}) \to \B \autbdry (M_{g+1,1})$.
To this end, we note that there is a canonical inclusion
\[ \iota_g \colon M_{g,1} = M_{g,0} \setminus \openDisk{k + l}  \longincl  M_{g,0} \connsum M_{1,1}  =  M_{g+1,1} \]
(where $\connsum$ denotes connected sum) which induces a map of topological monoids
\[ \phi_g  \colon  \autbdry(M_{g,1})  \longto  \autbdry(M_{g+1,1}) \]
by extending an element of $\autbdry(M_{g,1})$ by the identity.

We will now introduce some related objects, which we will need for stating and proving our main result.
We start by fixing some basepoints of $\Sphere k$ and $\Sphere l$.
Then, for $1 \le i \le g$, we denote by $\alpha^g_i \in \Ho{k}(M_{g,1}; \ZZ)$ the homology class represented by the canonical inclusion of $\Sphere k$ into the $i$-th summand, and by $\dualbasis {(\alpha^g_i)} \in \Ho{l}(M_{g,1}; \ZZ)$ the class represented by the canonical inclusion of $\Sphere l$ into the $i$-th summand (we assume that these inclusions do not intersect the disks used to form the connected sum).
These classes form bases of $\Ho{k}(M_{g,1}; \ZZ) \iso \ZZ^g$ and $\Ho{l}(M_{g,1}; \ZZ) \iso \ZZ^g$, respectively.
We choose the orientations such that the intersection pairing fulfills $\intpair{\alpha^g_i}{\dualbasis {(\alpha^g_i)}} = 1$.
Also note that $\intpair{\alpha^g_i}{\dualbasis {(\alpha^g_j)}} = 0$ when $i \neq j$.
For $1 \le i \le g$, we moreover have $(\iota_g)_*(\alpha^g_i) = \alpha^{g+1}_i$ and $(\iota_g)_* \big( \dualbasis {(\alpha^g_i)} \big) = \dualbasis {(\alpha^{g+1}_i)}$.

Using this, we can define a stabilization map of groups
\[ \psi_g \colon  {\Aut} \big( \rHo{*}(M_{g,1}; \ZZ) \big)  \longto  {\Aut} \big( \rHo{*}(M_{g+1,1}; \ZZ) \big) \]
by letting $\psi_g(f)$ be the unique automorphism such that $\psi_g(f) \after (\iota_g)_* = (\iota_g)_* \after f$ and such that $\psi_g(f)$ is the identity on $\alpha^{g+1}_{g+1}$ and $\dualbasis{(\alpha^{g+1}_{g+1})}$.
(Here $\Aut$ denotes automorphisms of graded abelian groups.)
We now set
\[ \Gamma_g  \defeq  {\im} \big( {\htpygrp 0 (\autbdry(M_{g,1}))}  \to  \Aut(\rHo{*}(M_{g,1}; \ZZ)) \big) \]
i.e.\ it is the group of those automorphisms of the reduced integral homology of $M_{g,1}$ that can be realized by a homotopy automorphism that fixes the boundary pointwise.
Together, the maps $\phi_g$ and $\psi_g$ induce a stabilization map of groups $\gamma_g \colon \Gamma_g \to \Gamma_{g+1}$.
The group $\Gamma_g$ has been identified explicitly by Grey \cite{Gre}.

\begin{proposition}[Grey] \label{prop:Grey}
  We have
  \[ \Gamma_g = {\Aut} \big( \rHo{*}(M_{g,1}; \ZZ), \intpair{\blank}{\blank} \big)  \subseteq  {\Aut} \big( \rHo{*}(M_{g,1}; \ZZ) \big) \]
  where $\intpair{\blank}{\blank}$ is the intersection pairing.
  Moreover, the canonical map
  \[ {\htpygrp 0} \big( \autbdry(M_{g,1}) \big)  \longto  \Gamma_g \]
  has finite kernel.
\end{proposition}

\begin{proof}
  This is \cite[Proposition~5.1]{Gre}.
  (The extra structure $Jq$ mentioned there is necessarily trivial:
  When $k+l$ is odd, it is trivial by definition.
  When $k+l$ is even, it is trivial since $\Ho {\nicefrac{(k + l)}{2}} (M_{g,1}; \ZZ) \iso 0$ as $k \neq l$.)
\end{proof}

This implies, by the same arguments as in the proofs of \Cref{lemma:identify_with_dual,lemma:aut_is_gl}, that
\[ \Gamma_g  \iso  {\Aut} \big( \Ho{k}(M_{g,1}; \ZZ) \big) \]
acts on $\Ho{k}(M_{g,1}; \ZZ)$ via the standard action and on $\Ho{l}(M_{g,1}; \ZZ) \iso \dual{\Ho{k}(M_{g,1}; \ZZ)}$ via its dual.
The bases $\alpha^g_i$ of $\Ho{k}(M_{g,1}; \ZZ)$ constructed above yield isomorphisms $\Gamma_g \iso \GL[g](\ZZ)$ compatible with the $\gamma_g$ and the standard inclusions $\GL[g](\ZZ) \to \GL[g+1](\ZZ)$.
Similarly we obtain isomorphisms
\[ \Gamma^\QQ_g \defeq {\Aut} \big( \rHo{*}(M_{g,1}; \QQ), \intpair{\blank}{\blank} \big)  \iso  \GL[g](\QQ) \]
compatible with the stabilization maps.

Lastly we set
\[ \lie g_g  \defeq  \trunc{{\Der[\omega]} \big( \freelie(\shift[-1] \rHo{*}(M_{g,1}; \QQ)) \big)} \]
i.e.\ it is the positive truncation of the graded Lie algebra of derivations that annihilate the element $\omega \defeq \sum_i \liebr{\beta_i}{\dualbasis{\beta_i}}$ where $\beta_1, \dots, \beta_g$ is some basis of $\shift[-1] \Ho{k}(M_{g,1}; \QQ)$ and $\dualbasis {\beta_1}, \dots, \dualbasis {\beta_g}$ is the basis of $\shift[-1] \Ho{l}(M_{g,1}; \QQ)$ dual to it with respect to the intersection pairing; see \cite[§3.4~f.]{BM} for more details.
We note that the canonical action of $\Gamma^\QQ_g$ on $\shift[-1] \rHo{*}(M_{g,1}; \QQ)$ preserves the element $\omega$.
Hence we obtain an action of $\Gamma^\QQ_g$ on $\lie g_g$ by conjugation.
Restricting along the canonical inclusion $\Gamma_g \to \Gamma^\QQ_g$, we also obtain an action of $\Gamma_g$ on $\lie g_g$.

We now construct a stabilization map $\delta_g \colon \lie g_g \to \lie g_{g+1}$.
For $\theta \in \lie g_g$, the derivation $\delta_g(\theta)$ is uniquely determined by requiring it to restrict to $\theta$ along the map induced by $\iota_g$ and asking that $\delta_g(\theta) \big( \shift[-1] \alpha^{g+1}_{g+1} \big) = 0$ and $\delta_g(\theta) \big( \shift[-1] \dualbasis{(\alpha^{g+1}_{g+1})} \big) = 0$.
Here $\alpha^{g+1}_i$ is the basis of $\Ho{k}(M_{g+1,1}; \ZZ) \subseteq \Ho{k}(M_{g+1,1}; \QQ)$ from above.
It is clear from the definition that $\delta_g(\theta)(\omega) = 0$, so that $\delta_g$ indeed defines a map $\lie g_g \to \lie g_{g+1}$ of graded Lie algebras.
Moreover, this becomes a map of $\Gamma^\QQ_g$-modules when the target is equipped with the $\Gamma^\QQ_g$-action obtained by restricting along $\gamma_g$.

\begin{remark} \label{rem:X_g_cover}
  The graded Lie algebra $\lie g_g$ is a Lie model for the homotopy fiber of the map $\B \autbdry (M_{g,1})  \longto  \B \Gamma_g$ and one can think of the action of $\Gamma_g$ on $\lie g_g$ as modeling the action of $\Gamma_g \iso \htpygrp{1}(\B \Gamma_g)$ on this fiber; see \cite[Proposition~5.6]{BM} and \cite{Sto} for precise statements (for non-relative self-equivalences this is \cite[Theorem~3.40]{BZ}).
\end{remark}

Berglund--Zeman \cite{BZ} recently developed methods for identifying the cohomology of a classifying space of self-equivalences $\B \aut(X)$ in terms of certain group cohomology.
The details of applying this to the relative situation will appear in the upcoming paper \cite{Sto} joint with Berglund.
In particular the following statement is proved there; one can think of it as a strengthening of collapse of the Serre spectral sequence associated to the fiber sequence of \cref{rem:X_g_cover}.

\begin{proposition}[Berglund--Zeman, Berglund--Stoll] \label{prop:BZ}
  There are, for any $g \in \NN$ and $\Gamma_g$-module $P$, isomorphisms of graded vector spaces
  \[ \Theta_g(P)  \colon  {\Coho *} \big( \B \autbdry (M_{g,1}); P \big)  \xlongto{\iso}  {\Coho *} \big( \Gamma_g; \CEcoho * (\lie g_g) \tensor P \big) \]
  that are compatible with the stabilization maps induced by $\phi_{g-1}$, $\gamma_{g-1}$, and $\delta_{g-1}$, in the sense that the diagram
  \[
  \begin{tikzcd}[column sep = 70]
    {\Coho *} \big( \B \autbdry (M_{g,1}); P \big) \rar{\Theta_g(P)} \ar{dd}[swap]{\phi_{g-1}^*} & {\Coho *} \big( \Gamma_g; \CEcoho * (\lie g_g) \tensor P \big) \dar{\gamma_{g-1}^*} \\
     & {\Coho *} \big( \Gamma_{g-1}; \gamma_{g-1}^* (\CEcoho * (\lie g_g) \tensor P) \big) \dar{\delta_{g-1}^*} \\
    {\Coho *} \big( \B \autbdry (M_{g - 1,1}); \gamma_{g-1}^*(P) \big) \rar{\Theta_{g-1}(\gamma_{g-1}^*(P))} & {\Coho *} \big( \Gamma_{g - 1}; \CEcoho * (\lie g_{g - 1}) \tensor \gamma_{g-1}^*(P) \big)
  \end{tikzcd}
  \]
  commutes when $g \ge 1$.
  The map $\Theta_g(P)$ is an isomorphism of graded algebras when $P = \QQ$ is trivial, and an isomorphism of modules over these algebras for general $P$.
\end{proposition}

\begin{proof}
  For $P = \QQ$ and without the compatibility with the stabilization maps this is stated (without a detailed proof) as \cite[Theorem~4.40]{BZ}.
  A full proof will appear in the upcoming paper \cite{Sto} joint with Berglund.
\end{proof}

The rest of this section is devoted to providing, for certain $\Gamma_g$-modules $P$, an explicit description of $\Coho{*}(\Gamma_g; \CEcoho{*}(\lie g_g) \tensor P)$ in a stable range of degrees.

\subsection{Reduction to invariants}

In this subsection our goal is to prove that there is an isomorphism of bigraded algebras
\[ \Coho{p} \big( \Gamma_g; \CEcoho{q}(\lie g_g) \tensor P \big)  \iso  \Coho{p}(\Gamma_g; \QQ) \tensor \Coho{q} \Big( \big( \CEcochains{*}(\lie g_g) \tensor P \big)^{\Gamma^\QQ_g} \Big) \]
in a stable range of degrees.
This will mainly rely on a result of Li--Sun \cite{LS19}.
We need some preliminaries.

\begin{definition}
  A \emph{Schur bifunctor} is a functor $F \colon \Vect \times \Vect \to \Vect$ of the form
  \[ F(V_1, V_2)  =  \Dirsum_{n_1, n_2 \in \NN} \bijmod M(n_1, n_2) \tensor[\Symm{n_1} \times \Symm{n_2}] (V_1^{\tensor n_1} \tensor V_2^{\tensor n_2}) \]
  where the $\bijmod M(n_1, n_2)$ are some right $(\Symm{n_1} \times \Symm{n_2})$-modules in $\Vect$.
  We say $F$ has \emph{degree~$\le d$} if $\bijmod M(n_1, n_2) \iso 0$ for all pairs $(n_1, n_2)$ such that $n_1 + n_2 > d$.
  We say $F$ is of \emph{finite type} if each $\bijmod M(n_1, n_2)$ is finite dimensional.
\end{definition}

\begin{lemma} \label{lemma:sl_vs_gl_inv}
  Let $F$ be a finite type Schur bifunctor of degree $\le d$.
  Then the canonical inclusion
  \[ F(V, \dual{V})^{\GL(V)}  \longto  F(V, \dual{V})^{\SL(V)} \]
  is an isomorphism for all finite-dimensional vector spaces $V$ such that $\dim V > d$.
  Here $f \in \GL(V)$ acts on $\dual{V}$ by $\dual{(\inv f)}$ and diagonally on $F(V, \dual{V})$.
\end{lemma}

\begin{proof}
  Using, for $G$ a group and $W$ a $G$-module, the natural (in both $W$ and $G$) isomorphism $\dual{(\coinv W G)} \iso (\dual W)^G$, we can pass to the dual situation, i.e.\ take coinvariants instead of invariants.
  (To see that $\dual{F(V, \dual{V})}$ is again a Schur bifunctor, we use that $\Sigma \defeq \Symm{n_1} \times \Symm{n_2}$ is finite and that we are working in characteristic zero, so that $\dual{(M \tensor[\Sigma] N)} \iso \dual M \tensor[\Sigma] \dual N$.)
  
  Both the $\GL(V)$-action and the $\SL(V)$-action are compatible with the direct sum composition of $F(V, \dual{V})$, so that it is enough to prove for a single summand
  \[ W  \defeq  \bijmod M(n_1, n_2) \tensor[\Symm{n_1} \times \Symm{n_2}] \big( V^{\tensor n_1} \tensor (\dual V)^{\tensor n_2} \big) \]
  with $n_1 + n_2 < \dim V$ that $\coinv W {\SL(V)} \to \coinv W {\GL(V)}$ is an isomorphism.
  Denote by $e_1, \dots, e_g$ some basis of $V$.
  Then the equivalence relation defining $\coinv W {\GL(V)}$ is generated by
  \begin{equation} \label{eq:GL_coinv_rel}
    m \tensor \Tensor_{j = 1}^{n_1} e_{i_j}  \tensor \Tensor_{j' = 1}^{n_2} \dualbasis e_{i'_{j'}}  \sim  m \tensor \Tensor_{j = 1}^{n_1} f(e_{i_j})  \tensor \Tensor_{j' = 1}^{n_2} \dual{(\inv f)} (\dualbasis e_{i'_{j'}})
  \end{equation}
  for all $f \in \GL(V)$, $m \in \bijmod M(n_1, n_2)$, and $1 \le i_j, i'_{j'} \le g$.
  Choose some $1 \le i \le g$ such that $i \neq i_j, i'_{j'}$ for all $j$ and $j'$ (this is possible by our assumption that $\dim V > n_1 + n_2$).
  Now let $f' \defeq f \after h \in \GL(V)$ where $h$ is defined by $h(e_i) = \inv{(\det f)} e_i$ and $h(e_j) = e_j$ for all $j \neq i$.
  Then $f' \in \SL(V)$ and it yields the same relation \eqref{eq:GL_coinv_rel} as $f$ did.
  Hence $\coinv W {\SL(V)} \to \coinv W {\GL(V)}$ is injective.
  Since it is clearly surjective, it is thus an isomorphism.
\end{proof}

Now we are ready to prove the main result of this subsection.
See \cite[Theorem~8.4]{BM} for the analogous statement in the case $k = l$; the stable range occurring in their result has been improved by Krannich \cite{Kra} using an argument similar to the one we will employ.

\begin{proposition} \label{prop:red_to_invariants}
  Let $Q$ be a finite type Schur bifunctor of degree $\le r$ and set $P \defeq Q \big( \Ho k (M_{g,1}; \QQ), \Ho l (M_{g,1}; \QQ) \big)$ as a $\Gamma^\QQ_g$-module.
  Then the canonical map of bigraded vector spaces
  \begin{multline*}
    \Coho{p}(\Gamma_g; \QQ) \tensor \Coho{q} \Big( \big( \CEcochains{*}(\lie g_g) \tensor P \big)^{\Gamma^\QQ_g} \Big)  \iso  \Coho{p} \Big( \Gamma_g; \Coho{q} \Big( \big( \CEcochains{*}(\lie g_g) \tensor P \big)^{\Gamma^\QQ_g} \Big) \Big) \longto \\
    \Coho{p} \Big( \Gamma_g; \Coho{q} \big( \CEcochains{*}(\lie g_g) \tensor P \big) \Big)  \iso  \Coho{p} \big( \Gamma_g; \CEcoho{q}(\lie g_g) \tensor P \big)
  \end{multline*}
  is an isomorphism for $g > {\max} \big( p + 1, \frac{1}{2} (q + 1) + r \big)$.
\end{proposition}

\begin{proof}
  A result of Li--Sun \cite[Example~1.10]{LS19} implies that the canonical map
  \[ \Coho{p} \Big( \Gamma_g; \Coho{q} \big( \CEcochains{*}(\lie g_g) \tensor P \big)^{\Gamma_g} \Big)  \longto  \Coho{p} \Big( \Gamma_g; \Coho{q} \big( \CEcochains{*}(\lie g_g) \tensor P \big) \Big) \]
  is an isomorphism for $g \ge p + 2$.
  This uses that \cite[(Proof of) Proposition 7.1]{Gre} implies that $\CEcochains{q}(\lie g_g)$, and hence also $\Coho{q} \big( \CEcochains{*}(\lie g_g) \tensor P \big)$, is a finite dimensional rational representation of the reductive group $\GL[g]$ and thus completely reducible (see e.g.\ \cite[Theorem~22.42]{Mil}).
  
  Now, by \cite[Lemma~3.4]{Gre}, the group $\Gamma_g$ is rationally perfect (cf.\ \cite[Definiton~B.3]{BM}).
  Hence, taking $\Gamma_g$-invariants is exact when restricted to the category of finite dimensional representations.
  In particular the canonical map
  \[ \Coho{q} \Big( \big( \CEcochains{*}(\lie g_g) \tensor P \big)^{\Gamma_g} \Big)  \longto  \Coho{q} \big( \CEcochains{*}(\lie g_g) \tensor P \big)^{\Gamma_g} \]
  is an isomorphism.
  
  Lastly we use that, by a result of Borel \cite[Theorem~1]{Bor66}, the arithmetic subgroup $\SL[g](\ZZ)$ is Zariski dense in $\SL[g](\QQ)$ (note that the condition of the theorem is fulfilled since $\SL[g]$ is almost simple for $g \ge 2$).
  Hence we have $V^{\SL[g](\ZZ)} = V^{\SL[g](\QQ)}$ for any rational representation $V$ of $\SL[g]$ over $\QQ$.
  In particular, for $F$ a finite type Schur bifunctor of degree $\le d$ and $V \defeq F(\QQ^g, \dual{(\QQ^g)})$, we thus have, when $g > d$, that
  \[ V^{\GL[g](\ZZ)}  =  \big( V^{\SL[g](\ZZ)} \big)^{\units \ZZ}  =  \big( V^{\SL[g](\QQ)} \big)^{\units \ZZ}  =  \big( V^{\GL[g](\QQ)} \big)^{\units \ZZ}  =  V^{\GL[g](\QQ)} \]
  where the third equality uses \Cref{lemma:sl_vs_gl_inv}.
  Now we note that, by \cite[(Proof of) Proposition 7.1]{Gre}, there exists a finite type Schur bifunctor $F_q$ of degree $\le \frac q 2$ such that there is an isomorphism
  \[ \CEcochains{q}(\lie g_g)  \iso  F_q \big( \Ho{k}(M_{g,1}; \QQ), \Ho{l}(M_{g,1}; \QQ) \big) \]
  of $\Gamma^\QQ_g$-modules (where the action on the right hand side is diagonally).
  Hence $\CEcochains{q}(\lie g_g) \tensor P$ is, in the same way, the value of a Schur bifunctor of degree $\le \frac q 2 + r$.
  This finishes the proof.
  For the condition on $g$ given in the statement, note that an isomorphism of chain complexes in a range induces an isomorphism of homology in a range smaller by $1$.
\end{proof}

\subsection{The main theorem}

We are now ready to deduce our main theorem from everything we have done so far.

\begin{theorem} \label{thm:main_coeff}
  Let $3 \le k < l \le 2k - 2$ and $2 \le g$ be integers, and let $I$ and $J$ be finite linearly ordered sets.
  Furthermore, set $H_{g;i} \defeq \Ho i (M^{k,l}_{g,1}; \QQ)$ as a ${\htpygrp{0}} \big( \autbdry (M^{k,l}_{g,1}) \big)$-module.
  Then there is, in cohomological degrees
  \[ \le \begin{cases} g - 2, & \text{when $\card I + \card J \le \frac 1 2 g$} \\ 2 (g - \card I - \card J - 1), & \text{otherwise} \end{cases} \]
  an isomorphism of graded $(\Symm I \times \Symm J)$-modules
  \begin{multline*}
    \Coho * \big( \B \autbdry (M^{k,l}_{g,1}); H_{g;l}^{\tensor I} \tensor H_{g;k}^{\tensor J} \big)  \iso \\
    \shift[(k-1) \card I + (l-1) \card J] \Coho * \big( \GL(\ZZ); \QQ \big) \tensor \Coho * \big( \dual{(\DGCtrunc {k + l - 2} {\Lie}_{I,J})} \big) \tensor \sgn_I^{\tensor k-1} \tensor \sgn_J^{\tensor l-1}
  \end{multline*}
  compatible with the stabilization maps on the left hand side (here the shift refers to the cohomological grading).
  It is an isomorphism of algebras when $I = J = \emptyset$, and an isomorphism of modules over these algebras for general $I$ and $J$.
  Here $\GL(\ZZ) \defeq \colim{g \in \NNo} \GL[g](\ZZ)$ and $\sgn_S$ is the sign representation of $\Symm S$.
  Also recall that $\DGCtrunc {k + l - 2} {\Lie}_{I,J}$ is the truncated directed $(I,J)$-graph complex (see \cref{def:trunc_dgc}) associated to the cyclic Lie operad.
  
  In particular, after stabilizing, we obtain an isomorphism
  \begin{multline*}
    \lim {g \in \NNo} \Coho * \big( \B \autbdry (M^{k,l}_{g,1}); H_{g;l}^{\tensor I} \tensor H_{g;k}^{\tensor J} \big)  \iso \\
    \shift[(k-1) \card I + (l-1) \card J] \Coho * \big( \GL(\ZZ); \QQ \big) \tensor \Coho * \big( \dual{(\DGCtrunc {k + l - 2} {\Lie}_{I,J})} \big) \tensor \sgn_I^{\tensor k-1} \tensor \sgn_J^{\tensor l-1}
  \end{multline*}
  of graded $(\Symm I \times \Symm J)$-modules.
  It is an isomorphism of algebras when $I = J = \emptyset$, and an isomorphism of modules over these algebras for general $I$ and $J$.
\end{theorem}

\begin{proof}
  Combining \cref{prop:BZ,prop:red_to_invariants}, we obtain, in the stable range
  \[ g > \max_{p + q = s} {\max} \big( p + 1, \tfrac{1}{2} (q + 1) + \card I + \card J \big) = {\max} \big( s + 1, \tfrac{1}{2} (s + 1) + \card I + \card J \big) \]
  an isomorphism of graded $(\Symm I \times \Symm J)$-modules
  \[ \Coho{s} \big( \B \autbdry (M^{k,l}_{g,1}); H_{g;l}^{\tensor I} \tensor H_{g;k}^{\tensor J} \big)  \iso  \Dirsum_{p + q = s} \Coho{p}(\Gamma_g; \QQ) \tensor \Coho{q} \Big( \big( \CEcochains{*}(\lie g_g) \tensor H_{g;l}^{\tensor I} \tensor H_{g;k}^{\tensor J} \big)^{\Gamma^\QQ_g} \Big) \]
  which is compatible with the stabilization maps and the algebra/module structures.
  It follows from work of Borel \cite{Bor74} (see also \cite[Theorem~A]{KMP}\footnote{They actually prove a slightly better range than Borel, which we will not need.}) that the canonical map $\Coho{p}(\GL(\ZZ); \QQ) \to \Coho{p}(\GL[g](\ZZ); \QQ)$ is an isomorphism for $g \ge p + 2$; in particular we can replace $\Coho{p}(\Gamma_g; \QQ) \iso \Coho{p}(\GL[g](\ZZ); \QQ)$ with $\Coho{p}(\GL(\ZZ); \QQ)$ in the formula above.
  
  Now, by \cite[Proposition~6.6]{BM}, there is an isomorphism of Lie algebras
  \[ \lie g_g  =  \trunc{{\Der[\omega]} \big( \freelie(\shift[-1] \rHo{*}(M_{g,1}; \QQ)) \big)}  \iso  \trunc {\Schur[\big] {(\shift[-(k + l - 2)] \operad \Lie)} {\shift[-1] \rHo{*}(M_{g,1}; \QQ)}} \]
  that is compatible with the $\Gamma^\QQ_g$-actions and the stabilization maps.
  Here the right hand side is equipped with the graded Lie algebra structure of \cref{lemma:Schur_cyclic_Lie}, where $\shift[-1] \rHo{*}(M_{g,1}; \QQ)$ is equipped with the anti-symmetric pairing given by $\iprod {\shift[-1] a} {\shift[-1] b} = (-1)^{\deg a + 1} \intpair a b$.
  Then \cref{lemma:trunc_dgc} implies that, in degrees $q \le \frac{4}{3} g - 1$ if $l = 2k - 2$ or $q \le 2g - 1$ if $l \neq 2k - 2$, there is an isomorphism of graded $(\Symm I \times \Symm J)$-modules
  \[ \Ho q \Big( {\coinv {\big( \CEchains * (\lie g_g) \tensor (\shift[k-1] H_{g;k})^{\tensor I} \tensor (\shift[l-1] H_{g;l})^{\tensor J} \big)} {\Gamma^\QQ_g}} \Big)  \iso  \Ho q \big( \DGCtrunc {k + l - 2} {\Lie}_{I,J} \big) \]
  compatible with the stabilization maps on the left hand side as well as the coalgebra/comodule structures.
  (Note that an isomorphism of chain complexes up to degree $p$ induces an isomorphism on homology up to degree $p - 1$.)
  Dualizing this yields the desired statement since there is an isomorphism
  \[ (\shift[k-1] H_{g;k})^{\tensor I} \tensor (\shift[l-1] H_{g;l})^{\tensor J}  \iso  \shift[(k-1) \card I + (l-1) \card J] H_{g;k}^{\tensor I} \tensor H_{g;l}^{\tensor J} \tensor \sgn_I^{\tensor k-1} \tensor \sgn_J^{\tensor l-1} \]
  of $(\Gamma^\QQ_g \times \Symm I \times \Symm J)$-modules and $H_{g;k}$ and $H_{g;l}$ are dual to each other.
\end{proof}

\begin{remark}
  Via Schur--Weyl duality (see e.g.\ \cite[§6.1]{FH} or \cite[§5.19]{Eti}), \cref{thm:main_coeff} can be used to deduce descriptions of $\Coho * \big( \B \autbdry (M^{k,l}_{g,1}); P \big)$ for more general $\Gamma^\QQ_g$-modules $P$.
\end{remark}

\begin{corollary} \label{thm:main}
  Let $3 \le k < l \le 2k - 2$ and $2 \le g$ be integers.
  Then there is, in cohomological degrees $\le g - 2$, an isomorphism of graded algebras
  \[ \Coho * \big( \B \autbdry (M^{k,l}_{g,1}); \QQ \big)  \iso  \Coho * \big( \GL(\ZZ); \QQ \big) \tensor \Coho * \big( \dual{\UGC {k + l - 2} \Lie} \big) \]
  compatible with the stabilization maps on the left hand side.
  (Recall that $\UGC {k + l - 2} \Lie$ denotes the undirected graph complex associated to the cyclic Lie operad, see \cref{def:UGC_diff}.)
  
  In particular, after stabilizing, we obtain an isomorphism
  \[ \lim {g \in \NNo} \Coho * \big( \B \autbdry (M^{k,l}_{g,1}); \QQ \big)  \iso  \Coho * \big( \GL(\ZZ); \QQ \big) \tensor \Coho * \big( \dual{\UGC {k + l - 2} \Lie} \big) \]
  of graded algebras.
\end{corollary}

\begin{proof}
  This follows from \cref{thm:main_coeff} with $I = J = \emptyset$ and \cref{lemma:UGC_diff}.
\end{proof}

\begin{remark}
  It should be possible to obtain a version of this corollary for cohomology with non-trivial coefficients by generalizing the arguments of \cref{sec:undirected_gc} (see \cref{rem:ugc_coeff}).
\end{remark}

\begin{remark}
  Note that
  \[ \Coho{*} \big( \GL(\ZZ); \QQ \big)  \iso  \freegca \gen \QQ {x_i \mid i \in \NNpos}  \qquad\text{where}\quad \deg{x_i} = 4i + 1 \]
  by a classical result of Borel \cite[11.4]{Bor74}.
  The homology of $\UGC {k + l - 2} \Lie$ is more mysterious.
  See \cref{rem:lie_graph_homology} for a summary of what is known.
\end{remark}

\printbibliography

\contact{Department of Mathematics, Stockholm University, 106 91 Stockholm, Sweden}{robin.stoll@math.su.se}

\end{document}